\documentclass{amsart}
\usepackage{amsmath}
\usepackage{amssymb}
\usepackage{amsfonts}
\usepackage{mathtools}
\usepackage{graphicx}
\usepackage{subfigure}
\usepackage{float}
\usepackage{bbm}
\usepackage{amsthm}
\usepackage{hyperref}
\usepackage{cleveref}
\usepackage{tikz}
\usepackage{tikz-3dplot}
\usepackage{mathrsfs}
\usepackage{bm}
\usepackage{esint}
\usepackage{appendix}

\numberwithin{table}{section}
\numberwithin{equation}{section}

\newtheorem{thm}[equation]{Theorem}
\newtheorem{crl}[equation]{Corollary}  
\newtheorem{lem}[equation]{Lemma} 
\newtheorem{prop}[equation]{Proposition}

\newtheorem{question}[equation]{Question}  
\newtheorem{openquestion}[equation]{Open question}

\theoremstyle{definition}
\newtheorem{defi}[equation]{Definition}

\theoremstyle{remark}
\newtheorem{rmk}[equation]{Remark}

\title{Compactification of metric moduli space of K3 surfaces}

\author{Zexuan Ouyang}
\address{BICMR, Peking University, Beijing, 100871, China}
\email{oyzx@pku.edu.cn}
\author{Gang Tian}
\address{BICMR and School of Mathematics, Peking University, Beijing, 100871, China}
\email{gtian@math.pku.edu.cn}

\date{}

\begin{document}
	
	\bibliographystyle{plain}

	\maketitle
	
	\begin{abstract}
		We prove a conjecture of Odaka--Oshima \cite[Conj.~III]{odaka-oshima-21}, which says that there is an algebraic description of the Gromov--Hausdorff compactification of all unit-diameter hyperkähler metrics on K3 surfaces.
		As a corollary, we obtain a classification of the Gromov--Hausdorff limits of those hyperkähler K3 surfaces with a fixed complex structure or with a fixed polarization.
	\end{abstract}
	
	\setcounter{tocdepth}{1}
	\tableofcontents
	
	\section{Introduction}
	As the first non-flat example of Calabi–Yau manifolds, K3 surfaces and their moduli spaces have long attracted considerable attention.
	A long-standing problem is to construct compactifications of the moduli spaces where points on the boundary have geometric meaning.
	In this paper, we investigate this moduli compactification problem from the metric viewpoint.
	
	Let $\mathfrak{M}$ denote the set of isometry classes of unit-diameter hyperkähler metrics on K3 surfaces. By Gromov's precompactness theorem, its Gromov--Hausdorff (GH) closure $\overline{\mathfrak{M}}$ is compact. Previous works have classified all elements of $\overline{\mathfrak{M}}$ (\cite[Thm.~1.1]{sun-zhang-24}, \cite[Thm.~3]{ouyang-25}), and there are extensive works on specific examples of collapsing of hyperkähler K3 surfaces \cite{chen-viaclovsky-zhang-19,foscolo-19,gross-tosatti-zhang-16,gross-wilson-00,hein-sun-Viaclovsky-Zhang-22,odaka-oshima-21,song-tian-zhang-19-collapsing,tosatti-weninkove-yang-18}.  
	However, it is not known whether there is an algebraic moduli space which characterizes the degenerations in $\overline{\mathfrak{M}}$. It is also natural to study the relation between $\overline{\mathfrak{M}}$ and various compactifications of K3 moduli spaces.
	
	Let $\mathfrak{M} \cup \mathfrak{M}_4$ be the partial compactification obtained by adding those hyperkähler K3 surfaces with orbifold singularities. It is known that these are all the non-collapsing limits of hyperkähler K3 surfaces \cite[Thm.~IV]{anderson-92}\cite[Thm.~7]{kobayashi-todorov-87} (cf. \S\ref{section-metric-moduli}), and that there is a continuous surjection 
	\begin{equation}
		\label{definition-phi}
		\Phi: \mathcal{M}_{\mathrm{K3}} \coloneq 
		O(\Lambda_{\mathrm{K3}})\backslash SO_0(3,19)/\bigl(SO(3) \times SO(19)\bigr)\to \mathfrak{M} \cup \mathfrak{M}^4,
	\end{equation}
	where $O(\Lambda_{\mathrm{K3}})$ is a discrete subgroup of $O(3,19)$ preserving the K3 lattice $\Lambda_{\mathrm{K3}}$.
	Thus, $\mathcal{M}_{\mathrm{K3}}$ can be viewed as the moduli space of unit-diameter hyperkähler metrics on K3 surfaces (allowing orbifold singularities), and $\mathcal{M}_{\mathrm{K3}}$ is an arithmetic quotient of a symmetric space of real dimension~57.
	
	Let $\mathcal{M}_{2d}$ denote the polarized moduli space of K3 surfaces of degree $2d$, for $d \in \mathbb{Z}^+$. Then $\mathcal{M}_{2d}$ is also an arithmetic quotient of a symmetric space, of real dimension~38, and there is a map from $\mathcal{M}_{2d}$ to $\mathcal{M}_{\mathrm{K3}}$ which is generally $2$-to-$1$.
	
	Let $\mathcal{M}$ be an arithmetic quotient of a symmetric space of the form $\Gamma \backslash G/K$ (both $\mathcal{M}_{2d}$ and $\mathcal{M}_{\mathrm{K3}}$ are of this form), where $G$ is a semisimple linear algebraic group defined over $\mathbb{Q}$ 
	with maximal compact subgroup $K$, and $\Gamma$ is an arithmetic subgroup of 
	$G(\mathbb{Q})$. There are several different compactifications of $\mathcal{M}$. Satake first constructed compactifications of $\mathcal{M}$ by adding certain rational boundary components \cite{satake-60-2,satake-60-1}. The compactification depends on a choice of faithful irreducible representation $\rho$ of $G$, and we denote it by $\overline{\mathcal{M}}^{\rho}$. There are only finitely many different Satake compactifications $\overline{\mathcal{M}}^{\rho_i}$ arising in this way. Among them, there is a \emph{maximal Satake compactification} $\overline{\mathcal{M}}^{\mathrm{max}}$ dominating the others; that is, there is a continuous map $\overline{\mathcal{M}}^{\mathrm{max}} \to \overline{\mathcal{M}}^{\rho_i}$ which extends the identity map on $\mathcal{M}$. There is also a \emph{Borel--Serre compactification} $\overline{\mathcal{M}}^{BS}$ \cite{borel-serre-73}, which is a real analytic manifold with corners and dominates the maximal Satake compactification.
	
	For the polarized moduli space $\mathcal{M}_{2d}$, there is an additional Hermitian structure on it, and $\mathcal{M}_{2d}$ also admits a Baily--Borel compactification $\overline{\mathcal{M}}_{2d}^{BB}$ \cite{baily-borel-66}, together with infinitely many toroidal compactifications.  
	Topologically, $\overline{\mathcal{M}}_{2d}^{BB}$ coincides with one of the minimal Satake compactifications. Baily and Borel proved that it carries the structure of a normal projective variety. We also refer to \cite{borel-ji-06} for a comprehensive survey of different compactifications of arithmetic quotients of symmetric spaces.
	
	\subsection{Statement of results}
	
	Our main result is the following, which is proved in \S\ref{section-proof-of-main-thm}. 
	\begin{thm}\label{thm-main-theorem}
		Let $\overline{\mathcal{M}}_{\mathrm{K3}}^{\mathrm{ad}}$ be the Satake compactification with respect to the adjoint representation. Then there is a continuous geometric realization map
		\[
		\overline{\Phi}: \overline{\mathcal{M}}_{\mathrm{K3}}^{\mathrm{ad}} \to \overline{\mathfrak{M}},
		\]
		which extends the map $\Phi$ on $\mathcal{M}_{\mathrm{K3}}$ defined in \eqref{definition-phi}.
	\end{thm}
	
	This confirms a conjecture of Odaka--Oshima \cite[Conj.~III]{odaka-oshima-21} and gives an algebraic description of the GH compactification $\overline{\mathfrak{M}}$.
	The expression of $\overline{\Phi}$ is explicitly written (see \S \ref{section-construction-geometric-realization} for the definition).    
	Thus by Theorem \ref{thm-main-theorem}, if we know the period behaviour of a sequence of hyperkähler K3 surfaces, we can determine its GH limit. Many structures are easy to see in the period domain. For example, we can deal with hyperkähler K3 surfaces with fixed polarization or fixed complex structures.	
	
	For the polarized moduli space, there is also a Satake compactification $\overline{\mathcal{M}}_{2d}^{\mathrm{ad}}$ with respect to the adjoint representation, and we have a map (cf. \S\ref{section-polarized-moduli-space})
	$$
	f_{2d}: \overline{\mathcal{M}}_{2d}^{\mathrm{ad}} \to \overline{\mathcal{M}}_{\mathrm{K3}}^{\mathrm{ad}}.
	$$
	\begin{crl}\label{tt-crl-12072200}
		There is a continuous geometric realization map for the Satake compactification of the polarized moduli space of K3 surfaces,
		\[
		\overline{\Phi}_{2d} = \overline{\Phi} \circ f_{2d} : \overline{\mathcal{M}}_{2d}^{\mathrm{ad}} \to \overline{\mathfrak{M}}.
		\]
	\end{crl}
	
	Let $\mathfrak{M}_{2d}$ denote the space of all isometry classes of unit-diameter hyperkähler metrics on K3 surfaces with a fixed polarization class $\lambda_{2d}$ in the K3 lattice  $\Lambda_{\mathrm{K3}}$, where $\lambda_{2d}\cdot \lambda_{2d}=2d$.  
	As a result of Corollary~\ref{tt-crl-12072200}, we know that the GH closure $\overline{\mathfrak{M}_{2d}}$ is given by the image of $\overline{\Phi}_{2d}$.  
	Once we know the image of $f_{2d}$ in $\overline{\mathcal{M}}_{\mathrm{K3}}^{\mathrm{ad}}$, we can give a classification of $\overline{\mathfrak{M}_{2d}}$.  
	The following is proved in \S\ref{section-polarized-moduli-space}, based on Theorem~\ref{thm-main-theorem}:
	
	\begin{crl}\label{crl-fix-polarization}
		The GH closure $\overline{\mathfrak{M}_{2d}}$ is obtained from $\mathfrak{M}_{2d}$ by adding the following:
		\begin{itemize}
			\item All generalized KE metrics (cf. \S\ref{section-generalized-KE}) on $\mathbb{P}^1$ arising from elliptic K3 surfaces $\pi : X \to \mathbb{P}^1$ with the property that we can choose a holomorphic volume form $\Omega$ on $X$ such that $[\operatorname{Re}\Omega]$ is identified with $\lambda_{2d}$ in $H^2(X,\mathbb{Z}) \cong \Lambda_{\mathrm{K3}}$;
			\item The unit segment $I^1\coloneq([0,1], dt^2)$.
		\end{itemize}
	\end{crl}
	
	Moreover, from Corollary~\ref{tt-crl-12072200} we also obtain GH limits of type~II or type~III Kulikov degenerations.  
	For type~II degenerations, the GH limit is the unit segment (\cite[Rem.~VII]{odaka-oshima-21}).  
	For type~III degenerations, the GH limit is a generalized KE metric on $\mathbb{P}^1$ (\cite[Cor.~VI]{odaka-oshima-21}).
	
	Similarly, one may consider collapsing limits while keeping the complex structure fixed. 
	The following is proved in \S\ref{section-collapsing-with-fixed-complex-structure}, based on an ergodic result for K3 surfaces due to Verbitsky \cite{verbitsky-15-acta} and Theorem~\ref{thm-main-theorem}.
	
	\begin{crl}\label{crl-fix-complex-structure}
		Let $\mathfrak{M}_J$ be the moduli space of isometry classes of unit-diameter hyperkähler metrics 
		on K3 surfaces with fixed complex structure~$J$.  
		Let $\Omega$ be a holomorphic volume form with respect to $J$.
		Define
		$
		W \coloneqq \langle \operatorname{Re}\Omega, \operatorname{Im}\Omega \rangle ,
		$
		then:
		\begin{enumerate}
			\item[(i)] If $\dim\big(W \cap H^2(X,\mathbb{Q})\big)=0$, then $\overline{\mathfrak{M}_{J}} = \overline{\mathfrak{M}}$.  
			In particular, all possible GH limits can be realized.
			
			\item[(ii)] If $\dim\big(W \cap H^2(X,\mathbb{Q})\big)=1$, then $\overline{\mathfrak{M}_{J}} = \overline{\mathfrak{M}_{2d}}$ and is classified in Corollary~\ref{crl-fix-polarization}, where $2d$ is the degree of the primitive element of $W \cap H^2(X,\mathbb{Z})$.
			
			\item[(iii)] If $\dim\big(W \cap H^2(X,\mathbb{Q})\big)=2$, equivalently the Picard number $\rho(X)=20$, then  
			$
			\partial \overline{\mathfrak{M}_{J}} \coloneqq 
			\overline{\mathfrak{M}_{J}} \setminus \mathfrak{M}_{J}
			$
			consists precisely of all generalized Kähler--Einstein metrics on $\mathbb{P}^1$ arising from elliptic fibrations 
			$\pi : (X,J) \to \mathbb{P}^1$, and there are only finitely many possible GH limits.
		\end{enumerate}
	\end{crl}
	
	\subsection{About the proof}

	The main difficulty in proving Theorem~\ref{thm-main-theorem} is how to relate the intrinsically defined GH compactification $\overline{\mathfrak{M}}$ with an algebraically defined Satake compactification. 
	Previously, there are only very special cases in which continuity is verified \cite[Thm.~6.9]{odaka-oshima-21}. However, the proofs in these cases depend on certain good properties of collapsing along elliptic fibers. The proof does not apply to general cases. 
	
	Our proof in fact proceeds in the reverse direction and does not rely on the partial results in \cite{odaka-oshima-21}. We aim to extract the period point from the data of a collapsing metric.
	See \cite[\S 7.2]{sun-zhang-24} for a discussion of this approach, and see \cite{liu} for related work.
	Since the GH compactification $\overline{\mathfrak{M}}$ has already been classified, Theorem~\ref{thm-main-theorem} is equivalent to the following (cf. \S\ref{section-proof-of-main-thm}):
	
	\begin{question}\label{question-main-question}
		Suppose $\eta_k \in \mathcal{M}_{\mathrm{K3}}$ with
		$\eta_k \to \eta_\infty \in \overline{\mathcal{M}}_{\mathrm{K3}}^{\mathrm{Sat,ad}}$
		and
		$\Phi(\eta_k) \to (X_\infty,g_\infty)$ in the GH sense.
		Determine all possible $\eta_\infty$ from $(X_\infty,g_\infty)$.
	\end{question}
	
	Let $d = \dim X_\infty$.  
	We will answer Question~\ref{question-main-question} in the cases $d=2$ and $d=3$ (Theorems~\ref{thm-2dimensional-collapsing-period-behaviour} and~\ref{thm-3dimensional-collapsing-period-behaviour}).  
	This is enough to prove Theorem~\ref{thm-main-theorem} and we then derive the $d=1$ case as a corollary.  
	
	In the case $d=2$, we show that for sufficiently large $k$, the K3 surface $X_k$ admits an elliptic fibration 
	$
	\pi_k : X_k \to \mathbb{P}^1,
	$
	and the generalized Kähler–Einstein metric on $\mathbb{P}^1$ induced by $\pi_k$ converges to $(X_\infty, g_\infty)$ (Theorem~\ref{thm-2dimensional-collapsing-elliptic-approximation}).  
	Denote by $\mathcal{M}_{\mathrm{Jac}}$ the moduli space of Jacobian elliptic K3 surfaces, and let $\xi_k \in \mathcal{M}_{\mathrm{Jac}}$ be the point corresponding to the Jacobian of $\pi_k$.  
	If $\xi_k$ converges to $\xi_\infty \in \mathcal{M}_{\mathrm{Jac}}$, then Question~\ref{question-main-question} is answered by Theorem~\ref{thm-2dimensional-collapsing-elliptic-approximation} and Theorem~\ref{thm-geometric-realization-GIT}.  
	
	When $\xi_k$ diverges in $\mathcal{M}_{\mathrm{Jac}}$, the situation is analogous to the $d=3$ case.  
	In these cases, we construct three embedded $2$-tori in $X_k$ and compute the integrals of the hyperkähler triple over them.  
	This allows us to answer Question~\ref{question-main-question} in these cases as well.
	
	The structure of the paper is as follows.  
	In \S\ref{section-2} we review some elementary properties and previous works on K3 surfaces.  
	In \S\ref{section-satake} we review the construction of the Satake compactification of the K3 moduli space and give more details on its topology. The definition of $\overline{\Phi}$ is given in \S\ref{section-construction-geometric-realization}.  
	In \S\ref{section-collaping-K3} we study the collapsing behaviours of hyperkähler K3 surfaces and build elliptic fibrations when $d=2$.  
	Finally, in \S\ref{section-period-collapsing} we answer Question~\ref{question-main-question} for $d=2,3$, and prove our main theorem.
	
	{\bf Acknowledgements}: Z. Ouyang would like to thank Liwei Fan, Yuji Odaka for helpful discussions. Z. Ouyang is supported by the National Key R.D. Program of China No.~2023YFA1009900.
	
	G. Tian is supported in part by NSFC No.~11890660 \& No.~12341105, and MOST No.~2020YFA0712800.
	\section{Background on K3 surfaces}\label{section-2}
	
	A \emph{K3 surface} $X$ is a simply connected, compact complex surface with trivial canonical bundle. 
	We refer to \cite[Chapter VIII]{barth-cptcplxsurface-04} and \cite{huybrechts-lecturenotes} for basic properties of K3 surfaces.
	Topologically, all K3 surfaces are diffeomorphic. Its second cohomology group $H^2(X, \mathbb{Z})$, equipped with the cup product pairing, is isomorphic to the \emph{K3 lattice}
	\[
	\Lambda_{\mathrm{K3}} := U^{\oplus 3} \oplus (-E_8)^{\oplus 2},
	\]
	where $U = \begin{pmatrix} 0 & 1 \\ 1 & 0 \end{pmatrix}$, and $E_8$ denotes the unique even
	unimodular positive-definite lattice of rank $8$. The lattice $\Lambda_{\mathrm{K3}}$ is even, unimodular, and has signature $(3,19)$.
	
	\subsection{Metric Moduli of K3 Surfaces}\label{section-metric-moduli}
	
	\begin{defi}
		Let $(X,d_X)$ and $(Y,d_Y)$ be two compact metric spaces. 
		Their \emph{Gromov--Hausdorff (GH) distance} $d_{GH}(X,Y)$ is defined as the infimum of $\varepsilon > 0$ such that there exists a metric $d$ on the disjoint union $X \sqcup Y$ extending $d_X$ and $d_Y$, with $X \subset B_\varepsilon(Y)$ and $Y \subset B_\varepsilon(X)$. 
		
		We say a map (possibly discontinuous) $f: X \to Y$ is an \emph{$\varepsilon$-GH approximation} if:
		\begin{enumerate}
			\item[(i)] $\bigl|d_X(p,q) - d_Y(f(p),f(q))\bigr| < \varepsilon$ for all $p,q \in X$;
			\item[(ii)] $Y \subset B_\varepsilon(f(X))$, i.e., $f(X)$ is $\varepsilon$-dense in $Y$.
		\end{enumerate}
	\end{defi}
	
	\begin{rmk}
		If $f: X \to Y$ is an $\varepsilon$-GH approximation, then $d_{GH}(X,Y) < 2\varepsilon$.
	\end{rmk}
	
	The moduli space of marked K3 surfaces with a Kähler form is an open dense set $K\Omega^0$ in the domain $K\Omega \coloneq O(3,19) / (SO(2) \times O(18))$. By Yau's solution of the Calabi conjecture \cite{yau-78}, $X$ admits a unique Ricci-flat Kähler metric in each Kähler class. Similarly, the boundary points in $K\Omega \setminus K\Omega^0$ are represented by Ricci-flat orbifold metrics \cite{kobayashi-todorov-87}. 
	
	By quotienting out the hyperkähler rotation, we obtain $(K\Omega/\sim )= O(3,19) / (SO(3) \times O(19)) \cong\mathrm{Gr}_3^{+,or}(\Lambda_{\mathrm{K3}, \mathbb{R}}) $, 
	where $\mathrm{Gr}^{+,or}_3(\Lambda_{\mathrm{K3},\mathbb{R}})$ denotes the Grassmannian of oriented positive-definite $3$-planes in $\Lambda_{\mathrm{K3}, \mathbb{R}}$. We define the period domain $\mathcal{M}_{\mathrm{K3}}$ of unit-volume hyperkähler metrics by further forgetting the markings:
	\begin{align*}
		\mathcal{M}_{\mathrm{K3}}	&\coloneq O(\Lambda_{\mathrm{K3}}) \backslash \mathrm{Gr}_3^{+,or}(\Lambda_{\mathrm{K3}, \mathbb{R}}) = O(\Lambda_{\mathrm{K3}}) \backslash \mathrm{Gr}_3^{+}(\Lambda_{\mathrm{K3}, \mathbb{R}}) \\
		&= O(\Lambda_{\mathrm{K3}}) \backslash O(3,19) / (O(3) \times O(19)).
	\end{align*}
	
	%
	
	Let $\mathfrak{M}$ denote the space of isometry classes of smooth unit-diameter hyperkähler metrics on K3 surfaces, and let $\mathfrak{M}^4$ denote the space of unit-diameter hyperkähler metrics on K3 surfaces with ADE singularities. We then have a surjective map:
	\[
	\Phi: \mathcal{M}_{\mathrm{K3}} \to \mathfrak{M} \cup \mathfrak{M}^4,
	\]
	which is known to be continuous (see \cite{anderson-92},\cite[Proposition~6.7]{odaka-oshima-21}). 
	
	Furthermore, the map $\Phi$ would be a bijection if we were to keep track of the orientation. Under our current definition, $\Phi$ is injective on the smooth locus. However, the injectivity of $\Phi$ fails on the orbifold locus. A prime example of this failure occurs in the case of flat Kummer orbifolds $T^4 / \mathbb{Z}_2$.
	
	By \cite[Thm.~II, Thm.~IV]{anderson-92}, if a sequence $\{\eta_k\} \subset \mathcal{M}_{\mathrm{K3}}$ diverges, then the corresponding hyperkähler K3 surfaces $(X,g_k)$ must be volume-collapsing. In particular, $\mathfrak{M} \cup \mathfrak{M}^4$ exhausts all non-collapsing GH limits of hyperkähler metrics on K3 surfaces. (cf.\cite[\S 4]{tian-90-invention})
	
	\subsection{Classification of Collapsing Limits of Hyperkähler K3 Surfaces}
	
	Let $\overline{\mathfrak{M}}$ be the closure of $\mathfrak{M}$ in the space of isometry classes of compact metric spaces.  
	By Gromov’s precompactness theorem \cite[Cor.~11.1.13]{petersen-16}, the space $\overline{\mathfrak{M}}$ is compact.  
	A complete classification of GH limits is given by \cite[Thm.~1.1]{sun-zhang-24}, \cite[Thm.~3]{ouyang-25} as follows:
	
	\begin{thm}\label{thm-classification-GH-limit}
		The GH closure of $\mathfrak{M}$ is given by
		\[
		\overline{\mathfrak{M}} \;=\; \mathfrak{M} \;\cup\; \bigcup_{d=1}^4 \mathfrak{M}^d,
		\]
		where $\mathfrak{M}^d$ is the set of limit spaces of Hausdorff dimension $d$, and
		\begin{itemize}
			\item $\mathfrak{M}^3$ consists of all flat orbifolds $T^3/\{\pm 1\}$;
			\item $\mathfrak{M}^2$ consists of all generalized KE metrics on $\mathbb{P}^1$ given by an elliptic K3 surface;
			\item $\mathfrak{M}^1$ consists of a single element, the unit segment $I^1$.
		\end{itemize}
	\end{thm}
	
	\begin{rmk}\label{rmk-limit-measure-volume-comparison}
		Suppose $(X_j,g_j)\to (X_\infty,g_\infty)$ is a collapsing sequence of hyperkähler K3 surfaces and $v_j$ is the renormalized volume measure on $X_j$.  
		By passing to a subsequence, we can get a renormalized limit measure $v_\infty$ on $X_\infty$.  
		In the case $d=2,3$, $v_\infty$ is proportional to the Hausdorff measure (\cite[Thm~4.3, Thm~5.1]{sun-zhang-24}).  
		In the case $d=1$, $v_\infty$ is given by a concave piecewise affine function (\cite[Thm.~1.1(3)]{honda-sun-zhang-19}). 
	\end{rmk}

	\subsection{Elliptic K3 surfaces}
	
	An \emph{elliptic K3 surface} is a K3 surface $X$ equipped with a holomorphic surjection 
	$
	\pi: X \to \mathbb{P}^1,
	$
	whose generic fiber is a smooth elliptic curve. We call such a map $\pi$ an \emph{elliptic fibration}.  
	In fact, any elliptic curve on a K3 surface naturally extends to an elliptic fibration:
	
	\begin{prop}\label{prop-elliptic-curve-elliptic-K3}
		Let $X$ be a K3 surface and let $C \subset X$ be a smooth elliptic curve. Then:
		\begin{enumerate}
			\item[(i)] The class $[C]$ is primitive in $H_2(X,\mathbb{Z})$.
			\item[(ii)] There exists an elliptic fibration $\pi : X \to \mathbb{P}^1$ such that $C$ is a fiber of $\pi$.
			\item[(iii)] If $C' \subset X$ is an elliptic curve with $[C'] = [C] \in H_2(X,\mathbb{Z})$, then $C'$ is also a fiber of $\pi$.
		\end{enumerate}
	\end{prop}
	
	\begin{proof}
		The result is essentially contained in \cite{huybrechts-lecturenotes}.  
		The first assertion directly follows from \cite[Rem.~2.3.13]{huybrechts-lecturenotes}.  
		
		For (ii), by \cite[Lem.~2.2.1]{huybrechts-lecturenotes}, we have $(C,C) = 0$ and $\mathcal{O}(C)$ is nef.  
		Then, by \cite[Prop.~2.3.10]{huybrechts-lecturenotes}, there exists a positive integer $m$ and a divisor $E$ such that $\mathcal{O}(C) \cong \mathcal{O}(mE)$, where $\mathcal{O}(E)$ defines an elliptic fibration $\pi : X \to \mathbb{P}^1$ such that $E$ is a fiber.  
		Since $[C]$ is primitive by (i), we must have $m = 1$, and $C$ is a fiber of $\pi$.  
		
		For (iii), if $C'$ is an elliptic curve with $[C'] = [C]$, then $c_1(\mathcal{O}(C')) = c_1(\mathcal{O}(C))$.  
		Hence
		\[
		\mathcal{O}(C) \otimes \mathcal{O}(C')^{-1} \in \operatorname{Pic}^0(X) = \{0\},
		\]
		so $\mathcal{O}(C) \cong \mathcal{O}(C')$.  
		Therefore, $C'$ defines the same elliptic fibration $\pi : X \to \mathbb{P}^1$ and is a fiber of it.
	\end{proof}
	
	\subsubsection{Jacobian fibrations}\label{subsection-jacobian-fibration}
	
	Given an elliptic K3 surface $\pi : X \to \mathbb{P}^1$, there exists a unique elliptic K3 surface $\pi_J : J \to \mathbb{P}^1$ with a holomorphic section $\sigma_J$ and the same functional and homological invariants as $\pi$ (see, e.g., \cite[Chap.~V, Thm~11.1(b)]{barth-cptcplxsurface-04}).  
	Such a fibration $\pi_J : J \to \mathbb{P}^1$ is called the \emph{Jacobian fibration} associated with $\pi$.
	Moreover, there exists a smooth section $\sigma$ of $\pi : X \to \mathbb{P}^1$ and a diffeomorphism $\psi : X \to J$ mapping $\sigma(\mathbb{P}^1)$ to $\sigma_J(\mathbb{P}^1)$
	(see, e.g., \cite[Thm.~3.1, Cor.~3.2]{chen-viaclovsky-zhang-19}).  
	Furthermore, one can choose holomorphic volume forms $\Omega_X$ and $\Omega_J$ on $X$ and $J$, respectively, and a 2-form $\alpha$ on $\mathbb{P}^1$, such that
	\[
	\psi^* \Omega_J = \Omega_X - \pi^*\alpha.
	\]
	
	\begin{prop}\label{jacobian-characterize}
		The cohomology class of $\pi^*\alpha$ is given by 
		\[
		[\pi^*\alpha] = (E,\Omega_X)\, f,
		\]
		where $E,F \in H_2(X,\mathbb{Z})$ are the homology classes of the section $\sigma$ and a general fiber of $\pi$, respectively, and $f = F^\vee$ denotes the Poincaré dual of $F$.
	\end{prop}
	
	\begin{proof}
		Note that
		\[
		\int_{\sigma(\mathbb{P}^1)} (\Omega_X - \pi^*\alpha) 
		= \int_{\sigma_J(\mathbb{P}^1)} \Omega_J 
		= 0.
		\]
		Thus
		\[
		\int_{\mathbb{P}^1} \alpha 
		= \int_{\sigma(\mathbb{P}^1)} \pi^*\alpha
		= \int_{\sigma(\mathbb{P}^1)} \Omega_X 
		= (E,\Omega_X).
		\]
		
		Let $R \subset \mathbb{P}^1$ be the regular locus of $\pi$, i.e., $\mathbb{P}^1$ minus at most 24 points.  
		For any closed smooth 2-form $\beta$ on $X$, since $\pi$ is locally trivial over $R$, we have
		\[
		\int_X \beta \wedge \pi^*\alpha 
		= \int_{\pi^{-1}(R)} \beta \wedge \pi^*\alpha 
		= (F,\beta) \int_R \alpha
		= (F,\beta)(E,\Omega_X).
		\]
		Hence $[\pi^*\alpha] = (E,\Omega_X)\, f$ in cohomology.
	\end{proof}
	
	\begin{prop}\label{prop-e,f-give-elliptic}
		Suppose $e,f \in \Lambda_{\mathrm{K3}}$ have intersection form $U = \begin{pmatrix} 0 & 1 \\ 1 & 0 \end{pmatrix}$.  
		Given an oriented positive-definite 2-plane $W$ in $\langle f \rangle^\perp$, there exists an elliptic K3 surface $\pi : X \to \mathbb{P}^1$ determined by $W$ up to isomorphism, and a marking $\phi : H^2(X,\mathbb{Z}) \to \Lambda_{\mathrm{K3}}$, such that $W$ is spanned by $(\mathrm{Re}\,\Omega_X, \mathrm{Im}\,\Omega_X)$ and $\phi^{-1}(f)$ is the Poincaré dual of the fiber class.  
		
		If moreover $W \subset \langle e,f \rangle^\perp$, then the resulting elliptic K3 surface is a Jacobian elliptic K3 surface.
	\end{prop}
	
	\begin{proof}
		By the surjectivity of the period map, there exists a marked pair $(X,\phi)$ such that $W = \phi(\langle \mathrm{Re}\,\Omega_X, \mathrm{Im}\,\Omega_X \rangle)$.  
		By \cite[Rem.~2.3.13(iii)]{huybrechts-lecturenotes}, after applying a sequence of Picard–Lefschetz reflections to $\phi$, we may assume $\phi^{-1}(e)$ corresponds to a nef line bundle $L$ on $X$, giving an elliptic K3 surface by \cite[Prop.~2.3.10]{huybrechts-lecturenotes}.  
		Then $\phi^{-1}(e)$ is the Poincaré dual of the fiber class.
		
		If there is another marked K3 $(X',\phi')$ such that $\phi'^{-1}(f)$ is nef, then $f$ lies in the closure of the Kähler cones $\phi(\mathcal{K}_X)$ and $\phi'(\mathcal{K}_{X'})$.  
		By \cite[Prop.~8.5.5]{huybrechts-lecturenotes}, there exists a sequence of Weyl reflections $s_i$ fixing $f$, such that $s_1 \circ \cdots \circ s_k$ maps $\phi(\mathcal{K}_X)$ to $\phi'(\mathcal{K}_{X'})$.  
		By the Torelli theorem \cite[Chap.~VIII, Thm.~11.1]{barth-cptcplxsurface-04}, the composition $\phi'^{-1} \circ s_1 \circ \cdots \circ s_k \circ \phi$ is induced by a biholomorphism $\tau : X \to X'$ mapping $\phi^{-1}(f)$ to $\phi'^{-1}(f)$.  
		Hence the associated elliptic K3 surfaces are isomorphic.
		
		Finally, the elliptic surface corresponding to $x \in \mathcal{M}(a)$ is in fact a Jacobian elliptic fibration.  
		Since $(e,\Omega_X) = 0$, we have $e \in H^{1,1}(X,\mathbb{Z})$.  
		By \cite[Chap.~VIII, Prop.~3.7]{barth-cptcplxsurface-04}, either $e$ or $-e$ is effective on $X$.  
		Since $(e,f) = 1$, we know $e$ is effective.  
		Thus there exists some component $D_i$ of $e$ with $(f,D_i) = 1$, which represents a section of $\pi : X \to \mathbb{P}^1$.
	\end{proof}
	The converse of Proposition~\ref{prop-e,f-give-elliptic} is also true.
	For a Jacobian elliptic K3 surface, we can take $e = E^\vee + F^\vee$, $f = F^\vee$ as in Proposition~\ref{jacobian-characterize}.  
	In particular, every Jacobian elliptic K3 surface is algebraic, since $e + f$ is an ample class by the Nakai criterion.
	
	\subsubsection{Generalized KE metrics}\label{section-generalized-KE}
	
	Let $\pi : X \to C$ be a smooth elliptic fibration over a Riemann surface, and let $\Omega$ be a holomorphic volume form on $X$.  
	We define a Kähler form $\omega$ on $C$ by
	\begin{equation}\label{formula-defi-gen-KE}
		\int_V \omega \;=\; \int_{\pi^{-1}(V)} \Omega \wedge \overline{\Omega}, 
		\qquad V \subset C \;\text{open}.
	\end{equation}
	
	\begin{defi}
		The \emph{generalized KE metric} on $C$ associated with $\pi : X \to C$ and $\Omega$ is the metric induced by the Kähler form $\omega$ defined above.
	\end{defi}
	
	We can describe the generalized KE metric more explicitly.  
	Let $U \subset C$ be an open set over which there exists a holomorphic section.  
	By \cite[Prop.~7.2]{gross-97}, there exists a lattice $L$ generated by two holomorphic 1-forms $\tau_1, \tau_2$ on $U$ and a biholomorphic map
	\[
	h: \mathcal{T}^*U \, / \, \langle \tau_1, \tau_2 \rangle \;\xrightarrow{\sim}\; \pi^{-1}(U),
	\]
	such that $h^* \Omega$ is the canonical holomorphic 2-form on $\mathcal{T}^*U$.  
	Assuming $\mathrm{Im}(\bar{\tau}_1 \tau_2) > 0$, the generalized KE metric is locally given by the Kähler form
	\[
	\omega \;=\; -\frac{1}{2} \, \mathrm{Re}\!(\tau_1 \wedge \bar{\tau}_2).
	\]
	This definition is independent of the choice of basis $\tau_1, \tau_2$.  
	If we write $\tau_i = w_i \, dz$ in local coordinates, then
	\[
	\omega \;=\; \frac{i}{2}\, \mathrm{Im}(\bar{w}_1 w_2)\, dz \wedge d\bar{z}, 
	\qquad 
	g \;=\; \mathrm{Im}(\bar{w}_1 w_2) |dz|^2.
	\]
	A direct calculation shows that the curvature $K(g) \ge 0$.
	
	Let $\pi : X \to \mathbb{P}^1$ be an elliptic fibration of a K3 surface, and denote by $S$ the singular locus in $\mathbb{P}^1$.  
	Then we have a generalized KE metric on $\mathbb{P}^1 \setminus S$.  
	From the definition, it follows immediately that the associated Jacobian elliptic K3 surface $\pi_J : J \to \mathbb{P}^1$ induces the same generalized KE metric on $\mathbb{P}^1 \setminus S$ as $\pi : X \to \mathbb{P}^1$.
	
	The asymptotic behavior of $\tau_1$ and $\tau_2$ near the singular locus is classified (see, e.g., \cite[Table~1]{hein-12}) according to the type of Kodaira singular fibers \cite{kodaira-63}.  
	Consequently, the local behavior of the generalized KE metric near singular points is well understood (\cite[Table~3.2]{ouyang-25}).  
	In particular, all such metrics on $\mathbb{P}^1 \setminus S$ have finite diameter, so the metric completion identifies naturally with $\mathbb{P}^1$.
	
	\subsection{Metric realization of the moduli space of elliptic K3 surfaces}
	
	\subsubsection{GIT quotient construction}\label{subsection-git-quotient-construction}
	
	We briefly review the GIT quotient construction of $\overline{\mathcal{M}_W}$, which provides a compactification of the moduli space of Jacobian elliptic K3 surfaces.
	
	Recall that an elliptic surface in Weierstrass normal form is given by the following (see, e.g., \cite[Thm.~2.1]{miranda-81})
	\begin{equation}\label{weierstrass}
		X = \Bigl\{ ([x:y:z], t) \in 
		\mathbb{P}\bigl(\mathcal{O}_{\mathbb{P}^1}(4)\oplus \mathcal{O}_{\mathbb{P}^1}(6)\oplus \mathcal{O}_{\mathbb{P}^1}\bigr) : 
		y^2 z = x^3 + h_8(t) x z^2 + h_{12}(t) z^3 \Bigr\},
	\end{equation}
	where 
	\[
	(h_8,h_{12}) \in W \coloneqq H^0(\mathbb{P}^1, \mathcal{O}(8)) \oplus H^0(\mathbb{P}^1, \mathcal{O}(12)).
	\]
	Here $h_8$ and $h_{12}$ can be viewed as polynomials of degree at most 8 and 12 in $t$, respectively.
	There is a natural action of $\mathbb{C}^\ast \times SL_2(\mathbb{C})$ on $W \setminus \{0\}$, where $SL_2(\mathbb{C})$ acts by automorphisms of $\mathbb{P}^1$ and $\lambda \in \mathbb{C}^\ast$ acts by 
	$
	\lambda \cdot (h_8,h_{12}) = (\lambda^2 h_8, \lambda^3 h_{12}).
	$
	Points in the same orbit correspond to isomorphic elliptic surfaces.  
	
	Denote
	\[
	\mathbb{P}W \;\coloneqq\; W / \mathbb{C}^\ast \;\cong\; \mathbb{P}(2^9, 3^{13}),
	\]
	which carries a natural polarization via a Veronese embedding into some projective space $\mathbb{P}^N$.  
	The GIT quotient \cite{mumford-f-k-git-94} with respect to this polarization is
	\[
	\overline{\mathcal{M}_W} \;\coloneqq\; \bigl( \mathbb{P}(2^9, 3^{13}) \bigr) /\!\!/ \, SL_2(\mathbb{C}),
	\]
	which is a projective variety.  
	Denote by $X^s, X^{ss}, X^{ps}$ the stable, semistable, and polystable loci of $\mathbb{P}W$, respectively.  
	Then (see \cite[Prop.~5.1]{miranda-81})
	\begin{align*}
		X^s &= \bigl\{ [h_8:h_{12}] \in \mathbb{P}W \;\bigm|\; v_p(h_8) < 4 \text{ or } v_p(h_{12}) < 6 \text{ for all } p \in \mathbb{P}^1 \bigr\}, \\
		X^{ss} &= \bigl\{ [h_8:h_{12}] \in \mathbb{P}W \;\bigm|\; v_p(h_8) \le 4 \text{ or } v_p(h_{12}) \le 6 \text{ for all } p \in \mathbb{P}^1 \bigr\},
	\end{align*}
	and every element in $X^{ps} \setminus X^s$ is $SL_2(\mathbb{C})$-equivalent to $[a t^4 : b t^6]$ with $a,b \in \mathbb{C}$.  
	From GIT theory, we have a surjection
	\[
	\pi_W : X^{ss} \longrightarrow \overline{\mathcal{M}_W}.
	\]
	
	Given $(h_8,h_{12}) \in W$, define the discriminant
	\[
	\Delta(t) = 4 h_8(t)^3 + 27 h_{12}(t)^2 \;\in\; H^0(\mathbb{P}^1, \mathcal{O}(24)).
	\]
	Define $\mathcal{M}_{\mathrm{Jac}} \subset \overline{\mathcal{M}_W}$ to be the locus of equivalence classes $[h_8:h_{12}]$ satisfying:
	\begin{itemize}
		\item $\Delta(t) \not\equiv 0$;
		\item for every $q \in \mathbb{P}^1$, either $v_q(h_8) < 4$ or $v_q(h_{12}) < 6$.
	\end{itemize}
	By \cite[Thm.~2.1, Cor.~2.5]{miranda-81}, $\mathcal{M}_{\mathrm{Jac}}$ is the moduli space of Jacobian elliptic K3 surfaces.  
	The associated K3 surfaces are obtained by resolving the singularities of the surfaces defined by \eqref{weierstrass}, and $\Delta(t)=0$ if and only if the corresponding fiber is singular.  
	Thus $\overline{\mathcal{M}_W}$ provides a compactification of $\mathcal{M}_{\mathrm{Jac}}$, with $\pi^{-1}(\mathcal{M}_{\mathrm{Jac}})$ an open dense subset of $X^s$.
	
	\subsubsection{Construction of the metric realization map}
	
	We now define the metric realization map $\overline{\Phi}_W : \overline{\mathcal{M}_W} \to \overline{\mathfrak{M}}$.  
	We actually define
	$
	\widetilde{\Phi}_W \coloneqq \overline{\Phi}_W \circ \pi_W : X^{ss} \to \overline{\mathfrak{M}},
	$ 
	and divide the definition into $3$ cases:
	
	\begin{enumerate}
		\item[(i)] If $[h_8:h_{12}] \in X^s$ with $\Delta \not\equiv 0$, then the point lies in $\mathcal{M}_{\mathrm{Jac}}$ and defines a smooth elliptic K3 surface.  
		We set $\widetilde{\Phi}_W([h_8:h_{12}])$ to be the unit-diameter generalized KE metric associated with it.
		
		\item[(ii)] If $[h_8:h_{12}] \in X^s$ with $\Delta \equiv 0$, then $[h_8:h_{12}]$ is $SL_2(\mathbb{C})$-equivalent to $[3 G_4^2: G_4^3]$, where
		\[
		G_4 = t(t-1)(t-2)(t-c), \quad c \neq 0,1,2,\infty.
		\]
		We define $\widetilde{\Phi}_W([h_8:h_{12}])$ to be the unit-diameter singular metric sphere $(\mathbb{P}^1,g)$, where
		\[
		g = C \left| \frac{1}{t(t-1)(t-2)(t-c)} \right| |dt|^2.
		\]
		This metric space is isometric to a flat orbifold $T^2 / \{\pm 1\}$.
		
		\item[(iii)] If $[h_8:h_{12}] \in X^{ss} \setminus X^s$, we set $\widetilde{\Phi}_W([h_8:h_{12}]) = I^1$.
	\end{enumerate}
	
	It is straightforward to check that this also defines $\overline{\Phi}_W : \overline{\mathcal{M}_W} \to \overline{\mathfrak{M}}$.  
	The following theorem is proved in \cite[Thm.~7.14]{odaka-oshima-21}; we will provide a simpler proof in Appendix~\ref{section-appendix}.
	
	\begin{thm}\label{thm-geometric-realization-GIT}
		The map
		\[
		\overline{\Phi}_W : \overline{\mathcal{M}_W} \longrightarrow \overline{\mathfrak{M}}
		\]
		is continuous, where $\overline{\mathcal{M}_W}$ is endowed with the analytic topology.
	\end{thm}
	
	\section{Satake compactification}\label{section-satake}
	
	Let $G = O(3,19)$ be the automorphism group of $\Lambda_{\mathrm{K3}}$, let $K = O(3)\times O(19)$ be a maximal compact subgroup, let $G_0=SO_0(3,19)$, and let $K_0=SO(3)\times SO(19)$ be the connected components containing $\operatorname{Id}$,
	and let $G_{\mathbb{Z}}$ be the arithmetic subgroup of $G$ preserving the K3 lattice $\Lambda_{\mathrm{K3}}$. Then we have $\mathcal{M}_{\mathrm{K3}} \cong G_{\mathbb{Z}} \backslash G / K\cong G_{\mathbb{Z}} \backslash G_0 / K_0$. 
	In this section, we review the construction of the Satake compactification of $\mathcal{M}_{\mathrm{K3}}$ and $\mathcal{M}_{2d}$ with respect to the adjoint representation. 
	The references for this section are \cite{borel-ji-06,odaka-oshima-21,satake-60-2,satake-60-1}. 
	We also give more information on the topology of the Satake compactification near the boundary, 
	which will be used in the subsequent sections.
	
	\subsection{Lie algebra set-up}\label{section-lie-algebra}
	In this subsection we fix some notation on Lie algebras.
	Let $w_1,\ldots,w_{22}$ be an orthogonal basis of $\Lambda_{K3, \mathbb{R}}$ with 	
	\begin{equation}\label{basis-of-W}
		w_i^2 = 1 \quad (i \le 3), 
		\qquad 
		w_i^2 = -1 \quad (i \ge 4).        
	\end{equation}
	Consider the representation of $\mathfrak{g} = \mathfrak{so}(3,19)$ on $\Lambda_{K3,\mathbb{C}}$. 
	In the basis $w_1,\ldots,w_{22}$, the Lie algebra $\mathfrak{g}$ can be written as
	\[
	\mathfrak{g} =
	\left\{
	\begin{pmatrix}
		A_{3\times 3} & B_{3\times 19} \\
		B^{T}_{19\times 3} & C_{19\times 19}
	\end{pmatrix}
	\;\middle|\;
	A + A^{T} = 0,\; C + C^{T} = 0
	\right\}.
	\]
	We take a maximal compact subalgebra 
	\[
	\mathfrak{k} = \mathfrak{so}(3)\oplus \mathfrak{so}(19) =
	\left\{
	\begin{pmatrix}
		A_{3\times 3} & 0 \\
		0 & C_{19\times 19}
	\end{pmatrix} \in \mathfrak{g}
	\right\},
	\]
	and its orthogonal complement (with respect to the Killing form)
	\[
	\mathfrak{p} =
	\left\{
	\begin{pmatrix}
		0 & B \\
		B^T & 0
	\end{pmatrix} \in \mathfrak{g}
	\right\}.
	\]
	
	Define $\mathfrak{h} = \{ H(\bm{a})=H(a_1,\dots,a_{11}) : a_i\in \mathbb{R}\}$,
	where
	\begin{equation} \label{expression-H(a)}
		H(\bm{a}) =
		\begin{pmatrix}
			\begin{array}{c|c|c}
				0 & 0 & \begin{matrix}
					&  & a_1 \\
					& a_2 & \\
					a_3 &  & \\
				\end{matrix} \\ \hline
				0 & \begin{matrix}
					0 & a_4 & & & \\
					-a_4 & 0 & & & \\
					& & \ddots & & \\
					& & & 0 & a_{11} \\
					& & & -a_{11} & 0
				\end{matrix} & 0 \\ \hline
				\begin{matrix}
					& & a_3 \\
					& a_2 & \\
					a_1 & &
				\end{matrix} & 0 & 0
			\end{array}
		\end{pmatrix}.
	\end{equation}
	
	Then $\mathfrak{h}$ is a Cartan subalgebra of $\mathfrak{g}$, and $\mathfrak{h}^- \coloneqq \mathfrak{h} \cap \mathfrak{p}$ is a maximal abelian subalgebra of $\mathfrak{p}$.  
	Define linear functions $e_k: \mathfrak{h} \to \mathbb{C}$ by
	\[
	e_k\big(H(a_1,\, \ldots,a_{11})\big) =
	\begin{cases}
		a_k, & k=1,2,3, \\[2mm]
		\mathrm{i}\,a_k, & 4 \le k \le 11.
	\end{cases}
	\]
	Then $\{\pm e_k\}$ gives all the weights of the representation $\Lambda_{K3,\mathbb{C}}$.  The weight vectors associated to $e_k$ are given by
	\begin{equation}
		\label{basis-of-W-w-pm-e_k}
		w_{\pm e_k} =
		\begin{cases}
			\frac{1}{\sqrt{2}}(w_k \pm w_{23-k}), & k=1,2,3, \\[2mm]
			w_{2k-4} \pm i w_{2k-3}, & 4 \le k \le 11.
		\end{cases}
	\end{equation}
	By choosing suitable $w_1,\ldots,w_{22}$ we may assume $w_{\pm e_k}$ are primitive elements in $\Lambda_{\mathrm{K3}}$ for $k=1,2,3$.
	
	Let $R$ be the root system of $\mathfrak{g}_{\mathbb{C}}$, and 
	\[
	\Pi = \{ e_1-e_2,\, e_2-e_3,\, \ldots,\, e_{10}-e_{11},\, e_{10}+e_{11} \}
	\]
	the set of (ordered) simple roots of $\mathfrak{g}_{\mathbb{C}}$, and
	\[
	\Pi^- = \{ e_1-e_2,\, e_2-e_3,\, e_3 \}
	\]
	the restricting roots of $\Pi$ on $\mathfrak{h}^-$.  
	Given a subset $\Pi'\subset\Pi$, we define the corresponding Lie subalgebra 
	\[
	\mathfrak{g}_\mathbb{C}' = \mathfrak{g}_\mathbb{C}(\Pi') := \{\Pi'\}_\mathbb{C} + 
	\sum_{\alpha \in R'} \mathfrak{g}_\alpha,
	\]
	where $\{\Pi'\}_\mathbb{C}$ denotes the linear space generated by $\Pi'$ (view $e_k$ as elements of $\mathfrak{h}_\mathbb{C}$ via the Killing form), and $R'$ is the root subsystem of $\mathfrak{g}_\mathbb{C}$ generated by $\Pi'$. Set
	\[
	\mathfrak{g}' = \mathfrak{g} \cap \mathfrak{g}_\mathbb{C}', 
	\quad
	\mathfrak{k}' = \mathfrak{k} \cap \mathfrak{g}_\mathbb{C}',
	\]
	and $G'_0,K'_0$ the corresponding Lie subgroups of $G_0,K_0$ associated with $\mathfrak{g}'$ and $\mathfrak{k}'$.
	
	\subsection{Compactification of $G/K\cong G_0/K_0$}
	We first study the structure of the Satake compactification of the symmetric space $\mathrm{Gr}^+_3(\Lambda_{\mathrm{K3},\mathbb{R}})\cong G/K$ with respect to the adjoint representation. 
	The Satake compactification of $G_{\mathbb{Z}} \backslash G / K$ will be constructed later by taking the $G_{\mathbb{Z}}$ quotient of a partial compactification of $G/K$.
	
	Let $V = \mathfrak{g}_{\mathbb{C}}$ and $\rho = \mathrm{Ad}: G\to SL(V)$. Then $\rho$ is a faithful irreducible representation of highest weight $\lambda_\rho = e_1+e_2$.
	By the unitary trick, there exists a Hermitian inner product $\langle \cdot,\cdot\rangle$ on $V$ such that $\rho(G)\cap SU(V)=\rho(K)$.  
	We can define the adjoint $A^*\in SL(V,\mathbb{C})$ associated with $A\in SL(V,\mathbb{C})$ by
	$
	\langle A^*v,u\rangle = \langle v,Au\rangle.
	$
	Let $\mathscr{P}(V)$ denote the space of positive-definite Hermitian matrices of determinant $1$, then there is an isomorphism
	\[
	SL(V,\mathbb{C})/SU(V) \;\cong\; \mathscr{P}(V), \qquad A \mapsto AA^*.
	\]
	Thus we obtain an embedding
	\[
	\rho : G/K \;\hookrightarrow\; SL(V,\mathbb{C})/SU(V)\;\cong\; \mathscr{P}(V)\;\hookrightarrow\; \mathbb{P}(\mathscr{H}(V)),
	\]
	where $\mathscr{H}(V)$ is the space of Hermitian matrices of degree $n$, and $\mathbb{P}(\mathscr{H}(V))$ is the associated real projective space.
	
	\begin{defi}
		The Satake compactification $\overline{G/K}^{Sat}$ with respect to the adjoint representation $\rho=\mathrm{ad}$ is defined as the closure of $\rho(G/K)$ in $\mathbb{P}(\mathscr{H}(V))$.  
	\end{defi}
	
	The structure of $\overline{G/K}^{Sat}$ is studied by Satake \cite[Thm.~1]{satake-60-1}. Let $\Pi'$ be a subset of $\Pi$, 
	define $\lambda(\Pi')$ to be the set of all the weights of $\rho$ of the form 	
	\begin{equation}
		\label{equation-definition-lambda-Pi}
		\lambda=\lambda_\rho - \sum m_i \gamma_i \quad (m_i \in \mathbb{Z}^{\ge 0}, \ \gamma_i \in \Pi'),
	\end{equation}
	and let $V' := \sum_{\lambda\in \lambda(\Pi')} V_\lambda$. Then $V'$ is invariant under $\rho(G'_0)$, and we obtain a map
	\[
	\rho': G_0'/K_0' \;\longrightarrow\; SL(V')/SU(V') \;\hookrightarrow\; \mathbb{P}(\mathscr{H}(V')) \;\hookrightarrow\; \mathbb{P}(\mathscr{H}(V)),
	\]
	where $\mathbb{P}(\mathscr{H}(V')) \hookrightarrow \mathbb{P}(\mathscr{H}(V))$ is defined by assigning $0$ on the orthogonal space of $V'$. 
	
	Note that $G$ acts on $\mathbb{P}(\mathscr{H}(V))$ via the representation $\rho$. 
	We denote by $P'$ the parabolic subgroup associated with $\Pi'$, which is the subgroup of $G$ preserving $V'$. Then $P'$ is exactly the subgroup of $G$ preserving $\rho'(G_0'/K_0')$ (\cite[\S4.2]{satake-60-1}).
	
	\begin{defi}
		A subset $\Pi'^- \subset \Pi^-$ is called \emph{$\rho$-connected} if $\Pi'^- \cup \{\lambda_\rho\}$ cannot be decomposed into two mutually orthogonal parts.
	\end{defi}
	
	Define the map $c:\Pi \to \Pi^-$ by
	\[
	c(\alpha) =
	\begin{cases}
		\alpha, & \alpha = e_1 - e_2 \text{ or } e_2-e_3, \\[2mm]
		e_3, & \text{for all other simple roots $\alpha \in \Pi$.}
	\end{cases}
	\]
	
	Let $\Pi'^-$ be a subset of $\Pi^-$ and set $\Pi' := c^{-1}(\Pi'^-)$. Then it follows that $\rho'=\rho(\Pi')$ is an embedding if and only if $\Pi'^- $ is $\rho$-connected \cite[p.88]{satake-60-1}. Moreover, by \cite[Thm.~1]{satake-60-1}, we have:
	
	\begin{thm} \label{thm-satake-G/K}
		The compactification has the form of a disjoint union:
		\[
		\overline{G/K}^{Sat} = \overline{\rho(G/K)} 
		= \bigsqcup_{\Pi'^-, g} g \cdot \rho'\left( G_0'/K_0'\right),
		\]
		where $\Pi'^-$ ranges over all $\rho$-connected subsets of $\Pi^-$ and $g$ ranges over $G/P'$. 
		Here $G_0',K_0'$ are given by $\Pi' := c^{-1}(\Pi'^-)$ as in \S\ref{section-lie-algebra}.
	\end{thm}
	
	It follows from \cite[Thm.~2]{satake-60-1} that the structure of $\overline{G/K}^{Sat}$ as a $G$-space is independent of the choice of $(\mathfrak{k},\mathfrak{h},\Pi)$ and the Hermitian inner product $\langle \cdot,\cdot\rangle$ on $V$.
	
	\subsubsection{Structure of boundary components}\label{section-structure-boundary-components}
	Since $\lambda_{\rho}=e_1+e_2$, we know that $V \cong \wedge^2 \Lambda_{\mathrm{K3},\mathbb{C}}$, and the weight vector corresponding to $\pm e_i \pm e_j$ in $V$ can be identified with $w_{\pm e_i}\wedge w_{\pm e_j}$. The boundary components are given by $\rho$-connected subsets $\Pi'^-\neq \Pi^-$, and there are four different types of boundaries:
	
	\begin{enumerate}
		\item[(a)] $\Pi'^- = \{e_2-e_3,\, e_3\}$.  
		Then 
		\[
		G_0'/K_0' = SO_0(2,18) / SO(2)\times SO(18),
		\]
		and $SO_0(2,18)$ is identified with the subgroup fixing $w_{\pm e_1}$. Moreover, $V'$ is spanned by the weight spaces $e_1 \pm e_i$. Thus $P'$ is the subgroup of $G$ preserving the isotropic weight vector $w_{e_1} = \frac{1}{\sqrt{2}}(x_1 +  x_{22})$.
		
		\item[(b)] $\Pi'^- = \{e_1-e_2,\, e_2-e_3\}$.  
		Then 
		\[
		G_0'/K_0' = SL(3,\mathbb{R}) / SO(3).
		\] 
		An element $g\in SL(3,\mathbb{R})$ acts on $\Lambda_{\mathrm{K3},\mathbb{C}}$ by sending
		\[
		(w_{e_1},w_{e_2},w_{e_3})^T \mapsto (g^{-1})^T \cdot (w_{e_1},w_{e_2},w_{e_3})^T, \quad 
		(w_{-e_1},w_{-e_2},w_{-e_3})^T \mapsto g \cdot (w_{-e_1},w_{-e_2},w_{-e_3})^T,
		\]
		while fixing $w_{\pm e_k}$ for $k\geq 4$. Moreover, $V'$ is spanned by the weight spaces corresponding to $e_1+e_2,\, e_1+e_3,\, e_2+e_3$. Thus $P'$ is the subgroup of $G$ preserving the isotropic subspace $\langle w_{e_1}, w_{e_2}, w_{e_3} \rangle$.
		
		\item[(c)] $\Pi'^- = \{e_2-e_3\}$.  
		Then 
		\[
		G_0'/K_0' = SL(2,\mathbb{R}) / SO(2).
		\] 
		An element $g\in SL(2,\mathbb{R})$ acts on $\Lambda_{\mathrm{K3},\mathbb{C}}$ by sending
		\[
		(w_{e_1},w_{e_2})^T \mapsto (g^{-1})^T \cdot (w_{e_1},w_{e_2})^T, \quad
		(w_{-e_1},w_{-e_2})^T \mapsto g \cdot (w_{-e_1},w_{-e_2})^T,
		\]
		while fixing all other weight vectors. Moreover, $V'$ is spanned by the weight spaces corresponding to $e_1+e_2,\, e_1+e_3$. Thus $P'$ is the subgroup of $G$ preserving the isotropic flag
		$
		\langle w_{e_1} \rangle \subset \langle w_{e_1}, w_{e_2}, w_{e_3} \rangle.
		$
		
		\item[(d)]  $\Pi'^- = \emptyset$. Then $G_0'$ is trivial and $V'$ is spanned by the weight spaces corresponding to $e_1+e_2$. Thus $P'$ is the subgroup of $G$ preserving the isotropic subspace $\langle w_{e_1}, w_{e_2} \rangle$.
	\end{enumerate}

	\subsection{Compactification of $\overline{\mathcal{M}}_{\mathrm{K3}}^{\mathrm{ad}}$ I: as a set}\label{section-compactification-as-set}

	From now on we need to use the rational structure of $G$.  
	By choosing a suitable basis $w_1,\dots,w_{22}$ of $\Lambda_{\mathrm{K3},\mathbb{R}}$ in \S \ref{section-lie-algebra}, we may assume $K, \mathfrak{h}^-, \Pi^-$ are defined over $\mathbb{Q}$. 	
	
	\begin{defi}
		Denote $S^\mathbb{Q}$ the partial compactification of $G/K$ by adding those rational boundaries, i.e.
		\[
		S^{\mathbb{Q}} := \bigcup_{\Pi^-} G_{\mathbb{Q}} \cdot \rho'\bigl(G_0'/K_0'\bigr) 
		\;\subset\; \overline{G/K}^{\mathrm{ad}}.
		\]
	\end{defi}
	
	\begin{defi}
		The Satake compactification of $\mathcal{M}_{\mathrm{K3}}$ associated with the adjoint representation is defined as the quotient
		$
		\overline{\mathcal{M}}_{\mathrm{K3}}^{\mathrm{ad}} 
		\;=\; G_{\mathbb{Z}} \backslash S^{\mathbb{Q}}
		$ as a set.
	\end{defi}
	The topology on $\overline{\mathcal{M}}_{\mathrm{K3}}^{\mathrm{ad}}$ will be defined later.
	From Theorem \ref{thm-satake-G/K}, we see that $\overline{\mathcal{M}}_{\mathrm{K3}}^{\mathrm{ad}}$ is obtained by adding some disjoint components of the form $(G_{\mathbb{Z}}\cap P')\backslash G'_0/K'_0$.
	Now we count the number of added boundary components.
	
	Let $q \cdot \rho'\!\left( G_0'/K_0' \right)$ be a boundary component in $S^\mathbb{Q}$. We can associate it with a flag $qF'$ as described in \S\ref{section-structure-boundary-components}.
	
	We claim that distinct boundary components give rise to distinct flags. In fact, suppose two components give rise to the same flag. Then they must be of the same type. Without loss of generality, we may assume that the two components are $\rho'\!\left( G_0'/K_0' \right)$ and $q \cdot \rho'\!\left( G_0'/K_0' \right)$, and that the associated flags satisfy $qF'=F'$. Thus, from the properties of $P'$, we know that $q\in P'$ and $\rho'\!\left( G_0'/K_0' \right)=q \cdot \rho'\!\left( G_0'/K_0' \right)$.
	
	Thus the number of boundary components added in $\overline{\mathcal{M}}_{\mathrm{K3}}^{\mathrm{ad}}$ equals the number of $G_{\mathbb{Z}}$-conjugacy classes of those rational flags $qF'$ in $\Lambda_{\mathrm{K3}}$.
	By \cite[V.Thm.~6 and V.\S2.3]{serre-73}, there is a single $G_\mathbb{Z}$-conjugacy class of rational isotropic subspaces of dimension $1$ and $2$, while there are two distinct $G_\mathbb{Z}$-conjugacy classes of rational isotropic subspaces $E$ of dimension $3$, distinguished by the quadratic form on $E^\perp/E$, which can be either $(-E_8)^2$ or $-\Gamma_{16}$.
	These two types of $E$ are conjugate under $G_\mathbb{Q}$ by \cite[IV.Thm.~3]{serre-73}.
	
	Hence, we obtain two boundary components of types (b) and (c), and one boundary component of types (a) and (d):
	\begin{prop}
		The Satake compactification has the structure
		\[
		\overline{\mathcal{M}}_{\mathrm{K3}}^{\mathrm{ad}}
		= \mathcal{M}_{\mathrm{K3}} \;\sqcup\; \mathcal{M}(a) \;\sqcup\; \mathcal{M}(b_1) \;\sqcup\; \mathcal{M}(b_2) \;\sqcup\; \mathcal{M}(c_1) \;\sqcup\; \mathcal{M}(c_2) \;\sqcup\; \mathcal{M}(d).
		\]
		The boundary components added are arithmetic quotients of $G_0'/K_0'$ as in \S\ref{section-structure-boundary-components}.
		Here $\mathcal{M}(b_1), \mathcal{M}(c_1)$ correspond to the case $-\Gamma_{16}$, while $\mathcal{M}(b_2), \mathcal{M}(c_2)$ correspond to the case $(-E_8)^2$.
	\end{prop}
	After we define the topology on $\overline{\mathcal{M}}_{\mathrm{K3}}^{\mathrm{ad}}$, we will see their closure relations are given by
	\[
	\mathcal{M}(d) \subset \overline{\mathcal{M}(c_i)} \subset \overline{\mathcal{M}(b_i)}, 
	\qquad \overline{\mathcal{M}(c_i)} \subset \overline{\mathcal{M}(a)}.
	\]
	
	\begin{rmk}
		Although not used in this paper, we mention that there is a map from $\overline{\mathcal{M}_{W}}$ to $\overline{\mathcal{M}(a)}$, which is generically 2-1 by identifying Jacobian elliptic K3 surfaces with conjugate complex structures.  In fact, $\mathcal{M}_W\cong O_\mathbb{Z}\backslash O(2,18)/SO(2)\times O(18)$, and $\overline{\mathcal{M}_W}$ is isomorphic to the Satake compactification of $\mathcal{M}_W$
		(see \cite[\S 7.2.3]{odaka-oshima-21}).
	\end{rmk}
	\subsection{Compactification of $\overline{\mathcal{M}}_{\mathrm{K3}}^{\mathrm{ad}}$ II: topology}
	
	\subsubsection{Siegel sets and fundamental sets.}
	We first recall the notion of a Siegel set, which is used to construct fundamental sets for the action of $G_{\mathbb{Z}}$, and is crucial in the description of the topology of the Satake compactification $\overline{\mathcal{M}}_{\mathrm{K3}}^{\mathrm{ad}}
	\;=\; \overline{G_{\mathbb{Z}} \backslash G/K}^{\mathrm{ad}}$.
	
	Let $G = NAK$ be the Iwasawa decomposition of $G$, where $A = \exp(\mathfrak{h}^-)$ and $N$ is the unipotent subgroup corresponding to the Lie subalgebra
	\[
	\mathfrak{n} \;=\; 
	\Biggl(\;\bigoplus_{\alpha \in R,\, \alpha>0} \mathfrak{g}_\alpha \Biggr) \cap \mathfrak{g}.
	\]
	
	\begin{defi}
		A \emph{Siegel set} for $G$ (with respect to $K,\mathfrak{h},\Pi$) is a subset of $G/K$ of the form
		\[
		\mathfrak{S}(N_0,t) \;=\; N_0 A_t K,
		\]
		where
		\[
		A_t \;=\; \exp\{H \in \mathfrak{h}^- : \lambda(H) \ge t \ \text{for all}\ \lambda \in \Pi^-\}
		\]
		and $N_0 \subset N$ is a compact set with nonempty interior.
	\end{defi}
	
	Given a subset $\Pi'$ of $\Pi$, define
	\[
	\mathfrak{n}' = \Bigl(\bigoplus_{\alpha \in R', \alpha > 0} \mathfrak{g}_\alpha \Bigr) \cap \mathfrak{g}
	= \mathfrak{n} \cap \mathfrak{g}', 
	\qquad 
	\tilde{\mathfrak{n}}' = \Bigl(\bigoplus_{\alpha \notin R', \alpha > 0} \mathfrak{g}_\alpha \Bigr) \cap \mathfrak{g},
	\qquad
	\mathfrak{h}'^- = \mathfrak{h}^- \cap \mathfrak{g}' ,
	\]
	and let $N', \widetilde{N}', A'$ be the corresponding Lie subgroups.  
	Then $N = N' \widetilde{N}'$, with $\widetilde{N}'$ acting trivially on $\rho(G_0'/K_0')$. Moreover, any nontrivial $n' \in N'$ has no fixed point in $G_0'/K_0'$ by the Iwasawa decomposition $G_0' = N'A'K_0'$.
	
	Now we describe the topology of the closure $\overline{\mathfrak{S}}$ of $\mathfrak{S}$ in $\overline{G/K}^{\mathrm{ad}}$. 
	The following theorem is proved in \cite[§4.1]{satake-60-1}.
	
	\begin{thm}\label{siegel_x}
		Let $\mathfrak{S}(N_0,t)$ be a Siegel set and $x \in \overline{\mathfrak{S}}$. Then there is a $\rho$-connected set $\Pi'^-$ such that $x \in \rho'(G_0'/K_0')$, and we may write $x = \rho'(n'a'K_0')$, where $n' \in N'$ and $a'=\exp(H') \in A'$.
		Let $N_0^*$ be a neighborhood of $N_0 \cap n'\widetilde{N}'$ in $N_0$.  
		For $\varepsilon > 0$, let $\mathfrak{h}'_\varepsilon$ be the set of $H \in \mathfrak{h}'^-$ such that
		\begin{enumerate}
			\item[(i)] $(\lambda_\rho - \lambda)(H - H') < \varepsilon$ for $\lambda \in \lambda(\Pi')$ (cf.~\eqref{equation-definition-lambda-Pi}), 
			\item[(ii)] $(\lambda_\rho - \lambda)(H) > 1/\varepsilon$ for $\lambda \in \lambda(\Pi) \setminus \lambda(\Pi')$.   
		\end{enumerate}
		Define
		\[
		\mathfrak{S}_x = \mathfrak{S}_x(N_0^*,\varepsilon) \coloneqq N_0^* \exp(\mathfrak{h}'_\varepsilon) K.
		\]
		Then $\overline{\mathfrak{S}_x}$ is a neighborhood of $x$ in $\overline{\mathfrak{S}}$, and such sets form a neighborhood system of $x$ in $\overline{\mathfrak{S}}$.
	\end{thm}
	
	By choosing suitable $N_0$ and $N_0^*$ in Theorem~\ref{siegel_x}, we may always assume $N_0^*$ is connected, thus $\mathfrak{S}_x$ is also connected. Moreover, from Theorem~\ref{siegel_x} we see that $\overline{\mathfrak{S}(N_0,t)} \cap \rho'(G_0'/K_0')$ is the $\rho'$-image of $N_0' A_t' K_0'$, where $N_0'$ is the image of $N_0$ under the projection $N \to N'$. Thus we have:
	
	\begin{crl}\label{crl-intersection-siegel-set}
		The intersection $\overline{\mathfrak{S}} \cap \rho'(G_0'/K_0')$ can be identified with a Siegel set of $G_0'/K_0'$.
	\end{crl}
	
	\begin{defi}\label{tt-1210-defi}
		We say that a subset $\Omega \subset G/K$ is a \emph{fundamental set} for the action of $G_{\mathbb{Z}}$ if 
		\begin{enumerate}
			\item[(i)] $G/K = G_{\mathbb{Z}} \cdot \Omega$;
			\item[(ii)] there are only finitely many $\gamma \in G_{\mathbb{Z}}$ such that $\gamma \Omega \cap \Omega \neq \varnothing$.
		\end{enumerate}
	\end{defi}
	
	The following theorem was proved by Borel and Harish-Chandra \cite[Thm.~6.5]{borel-harish-62} using reduction theory (see also \cite[Thm.~13.1, Thm.~15.4]{borel-69}).
	
	\begin{thm}\label{thm-siegel-set-fundamental-domain}
		Given a Siegel set $\mathfrak{S}$ and $q \in G_{\mathbb{Q}}$, there are only finitely many $\gamma \in G_{\mathbb{Z}}$ such that $\gamma \mathfrak{S} \cap q \mathfrak{S} \neq \varnothing$.
		Moreover, one can find a Siegel set $\mathfrak{S}$ and a finite set $F \subset G_{\mathbb{Q}}$ such that
		\[
		\Omega = \Omega(F,\mathfrak{S}) := F \mathfrak{S}
		\]
		is a fundamental set for $G_{\mathbb{Z}}$. 
	\end{thm}

	As a corollary, we may take a fundamental set $\Omega$ containing any prescribed Siegel set.  
	Further, we have $G_{\mathbb{Q}} \cdot \mathfrak{S} = G/K$ for any Siegel set $\mathfrak{S}$.  
	Moreover, we remark that Theorem~\ref{thm-siegel-set-fundamental-domain} also holds for a Siegel set of $G_0'$.
	
	\begin{prop}\label{prop-Gz-Omega=SQ}
		If $\Omega$ is a fundamental set as above, then
		\[
		G_{\mathbb{Z}} \cdot \overline{\Omega} 
		\;=\; G_{\mathbb{Q}} \cdot \overline{\Omega} 
		\;=\; S^{\mathbb{Q}}.
		\]
	\end{prop}
	
	\begin{proof}
		We first prove that $G_{\mathbb{Z}} \cdot \overline{\Omega} = G_{\mathbb{Q}} \cdot \overline{\Omega}$. 
		Suppose conversely that $x \in q \overline{\mathfrak{S}}$ for some $q \in G_{\mathbb{Q}}$ but $x \notin G_{\mathbb{Z}} \overline{\Omega}$.  
		By the property of Siegel sets discussed above, there are finitely many $\gamma_k \in G_{\mathbb{Z}}$ such that $\gamma_k \overline{\Omega} \cap q \overline{\mathfrak{S}} \neq \varnothing$.  
		Since $x \notin G_{\mathbb{Z}} \overline{\Omega}$, we can choose a neighborhood $U$ of $x$ in $q \overline{\mathfrak{S}}$ such that $G_{\mathbb{Z}} \overline{\Omega} \cap U = \varnothing$.  
		However, $U$ must contain points of $G/K\subset G_\mathbb{Z}\overline{\Omega}$, a contradiction.  
		Thus $G_{\mathbb{Z}} \overline{\Omega} = G_{\mathbb{Q}} \overline{\Omega}$.
		
		From Corollary~\ref{crl-intersection-siegel-set}, $\overline{\mathfrak{S}} \cap \rho'(G_0'/K_0')$ is identified with a Siegel set $\mathfrak{S}'(N_0',t')$ in $G_0'/K_0'$, thus $G_{\mathbb{Q}} \overline{\mathfrak{S}}$ covers $\rho'(G_0'/K_0')$.  
		Hence $G_{\mathbb{Q}} \overline{\Omega} = S^{\mathbb{Q}}$.
	\end{proof}
	
	\subsubsection{Topology on $\overline{\mathcal{M}}_{\mathrm{K3}}^{\mathrm{ad}}$}

Given a fundamental set $\Omega = F \mathfrak{S}$ as in Theorem \ref{thm-siegel-set-fundamental-domain}, Satake (\cite[Thm.~1]{satake-60-2}) constructs a topology $T^{\mathbb{Q}}$ on $G_{\mathbb{Z}} \cdot \overline{\Omega}$ (which equals $S^{\mathbb{Q}}$ by Proposition \ref{prop-Gz-Omega=SQ}).  
This topology has the following properties:

\begin{enumerate}
	\item The topology $T^{\mathbb{Q}}$ restricts to the standard topology on $\overline{\Omega}$ (which is the subspace topology induced from $\mathbb{P}(\mathscr{H}(V))$).
	\item The construction of $T^{\mathbb{Q}}$ is independent of the specific choice of the fundamental set $\Omega$ (\cite[Rem.~3]{satake-60-2}).
	\item The action of any rational element $q \in G_{\mathbb{Q}}$ on the space $(S^{\mathbb{Q}}, T^{\mathbb{Q}})$ is a homeomorphism. In particular, $T^{\mathbb{Q}}$ restricts to the standard topology on \emph{every} Siegel set.
\end{enumerate}
\begin{defi}
	The \textit{Satake topology} on $\overline{\mathcal{M}}_{\mathrm{K3}}^{\mathrm{ad}}$ is defined to be the quotient topology induced by $T^{\mathbb{Q}}$ under the $G_{\mathbb{Z}}$-quotient.  
	In this topology, $\overline{\mathcal{M}}_{\mathrm{K3}}^{\mathrm{ad}}$ is compact and Hausdorff.
\end{defi}

We can describe the neighborhood system of the Satake topology as follows (see \cite[\S2]{satake-60-2}) for more details).
Suppose $x_\infty \in q \rho'(G_0'/K_0')$ is a boundary point of $S^{\mathbb{Q}}$ for some $q \in G_{\mathbb{Q}}$.  
By Theorem \ref{thm-siegel-set-fundamental-domain} and Corollary \ref{crl-intersection-siegel-set}, the following set is finite:
\begin{equation}\label{equation-Gzx-intersect-Omega}
	\{(q,x) \in F \times \overline{\mathfrak{S}} \mid qx \in G_{\mathbb{Z}} x_\infty\},
\end{equation}
denote it by $\{(q_1, x_1), \dots, (q_n, x_n)\}$.  
Take $\mathfrak{S}_{x_i} \subset \mathfrak{S}$ as in Theorem~\ref{siegel_x}.  
Let $\pi:S^\mathbb{Q}\to \overline{\mathcal{M}}_{\mathrm{K3}}^{\mathrm{ad}}$ be the quotient map.
Then sets of the form
\begin{equation}\label{neighborhood-U}
	U \;=\; \bigcup_i \pi\!\left(q_i \, \overline{\mathfrak{S}_{x_i}}\right)
\end{equation}
constitute a neighborhood system of $\pi(x_\infty)$.

		\subsubsection{More on the topology}\label{section-more-on-topology}
	
	Here we give some results related to the topology for later use.
	
	\begin{prop}\label{prop-topological-basis-connectness}
		Let $\eta_\infty$ be a boundary point of $\overline{\mathcal{M}_{\mathrm{K3}}}^{Sat}$.  
		Then there exists a neighborhood system $\{U_j\}$ of $\eta_\infty$ in $\overline{\mathcal{M}_{\mathrm{K3}}}^{Sat}$ such that $U_j \cap \mathcal{M}_{\mathrm{K3}}$ is connected.
	\end{prop}
	\begin{proof}
		Let $\pi: S^{\mathbb{Q}} \to \overline{\mathcal{M}_{\mathrm{K3}}}^{Sat}$ be the $G_{\mathbb{Z}}$-quotient map, and write the boundary point as $\eta_\infty = \pi(x_\infty)$. Let $\Omega = F \mathfrak{S}$ be a fundamental set, and let $\{(q_1, x_1), \dots, (q_n, x_n)\}$ be the finite collection of pairs given by \eqref{equation-Gzx-intersect-Omega}. 
		It suffices to prove the following claim:
		\begin{itemize}
			\item 
			There exists an expanded Siegel set $\mathfrak{S}^* \supset \mathfrak{S}$ such that for every index $i \in \{1, \dots, n\}$, there is a sequence $\{\eta_k\} \subset \pi(q_i \mathfrak{S}) \cap \pi(q_1 \mathfrak{S}^*)$ converging to $\pi(x_\infty)$. 
		\end{itemize}
		
		If such a Siegel set $\mathfrak{S}^*$ exists, then $\pi(q_i \mathfrak{S}_{x_i}) \cap \pi(q_1 \mathfrak{S}_{x_1}^*) \neq \emptyset$ for any neighborhoods $\mathfrak{S}_{x_i}$ and $\mathfrak{S}_{x_1}^*$ given by Theorem~\ref{siegel_x}. Thus,
		\[
		U = \bigcup_i \pi\!\left(q_i \, \overline{\mathfrak{S}_{x_i}}\right)
		\]
		defined in \eqref{neighborhood-U} is an (arbitrarily small) neighborhood of $\eta_\infty$. Moreover, $U \cap \mathcal{M}_{\mathrm{K3}}$ is connected, because we may choose $\mathfrak{S}_{x_1}^*$ to be connected and each $\pi\!\left(q_i \, \overline{\mathfrak{S}_{x_i}}\right)$ intersects $\pi\!\left(q_1\, \overline{\mathfrak{S}_{x_1}}\right)$.
		
		We now prove the claim. Let $\Pi'' \subset \Pi$ be the set of simple roots orthogonal to $\Pi' \cup \{\lambda_\rho\}$. We can choose $H' \in \mathfrak{h}^-$ satisfying two conditions: $\gamma(H') = 0$ for all $\gamma \in \Pi' \cup \Pi''$, and $\lambda_\rho(H') > \lambda(H')$ for all weights $\lambda \notin \lambda(\Pi')$ (cf.~\eqref{equation-definition-lambda-Pi}). Thus, $H'$ commutes with any $g \in G_{\Pi' \cup \Pi''}$. 
		
		Define a sequence within our Siegel set:
		\[
		z_k = u \exp(kH') K \in \mathfrak{S}_{x_i},
		\]
		where $u \in NA$ is chosen so that $z_k \to x_i$ as $k \to \infty$. By setting $y_k = q_i z_k$ and $\eta_k=\pi(y_k)$, we obtain $\eta_k\to \pi(x_\infty)$ since $q_i x_i \in G_{\mathbb{Z}} x_\infty$.
		
		It suffices to prove $\eta_k\in\pi(q_1 \mathfrak{S}^*)$ for some Siegel set $\mathfrak{S}^*$. Let 
		\[
		G_{\Pi' \cup \Pi''} = N_{\Pi' \cup \Pi''} \, A_{\Pi' \cup \Pi''} \, K_{\Pi' \cup \Pi''}
		\]
		be the Iwasawa decomposition of $G_{\Pi' \cup \Pi''}$. By \cite[Prop.~7]{satake-60-1}, the parabolic subgroup $P'$ preserving $\rho'(G_0'/K_0')$ admits the decomposition
		\[
		P' = \widetilde{N}_{\Pi' \cup \Pi''} \, G_{\Pi' \cup \Pi''} \, Z = N \, A_{\Pi' \cup \Pi''} \, K_{\Pi' \cup \Pi''} \, Z,
		\]
		where $Z = \exp(i\mathfrak{h}^-) \cap G = \{\pm 1\}^2$. 
		
		Since $q_1^{-1} \gamma q_i$ maps $x_i$ to $x_1$ and both $x_1, x_i \in \rho'(G_0'/K_0')$, we have $q_1^{-1}\gamma q_i \in P'$. Writing $q_1^{-1}\gamma q_i u = n a k z$ according to the above decomposition of $P'$, we obtain
		\[
		q_1^{-1}\gamma y_k = q_1^{-1}\gamma q_i u \exp(kH') K = n a k z \exp(kH') K.
		\]
		
		Since $kz \in K$ commutes with $\exp(kH')$, this simplifies to
		\[
		q_1^{-1}\gamma y_k = n a \exp(kH') K.
		\]
		Therefore, $q_1^{-1}\gamma y_k \in \mathfrak{S}^*(N_0^*, t^*)$ provided $N_0^*$ contains $n$ and $t^*$ is sufficiently small. Thus $\eta_k=\pi(y_k) \in \pi(q_1 \mathfrak{S}^*)$ for some Siegel set $\mathfrak{S}^*$, and the conclusion follows.
	\end{proof}
	Next we write points in $G/K$ explicitly in matrix form.
	
	\begin{defi}\label{defi-std-basis}
		We say that $\{x_i\}$ $(1 \le i \le 22)$ is a \emph{standard basis} of $\Lambda_{\mathrm{K3},\mathbb{R}}\cong H^2(X,\mathbb{R})$ if $x_i \in \Lambda_{\mathrm{K3}}$ for $i \in \{1,2,3,20,21,22\}$, and its intersection matrix is
		\[
		\begin{pmatrix}
			\begin{array}{c|c|c}
				0 & 0 &
				\begin{matrix}
					&  & 1 \\
					& 1 &   \\
					1 &  &  
				\end{matrix}
				\\ \hline
				0 & A_{16\times 16} & 0
				\\ \hline
				\begin{matrix}
					&  & 1 \\
					& 1 &   \\
					1 &  &  
				\end{matrix} & 0 & 0
			\end{array}
		\end{pmatrix}.
		\]
		We say a standard basis $\{x_i\}$ is of \emph{$E_8$ type} if the integral quadratic form on $\langle x_1,x_2,x_3,x_{20},x_{21},x_{22} \rangle^\perp$ is $(-E_8)^{\oplus 2}$, and of \emph{$\Gamma_{16}$ type} if the integral quadratic form is $-\Gamma_{16}$.
	\end{defi}
	
	The basis $(w_{e_1}, w_{e_2}, w_{e_3}, w_4, \dots, w_{19}, w_{-e_3}, w_{-e_2}, w_{-e_1})$ given in \S\ref{section-lie-algebra} is a standard basis; conversely, any standard basis with $A = -\mathrm{Id}_{16\times 16}$ is obtained in this way.
	
	Given a standard basis $\mathcal{B} = \{x_i\}$, each $W \in \mathrm{Gr}^+_3(\Lambda_{\mathrm{K3},\mathbb{R}})$ can be represented by the coefficients of its orthonormal basis, and we may write the isomorphism  
	$G/K \cong \mathrm{Gr}^+_3(\Lambda_{\mathrm{K3},\mathbb{R}})$ by $gK \mapsto \boldsymbol{g} L_0$, where $\boldsymbol{g}$ denotes the coefficient matrix of $g$ in $\mathcal{B}$, and
	\[
	L_0 \coloneqq
	\begin{pmatrix}
		0 & 0 & 1 \\
		0 & 1 & 0 \\
		1 & 0 & 0 \\
		\vdots & \vdots & \vdots \\
		0 & 0 & 0 \\
		\vdots & \vdots & \vdots \\
		1 & 0 & 0 \\
		0 & 1 & 0 \\
		0 & 0 & 1
	\end{pmatrix}.
	\]
	
	\begin{lem}\label{lem-std-form-nexpHL0}
		Let $n \in N$ and $H = H(a+b+c, b+c, c, 0, \dots, 0)$ as in \eqref{expression-H(a)}. Then
		\[
		\boldsymbol{n \exp(H)} L_0 =
		\begin{pmatrix}
			\alpha_{11} & \alpha_{12} & \alpha_{13} \\
			\alpha_{21} & \alpha_{22} & \alpha_{23} \\
			\alpha_{31} & \alpha_{32} & \alpha_{33} \\
			\vdots & \vdots & \vdots \\
			e^{-c} & \alpha_{20,2} & \alpha_{20,3} \\
			0 & e^{-b-c} & \alpha_{21,3} \\
			0 & 0 & e^{-a-b-c}
		\end{pmatrix}.
		\]
		
		Conversely, if $x \in G/K$ is represented by $L$ of the above form, then $x = n \exp(H) K$ as above. Moreover, $x$ is $G_{\mathbb{Z}}$-equivalent to $n' \exp(H) K$, where $n'$ lies in a given compact set of $N$. In particular, if $a,b,c \ge -C$ for some constant $C$, then $n' \exp(H)$ lies in a given Siegel set.
	\end{lem}
	
	\begin{proof}
		From the definition of $\mathfrak{n}$, we know that $N$ is upper triangular in $\mathcal{B}$. Thus $\boldsymbol{n \exp(H)} L_0$ is of the stated form.  
		Conversely, since $L$ represents a positive definite $3$-plane in $\Lambda_{\mathrm{K3},\mathbb{R}}$, there exists $g \in G$ such that $L = \boldsymbol{g} L_0$.  
		Write $g = n a k$ as its Iwasawa decomposition. From the expression of $L$, we must have $\boldsymbol{k} L_0 = L_0$, and hence $L = \boldsymbol{n a} L_0$.  
		Since $n$ is upper triangular, we have $a = \exp(H)$ determined by the coefficients.
		
		Moreover, since $N_{\mathbb{Z}} \backslash N$ is compact by \cite[Thm.~8.7]{borel-69}, we may reduce to the case that $n$ lies in a compact set.
	\end{proof}
	
	\begin{prop}\label{prop-period-behaviour-characterized-boundary}
		Suppose we have a sequence given by $L_k = \boldsymbol{n_k \exp(H_k)} L_0$ as in Lemma~\ref{lem-std-form-nexpHL0} in a standard basis $\mathcal{B}$.
		Let ${[L_k]}$ be the associated point in $\mathcal{M}_{\mathrm{K3}}$. 
		Assume $a_k + b_k + c_k \to \infty$ and $a_k, b_k, c_k \ge -C$.Then there is a subsequence of ${[L_k]}$ convergent in $\overline{\mathcal{M}_{\mathrm{K3}}}$, and
		the limit is determined as follows:
		
		\begin{enumerate}
			\item If $b_k + c_k$ converges and $a_k \to \infty$, then ${[L_k]}$ converges to a boundary point of type (a).  
			The limit is identified with the limit of $(\alpha_{ij,k})_{2 \le i \le 21,\; 1 \le j \le 2}$, which can be viewed as a $2$-positive definite subspace in $\langle x_1, x_{22} \rangle^\perp$.
			
			\item If $b_k + a_k$ converges and $c_k \to \infty$, then ${[L_k]}$ converges to a boundary point of type (b).  
			The limit is identified with the limit of
			\[
			e^{c_k + \frac{2}{3} b_k + \frac{1}{3} a_k}
			\begin{pmatrix}
				e^{-c_k} & \alpha_{20,2,k} & \alpha_{20,3,k} \\
				0 & e^{-b_k - c_k} & \alpha_{21,3,k} \\
				0 & 0 & e^{-a_k - b_k - c_k}
			\end{pmatrix}.
			\]
			
			\item If $b_k$ converges and $a_k, c_k \to \infty$, then ${[L_k]}$ converges to a boundary point of type (c).  
			The limit is identified with the limit of
			\[
			e^{c_k + \frac{1}{2} b_k}
			\begin{pmatrix}
				e^{-c_k} & \alpha_{20,2,k} \\
				0 & e^{-b_k - c_k}
			\end{pmatrix}.
			\]
			
			\item If $b_k \to \infty$, then ${[L_k]}$ converges to a boundary point of type (d).
		\end{enumerate}
		
		In cases (2) and (3), the exact boundary component where the limit lies is determined by the type of the standard basis as in \S\ref{section-compactification-as-set}.
	\end{prop}
	
	\begin{proof}
		Let $\Pi'^- \subset \Pi^-$ denote the maximal $\rho$-connected set such that $\lambda(H_k)$ is bounded for $\lambda \in \Pi'^-$, and let $\Pi' = c^{-1}(\Pi'^-)$.  
		Write $H_k = H'_k + H''_k$, where $H'_k \in \mathfrak{h}'^{-}$ converges to $H'_\infty$, and $H''_k$ satisfies $\lambda(H''_k) = 0$ for $\lambda \in \Pi'$.
		For example, in case (1) we have $H''_k = H(a_k, 0, \dots, 0)$, and in case (2) we have $H''_k = H(c'_k, c'_k, c'_k, 0, \dots, 0)$ with $c'_k = c_k + \frac{2}{3} b_k + \frac{1}{3} a_k$.
		Then $a_k + b_k + c_k \to \infty$ implies $(\lambda_{\rho} - \lambda)(H''_k) \to \infty$ for $\lambda \in \lambda(\Pi) \setminus \lambda(\Pi')$.
		
		By Lemma \ref{lem-std-form-nexpHL0}, we may assume $n_k$ lies in a given compact set of $N$.
		After possibly taking a subsequence, we can write $n_\infty = \lim_{k\to\infty} n_k = n'_\infty \tilde{n}'_\infty \in N' \widetilde{N}'$.
		By the definition of the topology on $\overline{\mathfrak{S}}$, the limit of $q n_k e(H_k) K$ is identified with $q n'_\infty \exp(H'_\infty) K_0' \in q \rho'(G_0'/K_0')$.
		Thus the conclusion follows.
	\end{proof}
	
	\subsection{Construction of the geometric realization map}\label{section-construction-geometric-realization}
	
	Here we define the map $\overline{\Phi} : \overline{\mathcal{M}}_{\mathrm{K3}}^{\mathrm{ad}} \to \overline{\mathfrak{M}}$, following \cite[\S6.2]{odaka-oshima-21}.
	On $\mathcal{M}_{\mathrm{K3}}$, the map $\Phi : \mathcal{M}_{\mathrm{K3}} \to \mathfrak{M} \cup \mathfrak{M}^4$ is defined in \eqref{definition-phi}, and on the boundary $\overline{\Phi}$ is defined as follows:
	
	\begin{enumerate}
		\item[(i)] For $\eta \in \mathcal{M}(a) \cong O_\mathbb{Z}(2,18) \backslash SO_0(2,18)/(SO(2) \times SO(18))$,  
		$\eta$ can be represented by a $2$-positive definite plane in $\langle w_{e_1}, w_{-e_1} \rangle^\perp \subset \Lambda_{\mathrm{K3},\mathbb{R}}$.  
		After choosing an orientation, we can associate $\eta$ with a Jacobian elliptic K3 surface by Proposition~\ref{prop-e,f-give-elliptic}.  
		We define $\overline{\Phi}(\eta)$ to be the corresponding generalized KE metric of unit diameter.
		
		\item[(ii)] For $\eta \in \mathcal{M}(b_1) \cong GL(3,\mathbb{Z}) \backslash SL(3,\mathbb{R})/SO(3)$,  
		$\eta$ can be represented by $A \in SL(3,\mathbb{R})$.  
		We can associate to it a torus $T^3 = \mathbb{R}^3 / \langle v_1, v_2, v_3 \rangle$, where $v_i$ are the row vectors of $A$.  
		This gives a well-defined map $\overline{\Phi} : \mathcal{M}(b_1) \to \mathfrak{M}^3$, sending a point in $\mathcal{M}(b_1)$ to the flat torus quotient $T^3/\{\pm 1\}$ with unit diameter.
		
		\item[(iii)] For $\eta \in \mathcal{M}(c_1) \cong GL(2,\mathbb{Z}) \backslash SL(2,\mathbb{R})/SO(2)$, similarly define $\overline{\Phi}(\eta)$ to be the flat torus quotient $T^2/\{\pm 1\}$ determined by $A \in SL(2,\mathbb{R})$.
		
		\item[(iv)] Finally, for $\eta \in \mathcal{M}(b_2) \cup \mathcal{M}(c_2) \cup \mathcal{M}(d)$, define $\overline{\Phi}(\eta) = I^1$.
	\end{enumerate}
	
	We see from the definition that $\overline{\Phi}$ is injective on $\mathcal{M}(b_1) \cup \mathcal{M}(c_1)$, and we have $\overline{\Phi}(\mathcal{M}(b_1)) \subset \overline{\Phi}(\mathcal{M}(a))$.  
	Moreover, $\overline{\Phi}|_{\mathcal{M}(a)}$ is almost injective (\cite[Thm.~4]{ouyang-25}).  
	The continuity of $\overline{\Phi}$ will be proved in \S\ref{section-proof-of-main-thm}.  
	Since $\overline{\Phi}$ is a continuous surjection between compact sets, it is a closed map and hence a quotient map.
	
	For later use, we also define $\overline{\Phi}' : \overline{\mathcal{M}}_{\mathrm{K3}}^{\mathrm{ad}} \to \overline{\mathfrak{M}}$ by mapping $\eta \in \mathcal{M}(b_2) \cup \mathcal{M}(c_2)$ to a flat torus quotient as above, and letting $\overline{\Phi}'(\eta) = \overline{\Phi}(\eta)$ on the other components.
	\subsection{Compactification of polarized moduli spaces}\label{section-polarized-moduli-space}
	Fix a primitive element $\lambda \in \Lambda_{\mathrm{K3}}$ with $(\lambda,\lambda) = 2d > 0$.  
	Such a $\lambda$ is unique up to the action of the isometry group $O(\Lambda_{\mathrm{K3}})$ \cite[Cor.~14.1.10]{huybrechts-lecturenotes}.  
	Define the lattice
	$
	\Lambda_{2d} \coloneqq \lambda^{\perp} \subset \Lambda_{\mathrm{K3}}.
	$
	The polarized moduli space of K3 surfaces of degree $2d$ is then defined similarly as the quotient of the oriented Grassmannian:
	\begin{align*}
		\mathcal{M}_{2d} \;&\coloneqq\; O(\Lambda_{2d}) \backslash \mathrm{Gr}_2^{+,or}(\Lambda_{2d,\mathbb{R}})
		\;=\; O(\Lambda_{2d}) \backslash O(2,19) \big/ (SO(2)\times O(19)) \\
		\;&=\; O^+(\Lambda_{2d}) \backslash O(2,19) \big/ (O(2)\times O(19))
		\;=\; O^+(\Lambda_{2d}) \backslash G_{2d}/K_{2d}.
	\end{align*}
	Here
	$G_{2d}\cong O(2,19)$ is the stabilizer of $\lambda$ in $G$, $K_{2d}=K\cap G_{2d}\cong O(2)\times O(19)$, and $O^+(\Lambda_{2d})$ is the index-two subgroup of $O(\Lambda_{2d})$ preserving each component of the oriented Grassmannian $\mathrm{Gr}_2^{+,or}$.
	
	The moduli space $\mathcal{M}_{2d}$ parametrizes algebraic K3 surfaces (possibly with ADE singularities) equipped with a fixed polarization.
	There is a map
	\[
	f: \mathcal{M}_{2d} \longrightarrow \mathcal{M}_{\mathrm{K3}},
	\]
	which is generically $2$-to-$1$.
	

	As in the non-polarized case, we construct the Satake compactification $\overline{\mathcal{M}}_{2d}^{\mathrm{ad}}$ according to the adjoint representation of $G_{2d}$ in $V_{2d}=\mathfrak{g}_{2d,\mathbb{C}}$.  
	We have $\overline{\mathcal{M}}_{2d}^{\mathrm{ad}}=G_{\mathbb{Z},2d}\backslash S^{\mathbb{Q}}_{2d}$, where $S^{\mathbb{Q}}_{2d}$ is obtained by adding all rational boundary components in $\overline{G_{2d}/K_{2d}}^{\mathrm{ad}}$. Moreover, we have
	$
	\overline{G_{2d}/K_{2d}}^{\mathrm{ad}}
	\hookrightarrow
	\overline{G/K}^{\mathrm{ad}}
	$
	induced by
	\[
	G_{2d}/K_{2d}
	\hookrightarrow SL(V_{2d},\mathbb{C})/SU(V_{2d})
	\hookrightarrow SL(V,\mathbb{C})/SU(V)
	\hookrightarrow \mathbb{P}(\mathscr{H}(V)).
	\]
	This map sends $S^{\mathbb{Q}}_{2d}$ to $S^{\mathbb{Q}}$.  
	Hence, taking the quotient by $G_{\mathbb{Z},2d}$, we obtain a continuous map
	\[
	f_{2d}:\overline{\mathcal{M}}_{2d}^{\mathrm{ad}}\to \overline{\mathcal{M}}_{\mathrm{K3}}^{\mathrm{ad}},
	\]
	and we define the geometric realization map on $\overline{\mathcal{M}}_{2d}^{\mathrm{ad}}$ by $\overline{\Phi}_{2d}=\overline{\Phi}\circ f_{2d}$.  
	Then, as a corollary of Theorem \ref{thm-main-theorem}, $\overline{\Phi}_{2d}$ is continuous (Corollary \ref{tt-crl-12072200}).
	
	Furthermore, the GH closure of hyperkähler K3 surfaces with fixed polarization is given by the image of $\overline{\Phi}_{2d}$, and we obtain a classification of such collapsing limits assuming Theorem \ref{thm-main-theorem}:
	
	\begin{proof}[Proof of Corollary \ref{crl-fix-polarization}]
		Suppose we have a sequence $\eta_k\in \operatorname{Im}(f_{2d})$ converging to a boundary point $\eta_\infty$.  
		We may assume $\eta_k$ is the image of $\langle \omega_{1,k},\omega_{2,k},\omega_{3,k}\rangle$ under the $G_{\mathbb{Z}}$-quotient, and $(\omega_{1,k},\omega_{2,k},\omega_{3,k})$ is represented by a $3\times 22$ matrix with coefficients lying in Cases (1)--(4) of Proposition \ref{prop-period-behaviour-characterized-boundary}.  
		Let $\lambda_k=x_k\omega_{1,k}+y_k\omega_{2,k}+z_k\omega_{3,k}$ be the polarization class of degree $2d$.  
		Then $x_k^2+y_k^2+z_k^2=2d$ and all coefficients of $\lambda_k$ are integral.  
		Thus $c_k$ in Proposition \ref{prop-period-behaviour-characterized-boundary} must be bounded, and we are either in Case (4) with $\lambda=\sqrt{2d}\,\omega_{1,k}$, or in Case (1) with $\lambda=x_k\omega_{1,k}+y_k\omega_{2,k}$.  
		By Theorem \ref{thm-main-theorem}, in Case (4) the limit is a unit segment, and in Case (1) the limit is given by generalized KE metrics on $\mathbb{P}^1$ arising from elliptic K3 surfaces $\pi:X\to \mathbb{P}^1$, such that $[\omega]=[\operatorname{Re}\Omega]$ for a holomorphic volume form $\Omega$ on $X$.
	\end{proof}
	
	Next we analyze the boundary components of $\overline{\mathcal{M}}_{2d}^{\mathrm{ad}}$.  
	As in the non-polarized case, take  
	\[
	\Pi_{2d}=\{e_1-e_2,\, e_2-e_3,\, \dots,\, e_9-e_{10},\, 2e_{10}\}
	\]
	to be the set of simple roots of $\mathfrak{g}_{2d,\mathbb{C}}$, and let $\Pi_{2d}^-=\{e_1-e_2,\, e_2\}$ be its restriction to $\mathfrak{h}_{2d}^-$.  
	Since the maximal weight is $e_1+e_2$, there are three different $\rho$-connected subsets
	\[
	\Pi_{2d}'^-=\{e_1-e_2,\, e_2\},\qquad \{e_2\},\qquad \emptyset,
	\]
	and we have a decomposition
	\[
	\overline{\mathcal{M}}_{2d}^{\mathrm{ad}}
	=
	\mathcal{M}_{2d}
	\;\sqcup\;
	\mathcal{F}(l_i)
	\;\sqcup\;
	\mathcal{F}(p_j).
	\]
	Here $\mathcal{F}(l_i)$ is isomorphic to an arithmetic quotient of $SO_0(1,18)/SO(18)$; these boundary components map to $\mathcal{M}(a)$ by the description of the Satake compactification, and each $\mathcal{F}(p_j)$ consists of a single point mapping to $\mathcal{M}(d)$.
	
	The corresponding parabolic subgroups of $\mathcal{F}(l_i), \mathcal{F}(p_j)$ are identified with the stabilizers of $\langle v_1\rangle$ and $\langle v_1,v_2\rangle$ respectively, where $v_i$ are isotropic vectors in $\Lambda_{2d,\mathbb{Q}}$.  
	Thus the number of boundary components equals the number of $G_{\mathbb{Z}}$-equivalence classes of rational isotropic subspaces of dimension $1$ or $2$, computed in \cite{scattone-87}.
	
	Recall that $\Lambda_{2d}\coloneqq \lambda^\perp\subset \Lambda_{\mathrm{K3}}$.  
	Then we have an isomorphism
	\[
	G_{2d}/K_{2d}\cong \operatorname{Gr}_2^+(\Lambda_{2d,\mathbb{R}}).
	\]
	Moreover, by the construction of the Satake compactification, elements of the boundary component $\mathcal{F}(l_i)$ correspond to semi-positive $2$-planes in $\Lambda_{2d}$ of rank $1$, while elements of $\mathcal{F}(p_j)$ correspond to $2$-dimensional isotropic subspaces in $\Lambda_{2d}$.  
	In fact, we can prove the following:
	
	\begin{prop}
		$\overline{G_{2d}/K_{2d}}^{\mathrm{ad}}$ can be naturally identified with the closure of $\operatorname{Gr}_2^+(\Lambda_{2d})$ in $\operatorname{Gr}_2(\Lambda_{2d})$.
	\end{prop}
	
	\begin{proof}
		Since the representation of $G_{2d}$ on $\wedge^2(\Lambda_{2d,\mathbb{C}})$ is irreducible with highest weight $e_1+e_2$, we may identify $V_{2d}$ with $\wedge^2 \Lambda_{2d,\mathbb{C}}$.
		
		By definition of the Satake compactification, an element of $\overline{G_{2d}/K_{2d}}^{\mathrm{ad}}$ is of the form $\mathbb{R}\cdot A\in \mathbb{P}(\mathscr{H}(\wedge^2 \Lambda_{2d,\mathbb{C}}))$, where $A$ is a semi-positive Hermitian matrix.  
		Let $\sqrt{A}$ denote the positive square root of $A$, and define a map
		\[
		\mathscr{H}(\wedge^2 \Lambda_{2d,\mathbb{C}})
		\longrightarrow
		\wedge^2 \Lambda_{2d,\mathbb{C}},\qquad
		A\longmapsto \sqrt{A}\, v_0,
		\]
		where $v_0$ is the element corresponding to the $2$-plane
		$\langle w_{e_1}+w_{-e_1},\, w_{e_2}+w_{-e_2}\rangle\subset \Lambda_{2d,\mathbb{C}}$.
		This induces a map
		\[
		\iota:\overline{G_{2d}/K_{2d}}^{\mathrm{ad}}
		\longrightarrow
		\mathbb{P}(\mathscr{H}(\wedge^2 \Lambda_{2d,\mathbb{C}}))
		\longrightarrow
		\mathbb{P}(\wedge^2 \Lambda_{2d,\mathbb{C}}).
		\]
		
		It is straightforward to check that $\iota$ is injective into $\mathbb{P}(\wedge^2 \Lambda_{2d,\mathbb{R}})$, and the restriction of $\iota$ to $G_{2d}/K_{2d}$ coincides with the Plücker embedding
		\[
		G_{2d}/K_{2d}
		\cong \operatorname{Gr}_2^+(\Lambda_{2d,\mathbb{R}})
		\hookrightarrow \operatorname{Gr}_2(\Lambda_{2d,\mathbb{R}})
		\hookrightarrow \mathbb{P}(\wedge^2 \Lambda_{2d,\mathbb{R}}).
		\]
		Hence, since $\iota$ is a continuous injection from a compact space to a Hausdorff space, it is a homeomorphism onto its image, which is identified with the closure of $\operatorname{Gr}_2^+(\Lambda_{2d,\mathbb{R}})$ in $\operatorname{Gr}_2(\Lambda_{2d,\mathbb{R}})$.
	\end{proof}
	
	\begin{rmk}
		If one takes an irreducible representation on $V\cong \Lambda_{\mathrm{K3},\mathbb{C}}$ with highest weight $e_1$ instead of $e_1+e_2$, then the resulting Satake compactification of $\mathcal{M}_{2d}$ is isomorphic to the Baily–Borel compactification $\overline{\mathcal{M}}_{2d}^{BB}$ \cite{baily-borel-66}, which carries the structure of a projective variety.  
		Both $\overline{\mathcal{M}}_{2d}^{\mathrm{ad}}$ and $\overline{\mathcal{M}}_{2d}^{BB}$ are among the minimal Satake compactifications.
		It is also proved in \cite[Thm.~I]{odaka-oshima-21} that the compactification $\overline{\mathcal{M}}_{2d}^{\mathrm{ad}}$ is homeomorphic to the MSBJ compactification.
	\end{rmk}
	\section{Collapsing of K3 surfaces}\label{section-collaping-K3}
	
	Let $(X_k,g_k)$ be a sequence of smooth hyperkähler K3 surfaces converging in the GH sense to a compact metric space $(X_\infty,g_\infty)$.  
	Write $d = \dim(X_\infty)$. We will focus on the cases $d=2$ or $d=3$.  
	In this section, we study the geometric structure on $X_k$ as $k \to \infty$.  
	Denote by $\bm{\omega}_k = (\omega_{1,k},\omega_{2,k},\omega_{3,k})$ the associated hyperkähler triples on $X_k$.
	
	\subsection{Fibration over the regular region}\label{subsection-the-regular-region}
	
	We now investigate the metric properties of collapsing K3 surfaces.  
	Let $S \subset X_\infty$ be the \textit{singular locus} associated with the collapsing sequence, consisting of points $q \in X_\infty$ for which there exists a sequence $q_k \in X_k$ converging to $q$ with $\lvert \mathrm{Rm}(q_k)\rvert \to \infty$.  
	By Cheeger–Tian \cite{cheeger-tian-06}, the set $S$ is finite.
	
	Based on the earlier work of Cheeger–Fukaya–Gromov \cite{cheeger-fucaya-gromov-92}, Sun–Zhang \cite{sun-zhang-24} proved that the metric structure over the regular part of $X_\infty$ is well-approximated by an $\mathcal{N}$-invariant hyperkähler metric.  
	We state their result in the cases $d=2,3$ below. In these cases, $T^{4-d}$ is the unique oriented infranilmanifold of dimension $4-d$.
	
	\begin{thm}\cite[\S 3.5]{sun-zhang-24} \label{thm-sun-zhang-fibration}
		Let $R$ be a compact region in $X_\infty \setminus S$ with smooth boundary.  
		Assume $\dim(X_\infty)=2$ or $3$.  
		Then for $k$ sufficiently large, there exists a compact connected domain $R_k \subset X_k$ with smooth boundary, together with a smooth torus fibration
		\[
		\pi_k^\dagger \colon R_k \to R,
		\]
		with fiber $T^{4-d}$, such that $\pi_k^\dagger$ is a $\tau_k$-GH approximation for $\tau_k \to 0$.  
		Moreover, there is a torsion-free flat connection $\nabla$ on each $\pi_k^\dagger$-fiber, and $\nabla$ induces a smooth fiber-preserving $\mathbb{R}^{4-d}$-action on $(\pi_k^\dagger)^{-1}(U)$ for every simply connected domain $U \subset R$.  
		There is an $\mathbb{R}^{4-d}$-invariant metric $g_k^\dagger$ with hyperkähler triple $\bm{\omega}_k^\dagger$ on $R_k$ in the same cohomology class as $\bm{\omega}_k$, such that for any $l$,
		\[
		\|\bm{\omega}_k^\dagger - \bm{\omega}_k\|_{C^l(\bm{\omega}_k)} \to 0 \quad \text{as } k \to \infty.
		\]
		Furthermore, the quotient metric $g_{R,k}^\dagger$ converges smoothly to $g_\infty$ on $R$.
	\end{thm}
	
	From the proof in \cite{sun-zhang-24}, for any $\varepsilon>0$, we may choose $R$ and $R_k$ such that, for $k$ sufficiently large, the inclusion $i\colon R_k \to X_k$ is an $\varepsilon$–GH approximation.  
	By Remark~\ref{rmk-limit-measure-volume-comparison}, we can also assume $\mathrm{vol}(R_k) > (1-\varepsilon)\mathrm{vol}(X_k)$ for $k$ sufficiently large.
	
	\subsection{Standard model of hyperkähler structures with $\mathbb{R}^{4-d}$-symmetry}
	
	Let $\tilde{R}$ be the universal cover of $R$, and write
	\[
	\tilde{R}_k = \tilde{R} \times T_k^{4-d},
	\qquad
	Y_k = \tilde{R} \times \mathbb{R}^{4-d},
	\]
	for the torus fibration over $\tilde{R}$ and its universal cover, respectively.  
	Hyperkähler structures on $Y_k$ with an $\mathbb{R}^{4-d}$-symmetry can be described explicitly.  
	We summarize below the local models for the cases $d=3$ and $d=2$.  
	See \cite[§3.1--3.2]{sun-zhang-24} and \cite[§2]{gross-wilson-00} for more details.
	
	\subsubsection{Case $d=2$}\label{section-case-n=2}
	
	Models in this case are called \emph{semiflat metrics}, as their restriction to each torus fiber is a flat metric.  
	Consider the hyperkähler metric and hyperkähler triple on $Y = \tilde{R} \times \mathbb{R}^2$ given by
	\begin{align}\label{standard-form-n=2}
		\left\{
		\begin{array}{lr}
			g^\dagger = W^{\alpha\beta}\, dt_{\alpha}\, dt_{\beta} 
			+ W_{\alpha\beta}\, dx_{\alpha}\, dx_{\beta},  
			\qquad 1 \le \alpha,\beta \le 2,\\[6pt]
			\omega^\dagger_{1} = dt_{1} \wedge dx_{1} + dt_{2} \wedge dx_{2}, \\[6pt]
			\omega^\dagger_{2} = W^{2\alpha}\, dt_{\alpha} \wedge dx_{1}
			- W^{1\alpha}\, dt_{\alpha} \wedge dx_{2}, \\[6pt]
			\omega^\dagger_{3} = dt_{1} \wedge dt_{2} - dx_{1} \wedge dx_{2}.
		\end{array}
		\right.
	\end{align}
	
	Here $x_{1},x_{2} \in C^\infty(\tilde{R})$ and $t_1,t_2 \in C^\infty(\mathbb{R}^2)$ serve as local coordinates, and $W_{\alpha\beta} = W_{\beta\alpha}$ is a function of $\{x_\alpha\}$ satisfying
	\[
	\det(W_{\alpha\beta}) \equiv 1,
	\qquad 
	\partial_{x_\gamma} W_{\alpha\beta}
	= \partial_{x_\alpha} W_{\gamma\beta},
	\qquad 1 \le \alpha,\beta,\gamma \le 2.
	\]
	By an $SL_2(R)$ action
	\[
	(x_1,x_2) \mapsto (x_1,x_2)A, \qquad
	(t_1,t_2) \mapsto (t_1,t_2)(A^T)^{-1},
	\]
	we may assume that $(W_{\alpha\beta})(p)=\operatorname{Id}$ at a base point $p$.
	\begin{prop}
		Given an $\mathbb{R}^{4-d}$-invariant hyperkähler metric on $Y = \tilde{R}^2 \times \mathbb{R}^2$, there exist functions $x_\alpha, t_\alpha$ and $W_{\alpha\beta}$ as above such that the hyperkähler triple is given by \eqref{standard-form-n=2}.
	\end{prop}
	
	We include a proof of this proposition for completeness. The argument is essentially given in \cite[\S2.2]{apostolov2022k}.
	
	\begin{proof}
		Let $(g,J,\omega)$ be the invariant hyperkähler structure, and let $\xi_1, \xi_2$ be the vector fields generating the $\mathbb{R}^2$-action. Then $\mathcal{L}_{\xi_i}\xi_j = \mathcal{L}_{\xi_i}(J\xi_j) = 0$, and $\mathcal{L}_{J\xi_i}(J\xi_j)=0$ by the vanishing of the Nijenhuis tensor.  
		After a hyperkähler rotation, we may assume $\omega(\xi_1,\xi_2)=0$, hence $g(\xi_i,J\xi_j)=0$.  
		Thus $(\xi_1,\xi_2,J\xi_1,J\xi_2)$ is a frame of the tangent bundle, and the dual $1$-forms have the form $(\theta_1,\theta_2,-J\theta_1,-J\theta_2)$.
		
		A direct computation shows $d\theta_i = d(J\theta_i)=0$, hence $\theta_i = dt_i$.  
		Choose smooth functions $x_i$ such that $dx_i = \iota_{\xi_i}\omega$.  
		Then $dx_i = W_{ij} J\theta_j$ for $W_{ij} = g(\xi_i,\xi_j)$, so $J\theta_i = W^{ij} dx_j$.  
		We obtain
		\[
		g = W_{ij}(\theta_i\theta_j + J\theta_i J\theta_j)
		= W_{ij} dt_i dt_j + W^{ij} dx_i dx_j,
		\]
		and the forms $\omega_i$ are determined by $g$.  
		Finally, $\partial_{x_\gamma} W_{\alpha\beta} = \partial_{x_\alpha} W_{\gamma\beta}$ follows from $d(W_{ij}J\theta_j)=0$, and $\det(W_{\alpha\beta}) \equiv 1$ follows from $\omega^2 = \mathrm{Re}(\Omega)^2 = \mathrm{Im}(\Omega)^2$.
	\end{proof}
	Thus, $\tilde{R}_k$ is obtained as the quotient of $Y_k$ by
	\[
	(t_{1,k},t_{2,k}) \sim (t_{1,k},t_{2,k}) + m\,\bm{u}_{1,k} + n\,\bm{u}_{2,k},
	\qquad m,n \in \mathbb{Z},
	\]
	for some $\bm{u}_{\alpha,k} \in \mathbb{R}^2$ with $\bm{u}_{\alpha,k} \to 0$.
	The quotient metric on $\tilde{R}$ is given by  
	\[
	g^\dagger_{\tilde{R},k} \coloneq W_{\alpha\beta,k}\, dx_{\alpha,k}\, dx_{\beta,k}.
	\]
	
	\subsubsection{Case $d=3$}\label{section-case-n=1}
	
	In this case, the metric is given by the Gibbons–Hawking ansatz.  
	Suppose $g^\dagger$ is an $\mathbb{R}$-invariant hyperkähler metric on $Y=\tilde{R}^3\times \mathbb{R}$. Then there exist functions $x,y,z,t$ on $Y$ serving as local coordinates (\cite[\S2]{hein-sun-Viaclovsky-Zhang-22}) such that the hyperkähler triple takes the form
	\begin{align}\label{standard-form-n=1}
		\left\{
		\begin{array}{lr}
			g^\dagger = V(dx^2 + dy^2 + dz^2) + \dfrac{1}{V} \theta^2,\\[6pt]
			\omega^\dagger_{1} = V\, dx \wedge dy + dz \wedge \theta, \\[6pt]
			\omega^\dagger_{2} = V\, dy \wedge dz + dx \wedge \theta, \\[6pt]
			\omega^\dagger_{3} = V\, dz \wedge dx + dy \wedge \theta.
		\end{array}
		\right.
	\end{align}
	Here $V$ is a harmonic function with respect to $g^\flat = dx^2 + dy^2 + dz^2$, and there exists a $1$-form $\alpha$ on the base $\tilde{R}$ such that
	\[
	\theta = dt + (\pi^{\dagger})^*\alpha,
	\qquad
	d\alpha = *\, dV,
	\]
	where $*$ is the Hodge star operator defined by the volume form $dx \wedge dy \wedge dz$.
  By the rescaling $(x,y,z,t)\to (cx,cy,cz,t/c)$, we may assume that $V(p)=1$ for a base point $p$.
	
	Thus, $\tilde{R}_k$ is obtained as the quotient of $Y_k$ by the translation $t \mapsto t + \varepsilon_k$ for some $\varepsilon_k \to 0$.  
		The quotient metric on $\tilde{R}$ is given by  
	\[
	g^\dagger_{\tilde{R},k} \coloneq V_k(dx_k^2 + dy_k^2 + dz_k^2).
	\]
	
	\subsubsection{Convergence of $\mathbb{R}^{4-d}$-invariant metrics}\label{subsection-convergence-Rn-inv-metric}
	
	In the case $d=2$, there is a special Kähler structure \cite{freed-99} on $(\tilde{R}, g^\dagger_{\tilde{R}})$, and the metric on $Y_k$ is identified with the canonical metric on the cotangent bundle of $\tilde{R}$.  
	Thus the metric is determined by a pair of special Kähler coordinates $(z,w)$, which are holomorphic functions on $(\tilde{R}, g^\dagger_{\tilde{R}})$, and we have
	$
	(\mathrm{Re}(z), \mathrm{Re}(w)) = (x_1, x_2).
	$
	In the case $d=3$, the $\mathbb{R}^{4-d}$-invariant metric is determined by the harmonic function $V$.  
	
	In both cases, we have good convergence properties for $\mathbb{R}^{4-d}$-invariant metrics by the regularity theory for harmonic functions.  
	The following is proved in \cite[§3.1.2 and §3.2.2]{sun-zhang-24}.
	\begin{prop}
		Let $g^\dagger_{\tilde{R},k}$ be a sequence of $\mathbb{R}^{4-d}$-invariant hyperkähler metrics on $\tilde{R}\times \mathbb{R}^{4-d}$, such that the quotient metrics $g^\dagger_{\tilde{R},k}$ converge smoothly to $g_\infty$ as in Theorem \ref{thm-sun-zhang-fibration}.
		
		In the case $d=2$, we may choose $(x_{1,k}, x_{2,k}, t_{1,k}, t_{2,k}, W_{\alpha\beta,k})$ for the metric $g_k^\dagger$ as in \eqref{standard-form-n=2}, such that, after passing to a subsequence, $x_{\alpha,k}$ and $W_{\alpha\beta,k}$ converge smoothly to $x_{\alpha,\infty}$ and $W_{\alpha\beta,\infty}$ on $\tilde{R}$, and
		\[
		g_\infty = W_{\alpha\beta,\infty} \, dx_{\alpha,\infty} \, dx_{\beta,\infty},
		\qquad 1 \le \alpha, \beta \le 2.
		\]
		
		In the case $d=3$, we may choose $(V_k, \theta_k, dx_k, dy_k, dz_k)$ for the metric $g_k^\dagger$ as in \eqref{standard-form-n=1}, such that, after passing to a subsequence, $(V_k, dx_k, dy_k, dz_k)$ converge smoothly to $(V_\infty, dx_\infty, dy_\infty, dz_\infty)$ on $\tilde{R}$, and
		\[
		g_\infty = V_\infty \,(dx_{\infty}^2 + dy_{\infty}^2 + dz_{\infty}^2).
		\]
	\end{prop}
	
	In the case $d=2$, the limit $g_\infty$ also comes from a special Kähler structure $(dz,dw)$, and the special Kähler structure is determined by the orientation and the metric up to a twist \cite[Prop.~1]{ouyang-25}. In particular, if $X_\infty = T^2/\{\pm 1\}$ and $g_\infty$ is flat, then $dx_{\alpha,\infty}$ must also be flat, and we may assume that $W_{\alpha\beta,\infty} = \delta_\alpha^\beta$.
	
	In the case $d=3$, we can calculate the scalar curvature of the base metric
	$
	\mathrm{scal}(g^\dagger_{\tilde{R}}) = \frac{3|\nabla V|^2}{2V^3},
	$
	where $\nabla$ and $|\cdot|$ are taken with respect to $g^\flat=dx_{\infty}^2 + dy_{\infty}^2 + dz_{\infty}^2$.
	Thus, since $g_\infty$ is flat for GH limits of hyperkähler K3 surfaces (Theorem \ref{thm-classification-GH-limit}), we may assume that $V_\infty \equiv 1$ in our case.
	
	\subsection{Perturbations of almost holomorphic tori}

Consider the standard hyperkähler structure $\bm{\omega}=(\omega_1,\omega_2,\omega_3)$ on $B_r^2\times T^2$ given by
\begin{align}\label{standard-form-n=2-constant-efficient}
	\left\{
	\begin{array}{lr}
		\omega_{1} = dt_{1} \wedge dx_{1} + dt_{2} \wedge dx_{2}, \\[2mm]
		\omega_{2} = dt_{2} \wedge dx_{1} - dt_{1} \wedge dx_{2}, \\[1mm]
		\omega_{3} = dt_{1} \wedge dt_{2} - dx_{1} \wedge dx_{2},
	\end{array}
	\right.
\end{align}	
where $B_r^2 = \{ (x_1,x_2)\in \mathbb{R}^2 : x_1^2+x_2^2<r^2 \}$, and $T^2 = \mathbb{R}^2 / \Lambda$ is a torus.  
Fix a torus fiber $T^2_{y}$ over $y\in B_r^2$.  
We can identify $B_r^2 \times T^2_y$ with a subset of $T^* T^2_y$.  
A small deformation of $T^2_y$ can then be represented as a graph
\[
L(y,\sigma) = \mathrm{Graph}\bigl( x \mapsto y + \sigma(x) \bigr) \subset T^* T^2_y,
\]
where $\sigma$ is a $1$-form on $T^2_y$.  

The following result is due to Zhang \cite{zhang-yuguang-17}, whose proof is a standard application of the implicit function theorem.

\begin{thm}\label{thm-zhang}		
	Let $B_r^2\times T^2$ and $\bm{\omega}$ be as above.  
	Given $\varepsilon>0$, there exists $\delta = \delta(r,C)>0$ such that the following holds.  
	
	Suppose $\operatorname{diam}(T^2)<C$, and $\bm{\omega}'$ is another hyperkähler triple on $B_r^2\times T^2$ with
	\[
	[\bm{\omega}'] = [\bm{\omega}] \in H^2(B_r^2\times T^2,\mathbb{R})^{\oplus 3}, 
	\qquad
	\|\bm{\omega}'-\bm{\omega}\|_{C^{1,\alpha}(\omega)} < \delta.
	\]
	Then, for any $y \in B_{2r/3}^2$, there exists a unique
	$
	\sigma \in C^{1,\alpha}\bigl(d\Omega^0(T^2) \oplus d^*\Omega^2(T^2)\bigr)
	$
	with $\|\sigma\|_{C^{1,\alpha}} < \varepsilon$, such that the graph $L(y,\sigma)$ satisfies
	\[
	\omega'_1|_{L(y,\sigma)} = \omega'_2|_{L(y,\sigma)} = 0,
	\]
	i.e., $L(y,\sigma)$ is a holomorphic torus in the complex structure determined by $\Omega = \omega_1 + i \omega_2$.  
	
	Moreover, the collection of such $L(y,\sigma)$ forms a fibration over an open set $W \supset B_{r/3}^2\times T^2$.
\end{thm}

\begin{proof}
	Fix $B_r^2$.  
	Theorem~\ref{thm-zhang} is proved in \cite[\S4]{zhang-yuguang-17}, where $\delta$ depends continuously on $T^2$. We show that this dependence can be weakened to depend only on an upper bound of $\operatorname{diam}(T^2)$.  
	
	Recall that the standard fundamental domain for the action of $\mathrm{SL}(2,\mathbb{Z})$ on the upper half-plane $\mathbb{H}$ is
	\begin{equation}\label{definition-fundamental-domain-elliptic}
		\mathcal{D} = \Bigl\{ z \in \mathbb{H} \;\Big|\; |z| \ge 1, \; -\tfrac{1}{2} \le \mathrm{Re}(z) \le \tfrac{1}{2} \Bigr\}.
	\end{equation}
	
	Suppose $\operatorname{diam}(T^2) < C$. Then we can take a finite cover $\tilde{T}^2$ of $T^2$ with covering group $\Gamma$, such that $C \le \operatorname{diam}(\tilde{T}^2) \le 10 C$.  
	The torus $\tilde{T}^2$ corresponds to a point in a given compact subset of $\mathcal{D}$.  
	For such $\tilde{T}^2$, there exists a $\delta$ depending only on $C$ such that the conclusion holds.  
	Let $\tilde{L}$ be the graph in $B_r^2\times\tilde{T}^2$.
	By the uniqueness part of Theorem~\ref{thm-zhang}, the holomorphic graph $\tilde{L}$ is $\Gamma$-invariant and therefore descends to a holomorphic fiber in $B_r^2 \times T^2$.
\end{proof}

\begin{rmk}
	In the original formulation \cite{zhang-yuguang-17}, it is assumed that $r \gg 0$, but the same argument applies for any fixed $r>0$.  
	The higher regularity in Theorem~\ref{thm-zhang} follows from standard elliptic theory; the ${C^{1,\alpha}}$ estimate here is, however, sufficient for our purposes.
\end{rmk}
	\subsection{Elliptic fibration approximation}
	
	We now consider the case $d=2$. This section aims to prove the following theorem, which shows that 2-dimensional collapsing of hyperkähler K3 surfaces is characterized by elliptic K3 fibrations.
	
	\begin{thm}\label{thm-2dimensional-collapsing-elliptic-approximation}
		Let $(X_k, g_k)$ be a collapsing sequence of unit-diameter hyperkähler K3 surfaces converging in the GH sense to $(X_\infty, g_\infty)$ with $\dim(X_\infty) = 2$. Then, for $k \gg 1$, after a suitable hyperkähler rotation, there exist elliptic fibrations 
		$
		\pi_k : X_k \to \mathbb{P}^1_k,
		$
		such that if $h_k$ is the unit-diameter generalized KE metric on $\mathbb{P}^1_k$ induced by $\pi_k$, we have
		\[
		d_{GH}\bigl((\mathbb{P}^1_k, d_{h_k}), (X_\infty, g_\infty)\bigr) \to 0.
		\]
		Moreover, let $\lambda_k \in \mathbb{R}$ be such that $\lambda_k \bm{\omega_k}$ has unit volume. Then for any elliptic fiber $T_k^2$ of $\pi_k$, we have
		\[
		\int_{T_k^2} \lambda_k \bm{\omega_k} \to 0.
		\]
	\end{thm}
	
	As a corollary of Theorem~\ref{thm-2dimensional-collapsing-elliptic-approximation}, we obtain an alternative proof of the classification of 2-dimensional GH limits of hyperkähler K3 surfaces.  
	Indeed, if we define $\mathfrak{M}^2$ to be the set of all unit-diameter generalized KE metrics on $\mathbb{P}^1$ arising from elliptic K3 surfaces, then $\dim(X_\infty, g_\infty) = 2$ implies $(X_\infty, g_\infty) \in \overline{\mathfrak{M}^2}$ by Theorem~\ref{thm-2dimensional-collapsing-elliptic-approximation}.  
	It then follows from Theorem~\ref{thm-geometric-realization-GIT} that $(X_\infty, g_\infty) \in \mathfrak{M}^2$.
	
	\subsubsection{Construction of the elliptic fibration}
	
	Let $\bm{\omega}^\dagger_k$ be the approximating hyperkähler triple on $R_k \subset X_k$ given by Theorem~\ref{thm-sun-zhang-fibration}.  
	We can write it in the form \eqref{standard-form-n=2}, with the quotient metric on $R \subset X_\infty$ converging smoothly.  
	
	Let $p$ be an interior point of $R$, and let $\varepsilon_k$ be the diameter of the $T^2$ fiber over $p$ with respect to $\pi^\dagger_k$. Then $\varepsilon_k \to 0$, and there exists a diffeomorphism
	\[
	f_k : B_{\varepsilon_k r}(p, g_\infty) \times T^2 
	\;\longrightarrow\; 
	B_r(\mathbb{R}^2) \times T^2_0,
	\]
	such that $\tfrac{1}{\varepsilon_k} \bm{\omega}^\dagger_k$ is $C^2$-close to the standard hyperkähler triple $\bm{\omega}$ on $B_r(\mathbb{R}^2) \times T^2_0$ given by \eqref{standard-form-n=2-constant-efficient} after a hyperkähler rotation.  
	Hence, $\tfrac{1}{\varepsilon_k} \bm{\omega_k}$ is also $C^2$-close to the standard triple.
	
	By Theorem~\ref{thm-zhang}, there exists $k_0$ such that for $k \ge k_0$, there is a holomorphic fiber passing through every point of $T^2$.  
	By Proposition~\ref{prop-elliptic-curve-elliptic-K3}, we then obtain an elliptic K3 fibration
	\[
	\pi_k : X_k \to \mathbb{P}^1_k.
	\]
	
	Moreover, if $p$ lies in a compact set $R' \subset R^\circ$, then $k_0$ can be chosen independently of $p$.  
	Thus, for $k \ge k_0$, for each $y \in R'_k := (\pi_k^\dagger)^{-1}(R')$, there exists a holomorphic torus passing through $y$.  
	All these holomorphic tori are in the same homology class and correspond to the same elliptic fibration $\pi_k : X_k \to \mathbb{P}^1_k$ by Proposition~\ref{prop-elliptic-curve-elliptic-K3}.  
	It is clear that $\int_{T^2_k} \lambda_k \bm{\omega_k} \to 0$ after rescaling to unit volume.
	
	\subsubsection{Local convergence of generalized KE metric}
	
	Suppose $\pi_k$ is holomorphic in the complex structure determined by $\Omega = \omega_{1,k} + i\omega_{2,k}$, and let
	$
	\mu_k = \int_{T^2_k} \omega_{3,k}.
	$
	Define $h_k'$ to be the generalized KE metric on $\mathbb{P}^1_k$ associated with the hyperkähler triple $\tfrac{1}{\sqrt{\mu_k}} \bm{\omega}_k$ as in \eqref{formula-defi-gen-KE}.  
	This scaling ensures that if $\pi_k = \pi^\dagger_k$, then $h_k' = g^\dagger_{\tilde{R},k}$ is the quotient metric.
	
	Let $U \subset R^\circ$ be a simply connected open set with compact closure.  
	For each $k$, choose a section $i_k: U \to R_k$ orthogonal to the fibers with respect to $g_k^\dagger$.  
	For $k$ sufficiently large, the perturbed $T^2$ fibers intersect $U$ transversely, so
	\begin{equation}\label{equ-defi-f}
		f = \pi_k \circ i_k : U \to \mathbb{P}^1_k
	\end{equation}
	is a homeomorphism onto its image.  
	This defines a generalized KE metric on $U$ via the pullback $f^* h_k'$.
	
	\begin{defi}\label{defi-distance-function-manifold}
		Let $(M,g)$ be a Riemannian manifold. Define $d_{M,g}$ to be the distance function on $M$ determined by $g$, i.e.,
		\[
		d_{M,g}(x,x') = \inf\{\text{length}(\gamma) \mid \gamma \text{ is a path in } (M,g) \text{ connecting } x \text{ and } x'\}.
		\]
	\end{defi}
	
	\begin{lem}\label{prop-gen-KE-converge-local-U}
		Let $f^* h_k'$ be the pullback metric on $U$ as above. Then
		\[
		\| f^* h_k' - g_\infty \|_{C^{0,\alpha}(U)} \longrightarrow 0
		\quad \text{as } k \to \infty.
		\]
		
	\end{lem}
	
	\begin{proof}
		Since the quotient metric $g_{U,k}^\dagger$ converges smoothly to $g_\infty$, it suffices to prove
		\[
		\| f^* h_k' - g_{U,k}^\dagger \|_{C^{0,\alpha}(U)} \longrightarrow 0.
		\]
		
		Consider the two holomorphic torus fibrations
		\[
		\pi_k: X_{U,k} \to U, \qquad \pi_k^\dagger: X_{U,k}^\dagger \to U,
		\]
		associated with the hyperkähler triples $(\bm{\omega}_k, J_k)$ and $(\bm{\omega}_k^\dagger, J_k^\dagger)$, respectively, where $X_{U,k}, X_{U,k}^\dagger \subset R'_k$.  
		By Theorem~\ref{thm-zhang}, the fibrations $\pi_k$ and $\pi_k^\dagger$ are $C^{1,\alpha}$–close.  
		
		Let $\omega_{U,k}$ and $\omega_{U,k}^\dagger$ denote the Kähler forms on $U$ for the generalized KE metrics associated with $\tfrac{1}{\sqrt{\mu_k}} \bm{\omega}_k$ and $\tfrac{1}{\sqrt{\mu_k}} \bm{\omega}^\dagger_k$. Then
		\[
		\| \omega_{U,k}^\dagger - \omega_{U,k} \|_{C^{1,\alpha}} \longrightarrow 0.
		\]	
		Let $J_{U,k}$ and $J_{U,k}^\dagger$ be the complex structures on $U$ induced by $J_k$ and $J_k^\dagger$.  
		Since $\pi_k^\dagger$ is holomorphic,
		\[
		J_{U,k}^\dagger(X) = J_k^\dagger(\pi^\dagger_{k*} i_{k*} X) = \pi^\dagger_{k*}(J_k^\dagger(i_{k*} X)), \qquad \forall X \in TU,
		\]
		and similarly $J_{U,k}(X) = \pi_{k*}(J_k(i_{k*} X))$.  
		Since $\| J_k^\dagger - J_k \|_{C^l} \to 0$ for any $l$, it follows that
		\[
		\| J_{U,k}^\dagger - J_{U,k} \|_{C^{0,\alpha}} \longrightarrow 0.
		\]
		
		The metric $f^* h_k'$ is determined by $(\omega_{U,k}, J_{U,k})$, while $g_{U,k}^\dagger$ is determined by $(\omega_{U,k}^\dagger, J_{U,k}^\dagger)$.  
		Hence the above estimates imply
		\[
		\| f^* h_k' - g_{U,k}^\dagger \|_{C^{0,\alpha}} \to 0,
		\]
		and the conclusion follows.
	\end{proof}
	
	Moreover, from Theorem~\ref{thm-zhang} we know that, for any $\delta > 0$ and relatively compact subset $U \Subset U'$, if $k$ is sufficiently large, then for any two sections $i,i': U \to R'_k$ defining maps $f,f'$ as in \eqref{equ-defi-f}, we have
	\[
	f(U) \subset f'(U'), 
	\quad
	d_{f(U'),h_k'}\bigl(f(x), f'(x)\bigr) \le \delta, \qquad \forall x \in U.
	\]
	\subsubsection{Gluing local pieces}
	Next, we estimate the generalized KE metric on $\mathbb{P}^1_k$.  
	Note that there is no canonical choice of a section over $R$, so we cannot identify $R$ directly with a subset of $\mathbb{P}^1_k$.  
	This difficulty is resolved by gluing local pieces together as follows.
	
	\begin{lem}\label{prop-M-N-GH-close}
		Let $M,N$ be Riemannian manifolds.
		Suppose $M = \bigcup_{i=1}^n W_i$ for some open sets $W_i$, and for each $i$ there exists a map $F_i:W_i \to N$.  
		Assume that:
		\begin{enumerate}
			\item[(i)] For any $x \in W_i \cap W_j$, we have $d_N\bigl(F_i(x), F_j(x)\bigr) < \delta$.
			\item[(ii)] For any $x,x' \in W_i$, we have $d_N\bigl(F_i(x), F_i(x')\bigr) < d_M(x,x') + \delta$.
		\end{enumerate}
		Then for any $x\in W_i,\; x'\in W_j$, we have
		\[
		d_N\bigl(F_i(x), F_j(x')\bigr) < d_M(x,x') + (2n+1)\delta.
		\]
	\end{lem}
	
	\begin{proof}
		Let $\gamma:[0,1]\to M$ be a path from $x$ to $x'$ of length $L(\gamma)< d_M(x,x')+\delta$.  
		We can take $x_k = \gamma(t_k)$, where
		\[
		0=t_0 < t_1 < \cdots < t_m = 1, \qquad m \le n,
		\]
		such that for each $k=0,\dots,m-1$, there is an $i_k$ with $x_k, x_{k+1} \in W_{i_k}$, and $i_0=i$, $i_m=j$.  
		Then
		\[
		d_N\bigl(F_i(x),F_j(x')\bigr)
		< \sum_{k=0}^{m-1} d_N\bigl(F_{i_k}(x_k), F_{i_k}(x_{k+1})\bigr) 
		+ \sum_{k=0}^{m-1} d_N\bigl(F_{i_k}(x_{k+1}), F_{i_{k+1}}(x_{k+1})\bigr).
		\]
		By (ii), the first sum is at most $m\delta + L(\gamma)$, and by (i) the second sum is at most $m\delta$.  
		Thus the inequality follows.
	\end{proof}
	
	Given $r>0$, we take
	\begin{equation}\label{equ-defi-R}
		R \coloneqq X_\infty \setminus \bigcup_{q_i \in S} B_r(q_i,g_\infty).
	\end{equation}
	By choosing $r$ small enough, we can assume $B_r(q_i)$ are disjoint, and we take $R'\subset R$ of the same form with $r<r'\ll 1$.  
	We may choose them such that $R'$ is $\varepsilon$-dense in $X_\infty$, and for any $x,y\in R'$,
	\begin{equation}
		\label{inequality-d-small-large}
		d_{X_\infty,g_\infty}(x,y)\le d_{R',g_\infty}(x,y)
		< d_{X_\infty,g_\infty}(x,y)+\varepsilon \le d_{R,g_\infty}(x,y)+\varepsilon. 
	\end{equation}
	Thus $(R',d_{R,g_\infty})\to (X_\infty,d_\infty)$ is an $\varepsilon$-GH approximation, and hence they are $2\varepsilon$-GH close. 
	
	\begin{lem}\label{lem-key-estimate}
		For any $\varepsilon>0$, we may choose $R' \subset R$ such that for $k$ sufficiently large, for any $x, x'\in R'$ and $f=\pi_k\circ i_x$, $f'=\pi_k\circ i_{x'}$ as in \eqref{equ-defi-f}, we have 
		\[
		\bigl|d_{\mathbb{P}^1_k,h_k'}(f(x),f'(x')) - d_{R,g_\infty}(x,x')\bigr| < 3\varepsilon,
		\qquad \forall x, x'\in R'.
		\]
	\end{lem}
	
	\begin{proof}
		The idea is to use Lemma~\ref{prop-M-N-GH-close}).
		First, take $N=\mathbb{P}^1_k$ and $M=\bigcup_{i=1}^n U_i$ with $R'\subset M\subset R^\circ$, where $U_i$ are simply connected and
		\[
		f_i = \pi_k \circ i_{U_i}:U_i\to V_i \subset \mathbb{P}^1_k, \qquad 1\le i \le n,
		\]
		are homeomorphisms as in \eqref{equ-defi-f}.		
		We may choose $U_i$ such that there is a simply connected compact domain $U_i'$ with 
		$U_i \subset U_i' \subset R^\circ$ and for any $x,x'\in U_i$, the shortest path connecting $x$ and $x'$ in $R$ with respect to the metric $g_\infty$ lies in $U_i'$ .		
		Then, by Lemma~\ref{prop-gen-KE-converge-local-U}, assumptions (1)–(2) of Lemma~\ref{prop-M-N-GH-close} hold for any $\delta>0$ if $k\gg1$. 	
		Thus for $k$ large, we have
		\begin{align*}
			&d_{\mathbb{P}^1_k,h_k'}(f(x),f'(x')) \\
			\;\text{(Lemma \ref{prop-gen-KE-converge-local-U})} < &d_{\mathbb{P}^1_k,h_k'}(f_i(x),f_j(x')) + \varepsilon \\
			\;\text{(Lemma~\ref{prop-M-N-GH-close})} < &d_{M,g_\infty}(x,x') + 2\varepsilon
			\le d_{R',g_\infty}(x,x') + 2\varepsilon \\
			\;\text{\eqref{inequality-d-small-large}} < &d_{R,g_\infty}(x,x') + 3\varepsilon.
		\end{align*}
		
		On the other hand, take $R'\subsetneq R''\subsetneq R$ of the form \eqref{equ-defi-R}. Let $N=R$ and $M=\bigcup_{i=1}^m V_i \supset \pi_k(R_k'')$ with $f_i: U_i\to V_i$ as before.
		Then, by the same argument, for $k$ large, Lemma~\ref{prop-M-N-GH-close} and Lemma~\ref{prop-gen-KE-converge-local-U} yield
		\begin{equation*}
			d_{R,g_\infty}(x,x')
			< d_{\pi_k(R_k''),\,h_k'}(f_i(x),f_j(x')) + \varepsilon
			< d_{\pi_k(R_k''),\,h_k'}(f(x),f'(x')) + 2\varepsilon.
		\end{equation*}
		
		It remains to prove that for $k$ large enough,
		\begin{equation}\label{inequality-path-}
			d_{\pi_k(R_k''),h_k'}(y,y') < d_{\mathbb{P}^1_k,h_k'}(y,y') + \varepsilon,
			\qquad \forall\, y,y'\in \pi_k(R_k').
		\end{equation}		
		
		Let $C'_i$ be the boundary circles of $R'$, and $E'_i$ the corresponding boundary components of $R_k'$, which are torus bundles over $C'_i$. 		
		By choosing $r$ small enough, we may assume that for each $x\in X_k\setminus R_k'$, there is a unique boundary component $E'_i$ such that $d(E'_i,x) < \varepsilon$. Thus we may write $X_k = R_k' \bigcup_{i=1}^{|S|} S_i'$, where $S'_i$ are compact connected domains with boundary $E'_i$.		
		We may further assume $\operatorname{diam}(C'_i,M) < \frac{\varepsilon}{2|S|}$.
		
		Now, taking $M = \cup_{i=1}^l U_i \supset C'_i$ and $\pi_k(E'_i) \subset N \subset \pi_k(R_k'')$ in Lemma~\ref{prop-M-N-GH-close}, we have
		\[
		\operatorname{diam}(\pi_k(E'_i), d_{\pi_k(R_k''),h_k'}) < \operatorname{diam}(C'_i,d_{M,g_\infty}) + \frac{\varepsilon}{2|S|} < \frac{\varepsilon}{|S|} \qquad \text{as } k\to \infty.
		\]
		
		Let $\gamma:[0,s]\to \mathbb{P}^1$ be a minimizing geodesic between $y,y'\in \pi_k(R_k')$.  
		We can lift it to a path $\tilde{\gamma}:[0,s]\to X_k$ such that $\pi_k(\tilde{\gamma}(t))=\gamma(t)$.  
		For each $i$, define
		\[
		m_i = \min \{ t \in [0,s] : \tilde{\gamma}(t) \in S_i' \}, 
		\qquad 
		M_i = \max \{ t \in [0,s] : \tilde{\gamma}(t) \in S_i' \}.
		\]
		Then $\gamma(m_i),\gamma(M_i) \in \pi_k(E'_i)$.  
		We may replace each segment $\gamma|_{[m_i,M_i]}$ by an arc in $\pi_k(R''_k)$ of length at most $\frac{\varepsilon}{|S|}$.  
		This procedure prolongs the length by at most $\varepsilon$, so \eqref{inequality-path-} holds.
	\end{proof}
	\subsubsection{Estimate near singularities}\label{subsection-estimate-singularities}
	
	Note that the limit measure on $X_\infty$ is given by the Hausdorff measure on $(X_\infty, g_\infty)$ by Remark~\ref{rmk-limit-measure-volume-comparison}. For any $\delta>0$, 
	we may assume that for $k\gg 1$,
	\[
	\frac{\operatorname{Vol}(R'_k)}{\operatorname{Vol}(X_k)} > 1 - \delta.
	\]
	Thus, by the definition of the generalized KE metric, we have
	\[
	\frac{\operatorname{vol}(\pi_k(R_k'))}{\operatorname{vol}(\mathbb{P}^1_k)} \ge 
	\frac{\operatorname{vol}(R_k')}{\operatorname{vol}(X_k)} \ge 1 - \delta
	\]
	for $k\gg1$.
	By a volume comparison argument for the generalized KE metric (Proposition~\ref{prop-volume-comparison}(2)), there exists $\delta$ such that $\pi_k(R_k')$ is $\varepsilon$-dense in $\mathbb{P}^1_k$.
	
	We can define $F:(R',d_{R,g_\infty})\to (\mathbb{P}^1_k,d_{h_k'})$ such that $F(x)=f_i(x)$ for some $f_i$ given by \eqref{equ-defi-f}.
	By Lemma~\ref{prop-gen-KE-converge-local-U}, the image of $F$ is $2\varepsilon$-dense in $\pi_k(R_k')$, hence $3\varepsilon$-dense in $\mathbb{P}^1_k$. 
	Therefore, $F$ is a $3\varepsilon$-GH approximation by Lemma~\ref{lem-key-estimate}.
	Since we can choose $(R',d_{R,g_\infty})$ to be $\varepsilon$-GH close to $(X_\infty,g_\infty)$, this completes the proof of Theorem~\ref{thm-2dimensional-collapsing-elliptic-approximation}.
	
	\section{Period behaviour of collapsing K3 surfaces}\label{section-period-collapsing}

\subsection{Collapsing of codimension two}

Recall that $\overline{\Phi}$ and $\overline{\Phi}'$ are defined in \S\ref{section-construction-geometric-realization}.	
This subsection aims to prove the following.

\begin{thm}\label{thm-2dimensional-collapsing-period-behaviour}
	Suppose $\eta_k \in \mathcal{M}_{\mathrm{K3}}$ with
	$\eta_k \to \eta_\infty \in \overline{\mathcal{M}}_{\mathrm{K3}}^{\mathrm{Sat,ad}}$
	and
	$\Phi(\eta_k) \to (X_\infty,g_\infty)$ in the GH sense.
	If $\dim X_\infty = 2$, then $\overline{\Phi}'(\eta_\infty) = (X_\infty,g_\infty).$\end{thm}
Since the smooth locus is dense in $\mathcal{M}_{\mathrm{K3}}$, we may assume that $(X_k,g_k)$ is smooth.
Let $\pi_k:X_k\to \mathbb{P}^1_k$ be the elliptic fibration given by Theorem~\ref{thm-2dimensional-collapsing-elliptic-approximation} for $k\gg 1$.  
Let $\pi_{J_k}: J_k\to \mathbb{P}^1_k$ be the corresponding Jacobian fibration, which is represented by $\xi_k\in \mathcal{M}_{Jac}$ (cf. \S\ref{subsection-git-quotient-construction}). After passing to a subsequence, we may assume that $\xi_k\to \xi_\infty\in \overline{\mathcal{M}_W}$, with $\overline{\Phi}_W(\xi_\infty)=(X_\infty,g_\infty)$ by Theorems~\ref{thm-2dimensional-collapsing-elliptic-approximation} and \ref{thm-geometric-realization-GIT}. We divide the proof into two cases depending on the position of $\xi_\infty$.

\subsubsection{Case 1: $\xi_\infty\in \mathcal{M}_{\mathrm{Jac}}$.} 
Then $\xi_\infty$ represents an elliptic K3 fibration $\pi_\infty:J_\infty\to \mathbb{P}^1_\infty$, and $\pi_{J_k}: J_k\to \mathbb{P}^1_k$ is a small deformation of it.  
By standard deformation theory \cite[Chap.~IV, \S4]{barth-cptcplxsurface-04} (cf.~\cite{tian-87-deformation}), we can identify $J_k$ with $J_\infty$, and we can choose holomorphic volume forms $\Omega_{J_k}$ such that they converge to a holomorphic volume form $\Omega_{J_\infty}$ on $J_\infty$.

We can take a standard basis (Definition~\ref{defi-std-basis}) $\{e_i\}$ of $H^2(X_k,\mathbb{R})\cong H^2(J_k,\mathbb{R})\cong H^2(J_\infty,\mathbb{R})$,
such that $e_1$ is the Poincaré dual of the fiber class.  
Writing $\Omega_{J_k} = \omega_{J_k,1} + i \omega_{J_k,2}$, by Proposition~\ref{prop-e,f-give-elliptic}, we have $[\omega_{J_k,1}], [\omega_{J_k,2}] \in \langle e_1,e_{22}\rangle^\perp$, and by Proposition~\ref{jacobian-characterize}, there exists $a_k,b_k\in\mathbb{R}$ such that $\eta_k$ is represented by
\[
\langle [\omega_{k,1}], [\omega_{k,2}], [\omega_{k,3}] \rangle
= \langle [\omega_{J_k,1}] + a_k e_1, \, [\omega_{J_k,2}] + b_k e_1, \, [\omega_{k,3}] \rangle.
\]
Moreover, let 
$\lambda= 1/\sqrt{[\omega_{k,3}]\cdot [\omega_{k,3}]}$, then
the coefficient $\alpha_{22}$ of $\lambda[\omega_{k,3}]$ tends to $0$ by Theorem~\ref{thm-2dimensional-collapsing-elliptic-approximation}.  
Hence $\eta_k$ converges to $\eta_\infty \in \mathcal{M}(a)$ by Proposition~\ref{prop-period-behaviour-characterized-boundary}, and Theorem~\ref{thm-2dimensional-collapsing-period-behaviour} holds in this case.

\subsubsection{Case 2: $\xi_\infty\notin \mathcal{M}_{\mathrm{Jac}}$.}

In this case, we find three tori $T_{1,k}, T_{2,k}, T_{3,k}$ in $X_k$ that represent primitive isotropic elements in $H_2(X_k,\mathbb{Z}) \cong \Lambda_{\mathrm{K3}}^\vee$. The behaviour of $\xi_k$ is controlled by the integrals of $\omega_{\alpha,k}$ on these tori. 

Since $\overline{\Phi}_W(\xi_\infty)=(X_\infty,g_\infty)$ has dimension 2, we know $\xi_\infty=\pi([h_8:h_{12}])$ for some $[h_8:h_{12}]\in X^s$ with $\Delta\equiv 0$, and $(X_\infty,g_\infty)=T^2/\{\pm 1\}$ with a flat metric.	
Let $R\subset X_\infty$ be the complement of finitely many small disks, and let $\pi_k^\dagger: R_k\to R$ be as in Theorem~\ref{thm-sun-zhang-fibration}.  
This also induces a torus fibration $\pi_k^\dagger:\tilde{R}_k\to \tilde{R}$ over the universal cover $\tilde{R}$ of $R$.		

Let $x_{\alpha,k}, t_{\alpha,k}, W_{\alpha\beta,k} \in C^\infty(\tilde{R}_k)$ and $\bm{u}_{\alpha,k}\in \mathbb{R}^2$ be as in \S\ref{section-case-n=2}. We may choose $t_{\alpha,k}$ such that $\bm{u}_{1,k} = (u_{1,k},0)$ in the coordinates $t_{1,k},t_{2,k}$. By the discussion in \S \ref{subsection-convergence-Rn-inv-metric}, we may assume that $W_{\alpha\beta,k}\to \delta^a_b$, and thus $|\bm{u}_{\alpha,k}|\to 0$ since the diameter of fibers goes to $0$.  
Identifying $\bm{u}_{\alpha,k}$ as elements of $\mathbb{C}$ in the complex coordinate $z=t_{1,k}+i t_{2,k}$, we may assume $\frac{\bm{u}_{2,k}}{\bm{u}_{1,k}}$ lies in the fundamental set $\mathcal{D}$ \eqref{definition-fundamental-domain-elliptic}.  

\begin{lem}\label{lem-j-to-infty}
	Let $j_{k,0} \coloneq j\bigl(\frac{\bm{u}_{2,k}}{\bm{u}_{1,k}}\bigr)$. Then $j_{k,0} \to \infty$, i.e., $\operatorname{Im}\!\bigl(\frac{\bm{u}_{2,k}}{\bm{u}_{1,k}}\bigr)\to \infty$.
\end{lem}

\begin{proof}
	For any $p\in R^0$ and $k\to \infty$, the torus fiber $(\pi_k^\dagger)^{-1}(p)$ is close to some fiber of $\pi_k: X_k \to \mathbb{P}^1$ by Theorem~\ref{thm-zhang}. The $j$-invariant of these $\pi_k$-fibers is close to $j_{k,0}$, and these fibers correspond to a subset of $\mathbb{P}^1$ with positive volume. Thus, by Remark~\ref{rmk-j-to-infty}, we have $j_{k,0} \to \infty$.
\end{proof}

We can also view $x_{\alpha,k}, t_{\alpha,k}, W_{\alpha\beta,k}$ as multi-valued functions on $R$. 
Consider a loop $\gamma$ based at $p\in R$.  
Along $\gamma$, we obtain new functions $x'_{\alpha,k}, t'_{\alpha,k}$.  
Note that $x_\alpha = \iota_{\partial_{t_\alpha}} \omega$. We can write  
\[
(dt_{1,k}', dt_{2,k}') = (dt_{1,k}, dt_{2,k}) (A_k^T)^{-1}, \qquad
(dx_{1,k}', dx_{2,k}') = (dx_{1,k}, dx_{2,k}) A_k,
\]
for some monodromy matrix $A_k \in SL_2(\mathbb{R})$.
Since the limit monodromy is $\pm \mathrm{Id}$ on $X_\infty$, we know $A_k \to \pm \mathrm{Id}$. Moreover, since $(\pi_k^\dagger)^{-1}(p) = \mathbb{R}^2 / \langle \bm{u}_{1,k}, \bm{u}_{2,k} \rangle$, we know $A_k$ preserves the lattice $\langle \bm{u}_{1,k}, \bm{u}_{2,k} \rangle$. Thus, by Lemma~\ref{lem-j-to-infty}, $A_k$ must preserve $\mathbb{Z}\bm{u}_{1,k}$ for $k\gg 1$, otherwise the coefficients of $A_k$ would be unbounded.

Write
\[
X_\infty = T^2 / \{\pm 1\} = \mathbb{R}^2 / \langle \bm{v}_1, \bm{v}_2, \{\pm 1\} \rangle,
\]
where $\{\pm 1\}$ acts by $p\mapsto -p$, and $\bm{v}_\alpha \in \mathbb{R}^2$ act by translations.  
Let $C_1$ be a circle in $R$ given by $(p_0 + \mathbb{R} \bm{v}_1)/\mathbb{Z}$ for some $p_0 \in R$.  	
From the discussion above, there exist $\bm{u}_{\Delta,k} \in \mathbb{R}^2$ and $A_k \in SL_2(\mathbb{R})$ preserving $\mathbb{Z}\bm{u}_{1,k}$, such that
\[
(\pi_k^\dagger)^{-1} C_1 = ((p_0 + \mathbb{R}\bm{v}_1)\times T_k^2)/\sim, \qquad
(p,\bm{t}) \sim (p+\bm{v}_1, A_k\bm{t} + \bm{u}_{\Delta,k}),
\] 
and we may take $\bm{u}_{\Delta,k} = \lambda_1 \bm{u}_{1,k} + \lambda_2 \bm{u}_{2,k}$ for $\lambda_\alpha \in [0,1)$, thus $\bm{u}_{\Delta,k} \to 0$. 
Moreover, since the monodromy group $A_k$ converges to $\mathrm{Id}$ along $C_1$, we know $A_k(\bm{u}) = \bm{u} + c(\bm{u}) \bm{u}_{1,k}$ for some $c\in (\mathbb{R}^2)^\vee$.
Thus
\[
T_{1,k} = \{ (p_0 + s \bm{v}_1, \, \mathbb{R} \bm{u}_{1,k} + s \bm{u}_{\Delta,k}) : s \in \mathbb{R} \} / \sim
\]
defines a torus in $C_1 \times T_k^2 \subset X_k$.

Similarly, we can define $T_{2,k}$ as above, and we define $T_{3,k}$ to be a torus fiber of $\pi_k$.

\begin{lem}\label{lem-primitive-2dn=2}
	The tori $T_{\alpha,k}$ defined above represent primitive homology classes in $H_2(X_k,\mathbb{Z})$.
\end{lem}

\begin{proof}
	Let $V_1 \cong (-\varepsilon,\varepsilon) \times C_1$ be a neighborhood of $C_1$ in $X_\infty$, and $\mathbb{R}^2 / \langle \bm{u}_{1,k}\rangle$ be a covering space of $T_k$. 
	We can view $T_{1,k}$ as a torus in $V_1 \times \mathbb{R}^2 / \langle \bm{u}_{1,k}\rangle \cong T_1 \times (-\varepsilon,\varepsilon) \times \mathbb{R}$. We may choose a neighborhood of $T_{1,k}$ such that the hyperkähler triple $\omega^\dagger_{\alpha,k}$ is the standard one after a hyperkähler rotation.  
	Hence, $T_{1,k}$ can be deformed into a holomorphic torus $T_{1,k}'$ by Theorem~\ref{thm-zhang}, and similarly for $T_{2,k}$. The conclusion then follows from Proposition~\ref{prop-elliptic-curve-elliptic-K3}(i).
\end{proof}
Since the tori are represented by disjoint cycles in the local product model,
their pairwise intersection numbers vanish:
$T_{\alpha,k}\cdot T_{\beta,k}=0$. Thus, by Lemma~\ref{lem-primitive-2dn=2}, $[T_{\alpha,k}]^\vee$ can be extended to a standard basis
$\mathcal{B}_k = \{e_{i,k}\}$ of $H^2(X_k,\mathbb{R}) \cong \Lambda_{\mathrm{K3},\mathbb{R}}$
such that $e_{i,k} = [T_{i,k}]^\vee$ for $i = 1,2,3$.
By passing to a subsequence, we may assume that these standard bases are of the same type, and by choosing suitable $\mathcal{B}_k$, we may further assume that $\mathcal{B}_k$ are identified with a fixed standard basis $\mathcal{B} = \{e_i\}$ of $\Lambda_{\mathrm{K3},\mathbb{R}}$.

Recall that $\omega^\dagger_{\beta,k}$ is defined in \eqref{standard-form-n=2}.
One can compute the integrals of $\omega_{\alpha,k}$ over $T_{\beta,k}$ as follows:
\[
(A_{\alpha\beta})_k \coloneq
\Bigl( \int_{T_{\alpha,k}} \omega_{\beta,k} \Bigr)
= \Bigl( \int_{T_{\alpha,k}} \omega^\dagger_{\beta,k} \Bigr) \sim
u_{1,k} \begin{pmatrix}
	\int_{p_0}^{p_0+\bm{v}_1} dx_{1,\infty} & -\int_{p_0}^{p_0+\bm{v}_1} dx_{2,\infty} & * \\[1ex]
	\int_{p_0}^{p_0+\bm{v}_2} dx_{1,\infty} & -\int_{p_0}^{p_0+\bm{v}_2} dx_{2,\infty} & * \\[1ex]
	0 & 0 & \operatorname{Im}(\bm{u}_{2,k})
\end{pmatrix},
\]and we have $\int_{X_k} \omega_{\alpha,k}^2 \sim u_{1,k} \operatorname{Im}(\bm{u}_{2,k})$.
Thus, after rescaling so that $\int_{X_k} \omega_{\alpha,k}^2 = 1$, the resulting matrix is
\[
(A'_{\alpha\beta})_k \sim \frac{1}{\sqrt{u_{1,k} \operatorname{Im}(\bm{u}_{2,k})}} (A_{\alpha\beta})_k.
\]
Since $\frac{u_{1,k}}{\operatorname{Im}(\bm{u}_{2,k})}\to 0$ and $\operatorname{Im}(\bm{u}_{2,k})\to 0$, the sequence lies in Case~3 of Proposition~\ref{prop-period-behaviour-characterized-boundary} after applying an $SO(2)$ rotation to $\omega_{1,k}$ and $\omega_{2,k}$.
Hence, the sequence $\eta_k$ converges to a point $\eta_\infty \in \mathcal{M}(c_i)$ with $\overline{\Phi}'(\eta_\infty) = (X_\infty,g_\infty)$.  
This finishes the proof of Theorem~\ref{thm-2dimensional-collapsing-period-behaviour}.

\subsection{Collapsing of codimensional one}

We now deal with the case $d=3$.  
This section aims to prove the following:

\begin{thm}\label{thm-3dimensional-collapsing-period-behaviour}
	Suppose $\eta_k \in \mathcal{M}_{\mathrm{K3}}$ with
	$\eta_k \to \eta_\infty \in \overline{\mathcal{M}}_{\mathrm{K3}}^{\mathrm{Sat,ad}}$
	and
	$\Phi(\eta_k) \to (X_\infty,g_\infty)$ in the GH sense.
	If $\dim X_\infty = 3$, then $\overline{\Phi}'(\eta_\infty) = (X_\infty,g_\infty).$
\end{thm}

Similarly, we may assume $(X_k,g_k)$ is smooth. Let $\bm{v}_\alpha \in \mathbb{R}^3$ such that
\[
X_\infty = T^3/\{\pm 1\} = \mathbb{R}^3 / \langle \bm{v}_\alpha, \{\pm 1\} \rangle.
\]
Let $\pi^\dagger_k: R_k \to R$ be the fibration given by Theorem~\ref{thm-sun-zhang-fibration}, where $R \subset X_\infty$ is the complement of finitely many small disks.  
This induces an $S^1$ fibration $\pi^\dagger_k: \tilde{R}_k \to \tilde{R}$ over the universal cover of $R$.	

By \S\ref{section-case-n=1}, there exists an $\mathbb{R}$-invariant hyperkähler triple on $\tilde{R}_k$ given by
\begin{align}\label{equation-hyperkahler-triple-n=3}
	\left\{
	\begin{array}{lr}
		\omega^\dagger_{1,k} = V_k \, dx_{2,k} \wedge dx_{3,k} + dx_{1,k} \wedge \theta_k, \\[2mm]
		\omega^\dagger_{2,k} = V_k \, dx_{3,k} \wedge dx_{1,k} + dx_{2,k} \wedge \theta_k, \\[2mm]
		\omega^\dagger_{3,k} = V_k \, dx_{1,k} \wedge dx_{2,k} + dx_{3,k} \wedge \theta_k,
	\end{array}
	\right.
\end{align}
where $V_k$ is a harmonic function in $(x_{1,k}, x_{2,k}, x_{3,k})$, and there exists a one-form $\alpha_k$ on the base $\tilde{R}$ such that 
\[
\theta_k = dt_k + (\pi_k^\dagger)^*\alpha_k, \quad d\alpha_k = *\, dV_k.
\]
As in \S\ref{subsection-convergence-Rn-inv-metric}, we may assume
\[
\int_{S^1} \theta_k = \varepsilon_k \to 0, 
\qquad 
V_k \to 1,
\]
and the tuple $(dx_{1,k}, dx_{2,k}, dx_{3,k})$ converges smoothly to an affine flat 1-form on $\tilde{R}$.

Let $C_\alpha$ ($1 \le \alpha \le 3$) be circles in $R$ given by $(p_0 + \mathbb{R} \bm{v}_\alpha)/\mathbb{Z}$ for some $p_0 \in R$, and define $T_{\alpha,k} = (\pi^\dagger_k)^{-1}(C_\alpha)$. Then $T_{\alpha,k}$ is a torus since $X_k$ is oriented (cf.~\cite[§7.2]{sun-zhang-24}).

\begin{lem}\label{prop-primitive-n=3}
	The tori $T_{\alpha,k}$ defined above represent primitive homology classes in $H_2(X_k,\mathbb{Z})$.
\end{lem}
Assuming Lemma \ref{prop-primitive-n=3}, we can now prove Theorem \ref{thm-3dimensional-collapsing-period-behaviour} as follows:
\begin{proof}[Proof of Theorem \ref{thm-3dimensional-collapsing-period-behaviour} ]

	It is easy to see that $T_{\alpha,k} \cdot T_{\beta,k} = 0$. Thus, by Lemma~\ref{prop-primitive-n=3}, $[T_{\alpha,k}]^\vee$ can be extended to a standard basis
	$\mathcal{B}_k = \{e_{i,k}\}$ of $H^2(X_k,\mathbb{R}) \cong \Lambda_{\mathrm{K3},\mathbb{R}}$,
	as in the codimension~2 case.
	By passing to a subsequence, we may also assume that the bases $\mathcal{B}_k$ are identified with a fixed standard basis $\mathcal{B} = \{e_i\}$ of $\Lambda_{\mathrm{K3},\mathbb{R}}$.

	Calculation shows
	\[
	(A_{\alpha\beta})_k \coloneq \int_{T_{\alpha,k}} \omega_{\beta,k} = \varepsilon_k \int_{p_0}^{p_0 + \bm{v}_\alpha} dx_{\beta,k}, 
	\qquad 
	\int_{X_k} \omega_{\alpha,k}^2 \sim \varepsilon_k.
	\]
	Thus, $\frac{1}{\varepsilon_k} (A_{\alpha\beta})_k$ converges. 
	Hence, the sequence lies in Case~2 of Proposition~\ref{prop-period-behaviour-characterized-boundary}.  
	Therefore, the sequence $\eta_k$ converges to a point $\eta_\infty \in \mathcal{M}(b_i)$ with $\overline{\Phi}'(\eta_\infty) = (X_\infty,g_\infty)$, and Theorem~\ref{thm-3dimensional-collapsing-period-behaviour} follows.
\end{proof}

To prove Lemma~\ref{prop-primitive-n=3}, we first establish the following auxiliary lemma.

\begin{lem}\label{lem-primitive-n=3}
	Given $\varepsilon>0$, let
	\[
	Y = B \times S^1 = \{(x_1,x_2,x_3) \in \mathbb{R}^3 : x_1^2+x_2^2 < \varepsilon\}/\mathbb{Z},
	\]
	where the $\mathbb{Z}$-action is given by $x_3 \mapsto x_3+1$.  
	If $\alpha$ is a closed $2$-form on $Y$, then there exists a $1$-form $\beta$ such that $d\beta = \alpha$ and, for any $l>0$,
	\[
	\|\beta\|_{C^l} \leq C \|\alpha\|_{C^l}.
	\]
\end{lem}

\begin{proof}
	This follows from the standard homotopy operator construction.  
	Consider the maps
	\[
	Y \xrightarrow{i_t} Y \times [0,1] \xrightarrow{\pi} Y,
	\]
	where $i_t(x)=(x,t)$ and $\pi(x,t)=(tx_1, tx_2, x_3)$.  
	By \cite[Lemma~17.9]{lee-intro-to-smooth-manifolds-13}, we have
	\[
	\alpha = (i_1^* \pi^* - i_0^* \pi^*)\alpha = (S d + dS)\pi^*\alpha, \qquad \text{where } S\omega = \int_0^1 i_t^* \iota_{\partial_t} \omega \, dt.
	\]
	Thus $\beta = S \pi^* \alpha$ satisfies $d\beta = \alpha$.  
	Explicitly, writing $\beta = \beta_i dx_i$ and $\alpha = \alpha_{ij} dx_i \wedge dx_j$ $(i<j)$, and setting $x_t = \pi(x,t)$, we have
	\[
	\beta_1(x) = \int_0^1 (-\alpha_{12}(x_t) t x_2) \, dt, \quad 
	\beta_2(x) = \int_0^1 (\alpha_{12}(x_t) t x_1) \, dt, \quad 
	\beta_3(x) = \int_0^1 (\alpha_{13}(x_t) x_1 + \alpha_{23}(x_t) x_2) \, dt.
	\]
	Which implies $\|\beta\|_{C^l} \leq C \|\alpha\|_{C^l}$.
\end{proof}

\begin{proof}[Proof of Lemma~\ref{prop-primitive-n=3}]
	It suffices to prove the statement for $T_{1,k}$.  
	Assume $\bm{v}_1=(0,0,1)$. After a hyperkähler rotation, we may assume $\int_{p_0}^{p_0+\bm{v}_1} x_{\beta,k} = 0$ for $\beta=1,2$, and $\int_{p_0}^{p_0+\bm{v}_1} dx_{3,k} = 1$.  
	Thus $x_{\beta,k}$ is globally defined in a neighborhood of $T_{1,k}$ for $\beta=1,2$, and we assume $x_{1,\infty}(p_0) = x_{2,\infty}(p_0) = 0$.   
	
	For $k \gg 1$, a neighborhood of $T_{1,k}$ is given by
	\[
	Y \cong B_\varepsilon \times T
	= \{(x_{1,k}, x_{2,k}, x_{3,k}, t_k) : |(x_{1,k}, x_{2,k})| < \varepsilon \}/\sim,
	\]
	where
	\begin{equation}\label{T-equation-quotient-1-epsilon-k}
		(x_{1,k}, x_{2,k}, x_{3,k}, t_k) \sim (x_{1,k}, x_{2,k}, x_{3,k}+m, t_k + n\varepsilon_k), \qquad \forall m,n \in \mathbb{Z}.
	\end{equation}
	Then $T_{1,k}$ is a perturbation of the $T$ given by $(x_{1,k}, x_{2,k})=(0,0)$. 
	The hyperkähler triple on $Y_k$ is given by \eqref{equation-hyperkahler-triple-n=3}, with $d\alpha_k = * dV_k \to 0$.  	
	
	By Lemma~\ref{lem-primitive-n=3}, there exists $\alpha_k' \to 0$ such that $d\alpha_k' = d\alpha_k$, with $\|\alpha_k'\| \to 0$. 
	Note that $H^1(Y,\mathbb{R})$ is generated by $dx_{3,k}$, thus we can take $s \in C^\infty(Y)$ such that $\alpha_k - \alpha_k' = ds + C\, dx_{3,k}.$
	In coordinates $t_k' = t_k + s$, we have 
	\[
	\theta_k = dt_k + \alpha_k = dt_k' + C\, dx_{3,k} + \alpha_k'.
	\]
	From \eqref{T-equation-quotient-1-epsilon-k}, we can also take $t_k'' = t_k' + n\varepsilon_k x_{3,k}$ for $n \in \mathbb{Z}$.
	Hence, we may assume $C \leq \varepsilon_k$.
	
	In $(x_{1,k}, x_{2,k}, x_{3,k}, t_k'')$ coordinates, the hyperkähler triple is close to the standard one.  
	Hence $T$ can be perturbed into a holomorphic torus $T'$ by Theorem~\ref{thm-zhang}.  
	Since $T'$ is homologous to $T_{1,k}$, the conclusion follows from Proposition~\ref{prop-elliptic-curve-elliptic-K3}(i).
\end{proof}

\subsection{Proof of main Theorem~\ref{thm-main-theorem}}\label{section-proof-of-main-thm}

Denote by $2^{\overline{\mathfrak{M}}}$ the collection of subsets of $\overline{\mathfrak{M}}$.  
We define a set-valued map
\[
\mathfrak{F} : \overline{\mathcal{M}}_{\mathrm{K3}}^{\mathrm{ad}} \setminus \mathcal{M}_{\mathrm{K3}} \longrightarrow 2^{\overline{\mathfrak{M}}}
\]
by
\[
\mathfrak{F}(\eta_\infty)
= \bigl\{ \lim \Phi(\eta_k) : \eta_k \in \mathcal{M}_{\mathrm{K3}},\, \eta_k \to \eta_\infty \bigr\}.
\]
We know that the restriction of $\Phi$ to $\mathcal{M}_{\mathrm{K3}}$ is continuous.  
It is easy to see that $\Phi|_{\mathcal{M}_{\mathrm{K3}}}$ extends continuously to a map
$
{\overline{\Phi}}^* : \overline{\mathcal{M}}_{\mathrm{K3}}^{\mathrm{ad}} \longrightarrow \overline{\mathfrak{M}}
$
if and only if $\mathfrak{F}(\eta_\infty)$ is a singleton for all $\eta_\infty\in \overline{\mathcal{M}}_{\mathrm{K3}}^{\mathrm{ad}} \setminus \mathcal{M}_{\mathrm{K3}} $, and $\mathfrak{F}(\eta_\infty) = \{{\overline{\Phi}}^*(\eta_\infty)\}$ if so.

By Gromov's precompactness theorem, $\mathfrak{F}(\eta_\infty)$ is always nonempty. Moreover, we have the following:

\begin{prop}\label{tt-connected}
	For any $\eta_\infty\in \overline{\mathcal{M}}_{\mathrm{K3}}^{\mathrm{ad}} \setminus \mathcal{M}_{\mathrm{K3}}$, we have $\mathfrak{F}(\eta_\infty)$ is a connected subset of $\overline{\mathfrak{M}}$.
\end{prop}

\begin{proof}
	By Proposition~\ref{prop-topological-basis-connectness}, we can choose a neighborhood system $\{U_i\}$ at $\eta_\infty$ such that each $U_i \cap \mathcal{M}_{\mathrm{K3}}$ is connected.  
	Hence, $\overline{\Phi(U_i \cap \mathcal{M}_{\mathrm{K3}})}$ is connected as well.  
	Since
	\[
	\mathfrak{F}(\eta_\infty)
	= \bigcap_{i=1}^\infty \overline{\Phi(U_i \cap \mathcal{M}_{\mathrm{K3}})}
	\]
	is the intersection of a decreasing sequence of compact connected sets, it is connected by \cite[Cor.~6.1.20]{engelking-topology-89}.
\end{proof}
We already know that all limit spaces $(X_\infty,g_\infty) \in \mathfrak{F}(\eta_\infty)$ must be volume collapsing (\cite[Thm.~II, Thm.~IV]{anderson-92}). If $(X_\infty,g_\infty)$ has dimension 2 or 3, then by Theorems~\ref{thm-2dimensional-collapsing-period-behaviour} and~\ref{thm-3dimensional-collapsing-period-behaviour}, we have $(X_\infty,g_\infty) = \overline{\Phi}'(\eta_\infty)$.  Hence
\[
\mathfrak{F}(\eta_\infty) \subset \{\overline{\Phi}'(\eta_\infty), I^1\}.
\]
Since $\mathfrak{F}(\eta_\infty)$ is connected by Proposition~\ref{tt-connected}, we conclude that
\[
\mathfrak{F}(\eta_\infty) = \{\overline{\Phi}'(\eta_\infty)\} 
\quad \text{or} \quad 
\mathfrak{F}(\eta_\infty) = \{I^1\}.
\]

On the other hand, by \cite[Thm.~1.1]{gross-tosatti-zhang-16}, we have $\overline{\Phi}'(\eta_\infty) \in \mathfrak{F}(\eta_\infty)$ for $\eta_\infty \in \mathcal{M}(a)$.  
Thus $\mathfrak{F}(\eta_\infty) = \{\overline{\Phi}'(\eta_\infty)\}$ for $\eta_\infty \in \mathcal{M}(a)$.  
Moreover, by considering the moduli of flat orbifolds $T^4/\{\pm 1\}$ as in \cite[§6.3.3]{odaka-oshima-21}, we also know that $\overline{\Phi}'(\eta_\infty) \in \mathfrak{F}(\eta_\infty)$ for $\eta_\infty \in \mathcal{M}(b_1) \cup \mathcal{M}(c_1)$ and $I^1 \in \mathfrak{F}(\eta_\infty)$ for $\eta_\infty \in \mathcal{M}(b_2) \cup \mathcal{M}(c_2)$.  

Therefore,
\[
\mathfrak{F}(\eta_\infty) = \{\overline{\Phi}(\eta_\infty)\},
\]
and $\Phi$ extends continuously to $\overline{\Phi}$ on $\overline{\mathcal{M}}_{\mathrm{K3}}^{\mathrm{ad}}$.  
This completes the proof of Theorem~\ref{thm-main-theorem}.

\subsection{Collapsing with fixed complex structure} \label{section-collapsing-with-fixed-complex-structure}

Now we prove Corollary \ref{crl-fix-complex-structure}.  
Consider a K3 surface with a fixed complex structure $(X,J)$.  
Equivalently, we consider hyperkähler triples $(\omega_1,\omega_2,\omega_{3,k})$ where  
$W = \langle [\omega_1],[\omega_2] \rangle \subset \Lambda_{K3,\mathbb{R}}$ is fixed and  
$\Omega = \omega_1 + i \omega_2$ is the holomorphic volume form of $(X,J)$.  
Let $\partial(\overline{\mathfrak{M}_J})$ denote the set of all collapsing limits of unit-diameter
hyperkähler metrics on $(X,J)$.

First, assume $\dim\big(W \cap H^2(X,\mathbb{Q})\big) = 0$.  
Then by an ergodic theorem of Verbitsky \cite{verbitsky-15-acta} \cite[Thm.~2.5 (ii)]{verbitsky2017ergodic},  
the $G_\mathbb{Z}$-orbit of $W$ is dense in the space of all $2$-positive definite subspaces of $\Lambda_{K3,\mathbb{R}}$.  
Thus we can choose $\omega_3$ such that the associated point of $(\omega_1,\omega_2,\omega_3)$ in  
$\mathcal{M}_{\mathrm{K3}}$ is arbitrarily close to any given point of $\mathcal{M}_{\mathrm{K3}}$, and case (i) of Corollary \ref{crl-fix-complex-structure} follows.

Second, assume $\dim\big(W \cap H^2(X,\mathbb{Q})\big) = 1$.  
Let $\lambda$ be the primitive element of $W \cap H^2(X,\mathbb{Z})$.  
Then \cite[Thm.~2.5 (iii)]{verbitsky2017ergodic} implies that the closure of $G_\mathbb{Z}\cdot W$ is the set of all $2$-positive subspaces containing $\lambda$.  
Thus we can choose $\omega_3$ such that the associated point of $(\omega_1,\omega_2,\omega_3)$ in  
$\mathcal{M}_{\mathrm{K3}}$ is close to any given point in $\mathcal{M}_{2d}$.  
Then case (ii) of Corollary \ref{crl-fix-complex-structure} follows from Corollary \ref{crl-fix-polarization}.

Third, assume $\dim\big(W \cap H^2(X,\mathbb{Q})\big) = 2$, i.e., $W$ is spanned by two elements of $H^2(X,\mathbb{Q})$.  
Note that every such generalized KE metric can be realized as a GH limit by Theorem \ref{thm-main-theorem},  
and the finiteness follows from \cite[Cor.~8.4.6]{huybrechts-lecturenotes}.  
We are going to prove that all GH limits are of this form.

Let $(\omega_1,\omega_2,\omega_{3,k})$ be a sequence of hyperkähler triples of unit volume such that the associated metrics collapse after being rescaled to have unit diameter.
From the properties of Satake topology, we can choose a standard basis $\mathcal{B}=\{e_i\}$ of $H^2(X,\mathbb{R})$ and elements $g_k \in G_\mathbb{Z}$, $r_k \in \mathrm{SO}(3)$,  
such that the coefficients of $([\omega_1'],[\omega_2'],[\omega'_{3,k}]) = ([\omega_1],[\omega_2],[\omega_{3,k}]) r_k$ in the basis $g_k \mathcal{B}$ fit into one of the cases of Proposition~\ref{prop-period-behaviour-characterized-boundary} after passing to a subsequence.

Choose a basis $\kappa_1,\kappa_2$ of $W\cap H^2(X,\mathbb{Z})$.  
We may write
\[
\kappa_1 = u_k [\omega_1'] + v_k [\omega_2'] + w_k [\omega'_{3,k}],
\qquad
u_k^2 + v_k^2 + w_k^2 = \kappa_1^2.
\]
Since the coefficient $\alpha_{22}$ of $\kappa_1$ is an integer in the basis $g_k \mathcal{B}$, we obtain $w_k = 0$ for $k \gg 1$.  
The same holds for $\kappa_2$, hence
\[
W = \langle \kappa_1,\kappa_2 \rangle
= \langle [\omega_1'],[\omega_2'] \rangle.
\]
Comparing the coefficient $\alpha_{21}$, we see that we are in Case~(1) of Proposition~\ref{prop-period-behaviour-characterized-boundary}.  
Therefore, by Theorem~\ref{thm-main-theorem}, every possible limit is the
generalized Kähler–Einstein metric on $\mathbb{P}^1$ associated with the elliptic fibration  
$\pi: X \to \mathbb{P}^1$ determined by the isotropic vector $e_1$.  
This proves case (iii) of Corollary \ref{crl-fix-complex-structure}.

\subsection{Further discussion and questions}

Finally, we list some natural questions arising from our work.

\begin{itemize}
	\item 
	It is also interesting to consider more refined structures of collapsing. It is natural to ask:
	\begin{openquestion}
		Can we find a compactification of $\mathcal{M}_{\mathrm{K3}}$ that also characterizes the data of the renormalized limit measure in the case $\dim(X_\infty) = 1$?
	\end{openquestion}
	In this case of $\dim(X_\infty) = 1$, Theorem~\ref{thm-sun-zhang-fibration} gives a fibration over an open interval whose fibers are Heisenberg nilmanifolds. One can still produce two primitive isotrivial cycles by arguments similar to those in the proofs of Theorems~\ref{thm-2dimensional-collapsing-period-behaviour} and~\ref{thm-3dimensional-collapsing-period-behaviour}, but this is not sufficient to control the full moduli behaviour. We also refer to \cite[\S3.1.3]{MR4480094} for related discussions on measured GH convergence.
	
	\item 
	The main theorem of this paper may also extend to the case of Enriques surfaces; this will be discussed in a subsequent paper.
\end{itemize}

\appendix

	\section{Volume comparison for generalized KE metrics}
	
	In this appendix we establish the volume comparison for generalized KE metrics on $\mathbb{P}^1$.  
	The results in this appendix are used in \S\ref{subsection-estimate-singularities} and \S\ref{subsection-proof-of-thm-GIT}.
	
	\begin{prop}\label{prop-volume-comparison-gen-KE}
		Let $(\mathbb{P}^1, g)$ be a generalized KE metric induced by an elliptic K3 surface. Then $(\mathbb{P}^1, g)$ satisfies the relative volume comparison
		\[
		\frac{\mathrm{vol}(B_{r_1}(p))}{\mathrm{vol}(B_{r_2}(p))} \ge \left(\frac{r_1}{r_2}\right)^2, \qquad (r_1 < r_2).
		\]
	\end{prop}
	
	\begin{proof}
		We are going to show that $(\mathbb{P}^1, g)$ can be approximated by smooth metrics $g_t$ with almost nonnegative curvature.
		
		Near each singularity $p$, we can take a coordinate chart $B_3 = \{\bm{x} \in \mathbb{R}^2 : |\bm{x}| < 3\}$ with $p = (0,0)$ being the only singularity in $B_3$.  
		Let $g = e^{2u}(dx_1^2 + dx_2^2)$ be the singular metric on $B_3$.  
		There are standard models for such metrics near $p$ (see for example \cite[Definition 3.19]{sun-zhang-24}), 
		and from the definition we have  
		\[
		\partial_r u(\bm{x}) > 0, \qquad 0 \le u(\bm{x}) \le -\tfrac{5}{6}\log |\bm{x}| + C,
		\]
		for some constant $C > 0$.  
		We also know that the curvature $K(g) = - e^{-2u} \Delta u \ge 0$ when $\bm{x} \neq 0$.  
		Thus, for any smooth function $\phi \ge 0$ with compact support in $B_3$, we have
		\[
		\int_{B_3} u \, \Delta \phi \, dx_1 dx_2
		= \lim_{\varepsilon \to 0}
		\left(
		\int_{B_3 \setminus B_\varepsilon} \phi \, \Delta u \, dx_1 dx_2
		+ \int_{\partial B_\varepsilon} 
		\left(u \partial_r \phi - \phi \partial_r u\right) ds
		\right)
		\le 0.
		\]
		Thus $u$ is superharmonic in $B_3$ in the sense of distributions.
		
		Let $h = h(r)$ be a smooth radial bump function satisfying
		\[
		\operatorname{supp}(h) \subset B_1, \qquad 0 \le h \le 1, \qquad 
		H = \int_{B_1} h \, dx_1 dx_2 \ge \tfrac{10}{11} |B_1|.
		\]
		Define $h_t(\bm{x}) = \frac{1}{t^2 H} h\!\left(\frac{\bm{x}}{t}\right)$ for $t < 1$, and
		\[
		u_t(\bm{x}) = (h_t * u)(\bm{x}) = \int h_t(\bm{y})\, u(\bm{x}-\bm{y}) \, dy_1 dy_2.
		\]
		Then $u_t$ is a superharmonic function defined on $B_2$, and $u_t \to u$ smoothly away from $\bm{x}=0$.  
		Next, choose a smooth cutoff function $\rho: \mathbb{R}^+ \to \mathbb{R}$ satisfying $\rho(r) = 0$ for $r \le 1$ and $\rho(r) = 1$ for $r \ge 2$.  
		Let
		\[
		u_t' = \rho\, u_t + (1-\rho)\, u,
		\]
		and define the modified metric
		\[
		g_t = e^{2u_t'}(dx_1^2 + dx_2^2)
		\]
		on $B_2$, replacing $g$ there while keeping $g_t = g$ outside $B_2$.
		We see that the curvature $K(g_t) \ge 0$ outside the gluing region $\{1 \le |\bm{x}| \le 2\}$, and for any $\varepsilon > 0$,  
		$K(g_t) \ge -\varepsilon$ in the gluing region  $\{1 \le |\bm{x}| \le 2\}$ when $t \to 0^+$.
		
		Finally, we estimate the metric $g_t$ on $B_2$ as follows.  
		For $|\bm{x}| \le 2t$, there exist a constant $C>0$ (may change line by line),
		\[
		u_t(\bm{x})
		\le \frac{1}{t^2 H}\int_{B_t} \left(-\tfrac{5}{6}\log |\bm{y}| + C\right) dy_1 dy_2
		\le  \frac{|B_1|}{H}\left(-\tfrac{5}{6}\log t + C\right)
		\le -\tfrac{11}{12}\log|\bm{x}| + C.
		\]
		For $|\bm{x}| > 2t$,
		\[
		u_t(\bm{x}) 
		\le \max_{|\bm{y}| \le t} \{u(\bm{x}-\bm{y})\}
		\le -\tfrac{5}{6}\log(|\bm{x}| - t) + C 
		\le -\tfrac{11}{12}\log|\bm{x}| + C.
		\]
		Hence, for $|\bm{x}| < 2$, 
		\[
		g_t = e^{2u_t}(dx_1^2 + dx_2^2) \le C |\bm{x}|^{-11/6}(dx_1^2 + dx_2^2).
		\]
		Thus, for any $\varepsilon > 0$, there exists $\delta > 0$ such that the diameter and volume of $B_\delta$ with respect to $g_t$ are uniformly bounded by $\varepsilon$.  
		
		Since the metrics $g_t$ converge smoothly to $g$ outside $B_\delta$, 
		we obtain that $(B_2, g_t)$ converges to $(B_2, g)$ in the measured GH sense.  
		The conclusion follows from the Bishop--Gromov volume comparison theorem for $ (\mathbb{P}^1,g_t)$ since the curvature $K(g_t) > -\varepsilon$.
	\end{proof}
	\begin{prop}\label{prop-volume-comparison}
		Let $(\mathbb{P}^1, g)$ be a generalized KE metric induced by an elliptic K3 surface, and 
		\[
		S = \bigsqcup_{i=1}^n B_i \subset \mathbb{P}^1
		\]
		be a disjoint union of disks. Let $R \coloneqq \mathbb{P}^1 \setminus S$ and $d_R$ the distance function on $R$ defined in Definition \ref{defi-distance-function-manifold}. Assume $\mathrm{diam}(R, d_R) < 2$.  
		Then for every $\varepsilon > 0$, there exists $\delta > 0$ such that		
		\begin{enumerate}
			\item[(i)] If the perimeter $p(B_i) < \delta$ for all $i$, then 
			\[
			d_{\mathbb{P}^1}(x, y) \le d_R(x, y) < d_{\mathbb{P}^1}(x, y) + \varepsilon, \quad \forall x, y \in R.
			\]
			\item[(ii)] If $\frac{\mathrm{vol}(S)}{\mathrm{vol}(\mathbb{P}^1)} < \delta$, then 
			$
			\mathbb{P}^1 \subset B_\varepsilon(R).
			$
		\end{enumerate}
		Hence, if both assumptions in (i) and (ii) hold, then $(R, d_R) \to (\mathbb{P}^1, d)$ is an $\varepsilon$-Gromov--Hausdorff approximation, and
		\[
		d_{GH}\!\big( (R, d_R), (\mathbb{P}^1, d) \big) < 2\varepsilon.
		\]
	\end{prop}
	
	\begin{proof}
		(i) Let $\gamma: [0, L] \to \mathbb{P}^1$ be a minimizing geodesic between $x, y \in R$.  
		For each $i$, define
		\[
		m_i = \min \{ t \in [0, L] : \gamma(t) \in B_i \}, 
		\qquad 
		M_i = \max \{ t \in [0, L] : \gamma(t) \in B_i \}.
		\]
		We may replace the segment $\gamma|_{[m_i, M_i]}$ by an arc along $\partial B_i$.  
		This produces a new path $\gamma'$ lying in $\overline{R}$ whose length increases by at most $n \delta$.  
		Thus,
		\[
		d_{\mathbb{P}^1}(x, y) \le d_R(x, y) \le d_{\mathbb{P}^1}(x, y) + n\delta, 
		\quad \forall x, y \in R.
		\]
		The conclusion follows if we choose $\delta$ such that $n \delta < \varepsilon$.
		
		(ii) Choose $x \in \mathbb{P}^1$ such that $r_x = \mathrm{dist}(x, R)$ is maximal.  
		Then $B(x, r_x) \subset S$ and $\mathbb{P}^1 \subset \overline{B(x, 2 r_x + 2)}$.  
		By Proposition \ref{prop-volume-comparison-gen-KE}, we have
		\[
		\delta
		> \frac{\mathrm{vol}(S)}{\mathrm{vol}(\mathbb{P}^1)}
		\ge \frac{\mathrm{vol}(B(x, r_x))}{\mathrm{vol}(B(x, 2 r_x + 2))}
		\ge \left( \frac{r_x}{2 r_x + 2} \right)^2.
		\]
		Thus, by choosing $\delta$ sufficiently small, we obtain $r_x < \varepsilon$, and the conclusion follows.
	\end{proof}
	
	\section{Continuity of generalized KE metrics on $\mathbb{P}^1$}\label{section-appendix}
	
	In this section, we give a simpler proof of Theorem \ref{thm-geometric-realization-GIT}.
	\subsection{Background on elliptic curves}
	
	We first recall some standard facts about elliptic curves; see, for example, \cite[Chap.~VII]{serre-73} for more details.
	
	Denote by $\mathcal{L}$ the set of lattices in $\mathbb{C}$, and set 
	\[
	\mathcal{C} \coloneqq \{ (a, b) \in \mathbb{C}^2 : a^3 - 27 b^2 \neq 0 \}.
	\] 
	There is a bijection between $\mathcal{L}$ and $\mathcal{C}$, sending $L \in \mathcal{L}$ to $(a, b)$, where
	\begin{equation}\label{formula-g8g12}
		a = 60 \sum_{\gamma \in L \setminus \{0\}} \frac{1}{\gamma^4}, 
		\qquad 
		b = 140 \sum_{\gamma \in L \setminus \{0\}} \frac{1}{\gamma^6}.
	\end{equation}
	The equation of the elliptic curve $\mathbb{C}/L$ is then given by
	$
	y^2 z = 4 x^3 - a x z^2 - b z^3.
	$
	
	If $\tau_1, \tau_2$ is a basis of $L$ with $\mathrm{Im}(\bar{\tau}_1 \tau_2) > 0$, we can define the associated $j$-invariant by
	\[
	j = \frac{1728 a^3}{a^3 - 27 b^2} = j\!\left(\frac{\tau_2}{\tau_1}\right).
	\]
	Then $j: \mathbb{H} \to \mathbb{C}$ is a holomorphic function, invariant under the $\mathrm{SL}(2, \mathbb{Z})$-action, and satisfies
	\[
	j(\tau) e^{2 \pi i \tau} \to 1 \quad \text{as } \mathrm{Im}(\tau) \to \infty.
	\]
	
	\begin{lem}\label{lem-estimate-mu(a,b)} $\mathrm{(cf}$.~\cite[Lem.~7.16]{odaka-oshima-21}$\mathrm{)}$
		Given $(a, b) \in \mathcal{C}$, define $\mu(a, b)$ to be the area of the corresponding torus $\mathbb{C}/L(a, b)$.  
		Then there exists a continuous function $\lambda$ on $\mathbb{C} \cup \{\infty\}$ such that
		\[
		\mu(a, b) = \frac{\lambda(j)}{|3a|^{1/2} + |b|^{1/3}} \cdot \log^{+} |j|,
		\]
		where $\log^{+}(x) := \max\{1, \log(x)\}$, with the convention $\log^{+}(0) = 1$. Moreover, $\lambda(j) \to \frac{8}{3} \pi^2$ as $j \to \infty$.
	\end{lem}
	
	\begin{proof}
		Let $L = \langle w_1, w_2 \rangle$ be the lattice associated to $(a, b)$.  
		Then \eqref{formula-g8g12} expresses $(a, b)$ in terms of $w_1, w_2$.  
		Since $(|3a|^{1/2} + |b|^{1/3}) \mu(a, b)$ is homogeneous in $w_1, w_2$, it can be represented in the form $\lambda(j) \log^{+}|j|$.
		
		Now, let $j \to \infty$. It suffices to consider the case $L = \langle 1, w \rangle$ with $w$ lying in the standard fundamental domain $\mathcal{D}$ (Definition \ref{definition-fundamental-domain-elliptic}) and $\mathrm{Im}(w) \to \infty$.  
		In this situation, we have
		\[
		\frac{\log^{+}|j(w)|}{\mu(a, b)} = \frac{\log^{+}|j(w)|}{\mathrm{Im}(w)} \;\to\; 1,
		\]
		and $a, b$ converge to $\frac{4}{3} \pi^4, \frac{8}{27} \pi^6$ respectively.  
		This completes the proof.
	\end{proof}
	
	\subsection{Formula of generalized KE metrics}
	
	Let $(h_8, h_{12}) \in H^0(\mathbb{P}^1, \mathcal{O}(8)) \oplus H^0(\mathbb{P}^1, \mathcal{O}(12))$ represent a point in $\mathcal{M}_{\mathrm{Jac}}$.  
	In a local coordinate $t$, we can write
	\[
	(h_8, h_{12}) = \big(a(t) (dt)^{-4},\, b(t) (dt)^{-6}\big).
	\]
	
	We define a singular metric $g$ on $\mathbb{P}^1$ associated to $(h_8, h_{12})$ by
	\begin{equation}\label{equation-fomular-g}
		g = \mu(a(t), b(t)) \, |dt|^2,
	\end{equation}
	where $\mu(a,b)$ is as in Lemma~\ref{lem-estimate-mu(a,b)}. This metric is well-defined and independent of the choice of coordinate $t$.
	
	Let $\pi: X \to \mathbb{P}^1$ be the corresponding Jacobian elliptic K3 surface given by $(h_8, h_{12})$, and denote by $g^*$ the associated unit-diameter generalized KE metric on $\mathbb{P}^1$.
	
	\begin{prop}
		There exists a constant $c \in \mathbb{R}$ such that
		$
		g^* = c \, g.
		$
	\end{prop}
	
	\begin{proof}
		Denote by $\mathbb{P}^1_0 \subset \mathbb{P}^1$ the open subset $\{\Delta \neq 0\}$.  
		Then there is a lattice $L = L(h_8, h_{12}) \subset T^* \mathbb{P}^1_0$ associated with $(h_8, h_{12})$.  
		This gives a torus fibration $T^* \mathbb{P}^1_0 / L \to \mathbb{P}^1_0$.  
		
		Both $g$ and $g^*$ are generalized KE metrics with respect to this torus fibration (possibly with different holomorphic volume forms). Hence, there exists a holomorphic function $f$ on $\mathbb{P}^1_0$ such that
		$
		g^* = |f|^2 g.
		$	
		
		Suppose $p$ is a singular point corresponds to $t=0$ in a local coordinate.  
		For $g^*$, we have the estimate (see \cite[Table~3.2]{ouyang-25}, \cite[Definition~3.19]{sun-zhang-24})
		$
		g^* = O\big( |t|^{-4/3} \, |dt|^2 \big).
		$
		On the other hand, by Lemma~\ref{lem-estimate-mu(a,b)}, we have 
		$
		g \ge C |dt|^2
		$
		for some constant $C>0$.  
		Therefore, the ratio satisfies
		\[
		|f| = \sqrt{\frac{g^*}{g}} = O(|t|^{-2/3}).
		\]
		Hence $f$ is bounded near $p$, and extends to a holomorphic function on $\mathbb{P}^1$. Therefore, it must be constant, and the proposition follows.
	\end{proof}
	
	\subsection{Proof of Theorem \ref{thm-geometric-realization-GIT}} \label{subsection-proof-of-thm-GIT}
	
	By \cite[Cor.~4.8]{schwarz-89}, the map $\pi_W: X^{ss} \longrightarrow \overline{\mathcal{M}_W}$
	is topologically a quotient map in the analytic topology. Hence, it suffices to prove that 
	$\widetilde{\Phi}_W : X^{ss} \longrightarrow \overline{\mathfrak{M}}$
	is continuous. Let $x_k \to x_\infty \in X^{ss}$ with $x_k \in \pi^{-1}(\mathcal{M}_{\mathrm{Jac}})$. We shall prove that $\widetilde{\Phi}_W(x_k) \to \widetilde{\Phi}_W(x_\infty)$ in $\overline{\mathfrak{M}}$.
	
	We may write $x_k = [h_{8,k} : h_{12,k}]$ and $x_\infty = [h_{8,\infty} : h_{12,\infty}]$, with $(h_{8,k}, h_{12,k}) \to (h_{8,\infty}, h_{12,\infty})$.
	Let $g_k$ be the generalized KE metric on $\mathbb{P}^1$ associated with $(h_{8,k}, h_{12,k})$ as in \eqref{equation-fomular-g}. We divide the proof into five cases.
	
	\subsubsection{Case 1: $x_\infty \in X^s$ with $\Delta \not\equiv 0$.}
	
	These $x_\infty$ correspond to smooth elliptic K3 surfaces. Denote by $g_\infty$ the associated metrics on $\mathbb{P}^1$ given by $(h_{8,\infty},h_{12,\infty})$. Let $S$ be the singular locus of $\mathbb{P}^1$ defined by $\Delta_\infty=0$. Then $g_k \to g_\infty$ smoothly away from $S$. 
	
	\begin{lem}\label{lem-case1-volume-estimate}
		Given $p \in S$ and $\delta > 0$, there exists a disk $B(p)$ containing $p$ such that
		$\mathrm{vol}_{g_k}\bigl(B(p)\bigr) < \delta$
		for all $k$.
	\end{lem}
	
	\begin{proof}
		Take a local coordinate chart $B_1$ such that $p=0$ is the only singularity in $B_1$.
		Let $m,n,l$ denote the vanishing orders of $h_8, h_{12}, 1/j$ at $p$.  
		One may write
		\[
		h_{8,k}(t) = a_k(t) \prod_{i=1}^m (t-\alpha_{i,k}), 
		\quad 
		h_{12,k}(t) = b_k(t) \prod_{i=1}^n (t-\beta_{i,k}),
		\quad
		j_k(t) = \frac{1}{c_k(t)\prod_{i=1}^l (t-\gamma_{i,k})},
		\]
		where $\alpha_{i,k}, \beta_{i,k}, \gamma_{i,k}$ converge to $0$, and
		$|a_k|, |b_k|, |c_k|$ are uniformly bounded below on $B_1$ as $k \to \infty$.  
		
		In the following, we assume that $0 \le m \le 3$ (the case $n \le 5$ is analogous). 
		By Lemma~\ref{lem-estimate-mu(a,b)}, we have
		\[
		\mathrm{vol}_{g_k}\bigl(B_r\bigr)
		\;\le\; 
		\int_{B_r} \mu(h_{8,k}(t), h_{12,k}(t))
		\;\le\; 
		\int_{B_r} \frac{C}{|h_{8,k}(t)|^{1/2}} \,\log^{+}|j_k(t)|.
		\]		
		Fix $\lambda_0 = \tfrac{1}{2}$. Then for some constant $C$ (which may change line by line), we have
		\begin{align*}
			\frac{1}{|h_{8,k}(t)|^{1/2}} \, \log^{+}|j_k(t)|
			&\;\le\;
			C \prod_{i=1}^m |t-\alpha_{i,k}|^{-1/2} 
			\prod_{i=1}^l |t-\gamma_{i,k}|^{-\lambda_0^2/l} \\[2mm]
			&\;\le\;
			C\left(
			\frac{m}{m+\lambda_0} \prod_{i=1}^m |t-\alpha_{i,k}|^{-(m+\lambda_0)/(2m)}
			+
			\frac{\lambda_0}{m+\lambda_0} \prod_{i=1}^l |t-\gamma_{i,k}|^{-\lambda_0(m+\lambda_0)/l}
			\right) \\[1mm]
			&\;\le\;
			C\left(
			\sum_{i=1}^m |t-\alpha_{i,k}|^{-(m+\lambda_0)/2}
			+
			\sum_{i=1}^l |t-\gamma_{i,k}|^{-\lambda_0(m+\lambda_0)}
			\right),
		\end{align*}
		which is integrable in $B_1$.
		Thus, for sufficiently small $r$, the estimate above implies
		$
		\mathrm{vol}_{g_k}\bigl(B_r\bigr) < \delta.
		$
	\end{proof}
	
	It is easy to see that $p_{g_k}(B_r(p)) < \delta$ if $B_r$ is small.  
	Thus, by Proposition~\ref{prop-volume-comparison}, for any $\varepsilon > 0$, we can choose $\delta$ and $Q \coloneqq \mathbb{P}^1 \setminus \bigsqcup_{p \in S} B_r(p)$ as above, such that as $k \to \infty$,
	\begin{align*}
		&d_{GH}\bigl((\mathbb{P}^1,d_{g_k}),(\mathbb{P}^1,d_{g_\infty})\bigr)\\
		\le\;& d_{GH}\bigl((\mathbb{P}^1,d_{g_k}),(Q,d_{Q,{g_\infty}})\bigr) 
		+ d_{GH}\bigl((Q,d_{Q,{g_\infty}}),(Q,d_{Q,g_k})\bigr)
		+ d_{GH}\bigl((Q,d_{Q,g_k}),(\mathbb{P}^1,d_{g_\infty})\bigr) \\
		<\;& \varepsilon.
	\end{align*}
	Thus we have $\widetilde{\Phi}_W(x_k) \to \widetilde{\Phi}_W(x_\infty)$ in this case.
	
	\subsubsection{Case 2: $x_\infty \in X^s$ with $\Delta \equiv 0$.}
	
	After applying a $\mathrm{PSL}_2(\mathbb{C})$ action, we may assume
	\[
	h_{8,k} = 3 G_4^2 + u_k, 
	\qquad 
	h_{12,k} = G_4^3 + v_k,
	\]
	where $G_4 = t(t-1)(t-2)(t-c)$ with $c \neq 0,1,2,\infty$ (the case $c = \infty$ is equivalent to $c = \frac{2}{3}$ or $\frac{4}{3}$), and $u_k, v_k$ are polynomials converging to $0$.  	
	
	It is equivalent to prove that for any $x_k \to x_\infty$, there exists a subsequence $x_{k_i}$ such that $\widetilde{\Phi}_W(x_{k_i}) \to \widetilde{\Phi}_W(x_\infty)$. Thus, by passing to a subsequence, we may write the discriminant in the form
	\[
	\Delta_k = 4 h_{8,k}^3 - 27 h_{12,k}^2 
	= \varepsilon_k 
	\prod_{i=1}^l (t - \gamma_{i,k}) 
	\prod_{i=l+1}^{24} \left( \frac{t}{\gamma_{i,k}} - 1 \right),
	\]
	with $\varepsilon_k \to 0$, $\gamma_{i,k} \to \gamma_{i,\infty} \in \mathbb{P}^1$, and
	$
	\begin{cases}
		|\gamma_{i,\infty}| \le 1, & \text{if } i \le l, \\
		|\gamma_{i,\infty}| > 1, & \text{if } i > l.
	\end{cases}
	$
	
	Define $S = \{0,1,2,c,\gamma_{i,\infty}\}$. For any $M^0 \Subset \mathbb{P}^1 \setminus S$, there exists a constant $C$ such that on $M^0$ we have $\frac{1}{C} < \frac{\Delta_k}{\varepsilon_k} < C$ as $k \to \infty$. Thus, by \eqref{equation-fomular-g},	
	\[
	g_k = \mu(h_{8,k},h_{12,k}) |dt|^2
	= \frac{\lambda(j)}{|3h_{8,k}|^{1/2} + |h_{12,k}|^{1/3}} \log^+ \Big| \frac{1728 h_{8,k}^3}{\Delta_k} \Big|
	= (1+o(1)) \frac{2\pi^2}{3|G_4|} \log \frac{1}{\varepsilon_k} |dt|^2.
	\]
	
	Consider the rescaled metric $g_k' = \frac{1}{\log(1/\varepsilon_k)}\,g_k$. We then obtain $g_k' \to \frac{2\pi^2}{3|G_4|} |dt|^2$ on $M^0$ as $k \to \infty$.  
	
	\begin{lem}\label{lem-case2-volumeestimate}
		For every $p \in S$ and every $\delta > 0$, there exists a disk $B(p)$ containing $p$ such that $\mathrm{Vol}_{g_k'}(B(p)) < \delta$
		for all $k$.
	\end{lem}
	
	\begin{proof}
		Assume $1/\varepsilon_k > e$. Then
		\[
		\log^+\Big(\frac{1728 h_{8,k}^3}{\Delta_k}\Big)
		\le \log(1/\varepsilon_k) + \log^+(j_k \varepsilon_k)
		\le \log(1/\varepsilon_k) \, \log^+(j_k \varepsilon_k).
		\]
		This implies
		\[
		\mathrm{Vol}_{g_k'}(B(p))
		= \int_{B(p)} \frac{\mu(h_{8,k},h_{12,k})}{\log(1/\varepsilon_k)}
		\le \int_{B(p)} \frac{1}{|h_{8,k}|^{1/2}} \, \log^+(j_k \varepsilon_k).
		\]
		Now $j_k \varepsilon_k = \frac{1728 h_{8,k}^3 \varepsilon_k}{\Delta_k} \to j_\infty$ for some meromorphic function on $\mathbb{P}^1$, and the desired estimate then follows in the same way as in Lemma~\ref{lem-case1-volume-estimate}.
	\end{proof}
	
	Thus, $\widetilde{\Phi}_W(x_k) \to \widetilde{\Phi}_W(x_\infty)$ follows by using Proposition~\ref{prop-volume-comparison} and Lemma~\ref{lem-case2-volumeestimate} in the same way as in Case~(1).
	
	\begin{rmk}\label{rmk-j-to-infty}
		In this case, we see that for any $\varepsilon > 0$, we can take $M^0$ such that $j(t) \to \infty$ on $M^0$, and moreover for $k \gg 1$,
		\[
		\frac{\mathrm{Vol}_{g_k}(M^0)}{\mathrm{Vol}_{g_k}(M)} > 1 - \varepsilon.
		\]
	\end{rmk}
	\subsubsection{Case 3: $x_\infty=[a t^4: b t^6]\in X^{ps}\setminus X^s$ with $[a:b]\neq [3:1]$.}
	
	We prove that in this case the sequence $\widetilde{\Phi}_W(x_{k})$ is volume collapsing in the sense of \cite{cheeger-colding-97}. Hence, the GH limit of $\widetilde{\Phi}_W(x_{k})$ is $I^1$ by \cite[Thm.~1.1]{honda-sun-zhang-19}.
	
	We may write 
	\[
	(h_{8,k},h_{12,k}) = (a t^4 + u_k,\, b t^6 + v_k),
	\]
	where $u_k, v_k$ are polynomials converging to $0$, 	
	and the discriminant has the form
	\[
	\Delta_k 
	= h_{8,k}^3 - 27 h_{12,k}^2 
	= c_k \prod_{i=1}^{12} (t - \gamma_{i,k}) 
	\prod_{i=13}^{24} \left(\frac{t}{\gamma_{i,k}}-1\right),
	\]
	where $c_k \to a^3 - 27 b^2$, and	
	$
	\begin{cases}
		\gamma_{i,k} \to 0, & \text{if } i \le 12, \\
		\gamma_{i,k} \to \infty, & \text{if } i \ge 13.
	\end{cases}
	$
	
	Fix $\delta > 0$ small, and define $M^0 = \{\, \delta < |t| < 1/\delta \,\}$. Then we have $j(t) \to \frac{a^3}{a^3 - 27b^2}$ on $M^0$.  
	Thus, from \eqref{equation-fomular-g}, there exists a constant $C$ such that
	\[
	g_k \to g_\infty \coloneq \frac{C}{t^2} |dt|^2 \quad \text{on } M^0.
	\]
	The metric space $(M^0, g_\infty)$ is isometric to a rescaling of the standard cylinder $S^1 \times [\log\delta, -\log \delta]$ with flat metric.
	
	Let $r = -\log \delta - 1$, which can be arbitrarily large, and let $d = \operatorname{diam}(\mathbb{P}^1, g_k)$. Then for $k$ sufficiently large, we can choose $p \in M^0$ such that $\mathrm{vol}_{g_k}(B_r(p)) \le 5\pi r$. By Proposition~\ref{prop-volume-comparison-gen-KE}, 
	\[
	\mathrm{vol}_{g_k}(\mathbb{P}^1) = \mathrm{vol}_{g_k}(B_d(p)) \le \frac{d^2}{r^2} \mathrm{vol}_{g_k}(B_r(p)) \le \frac{5 \pi d^2}{r}.
	\]
	Thus, in the metric $g_k' = g_k/d^2$, the diameter rescales to 1, and $\mathrm{vol}_{g_k'}(\mathbb{P}^1) \le 5\pi/r$ can be made arbitrarily small. Hence $\widetilde{\Phi}_W(x_{k})$ is volume collapsing.
	
	\subsubsection{Case 4: $x_\infty = [3t^4: t^6] \in X^{ps}\setminus X^s$.}
	
	In this case, we may assume
	\[
	(h_{8,k},h_{12,k}) = (3t^4 + u_k,\, t^6 + v_k), \qquad u_k, v_k \to 0,
	\]
	and the discriminant is
	\[
	\Delta_k 
	= 4 h_{8,k}^3 - 27 h_{12,k}^2 
	= \varepsilon_k \prod_{i=1}^l (t - \gamma_{i,k}) 
	\prod_{i=l+1}^{24} \left(\frac{t}{\gamma_{i,k}}-1\right),
	\]
	with $\varepsilon_k \to 0$, $\gamma_{i,k} \to \gamma_{i,\infty} \in \mathbb{P}^1$, and 
	$
	\begin{cases}
		|\gamma_{i,\infty}| \le 1, & \text{if } i \le l, \\
		|\gamma_{i,\infty}| > 1, & \text{if } i > l.
	\end{cases}
	$
	
	Let $m = \tfrac{1}{2} \min\{ |\gamma_{i,\infty}| \neq 0 \}$, and choose $\delta < m$ very small.  
	Define 
	\[
	M^0 = \{\, \delta < |t| < m \,\}.
	\]
	Then $\varepsilon_k j_k(t)$ converges uniformly on $M^0$.  
	By \eqref{equation-fomular-g}, as $k \to \infty$, the restriction of $g_k$ to $M^0$ has the form
	\[
	g_k = (1 + o(1)) \, \bigl(\log(1/\varepsilon_k)\bigr) \frac{C_0}{|t|^2} |dt|^2.
	\]
	Hence,
	\[
	\Bigl(M^0,\, \frac{1}{\log(1/\varepsilon_k)} g_k \Bigr) \longrightarrow S^1 \times [\log \delta, \log m]
	\]
	with the standard flat metric. Thus, as in Case~(3), the sequence $(\mathbb{P}^1, g_k)$ is volume collapsing after rescaling to unit diameter, and we obtain convergence of $\widetilde{\Phi}_W(x_{k})$ to $I^1$.
	
	\subsubsection{Case 5: $x_\infty \in X^{ss} \setminus X^{ps}$.}
	
	By properties of the GIT quotient, there exists $y \in X^{ps} \setminus X^s$ such that $y \in \overline{SL_2(\mathbb{C}) \cdot x_\infty}$. Thus, we may take $g_k \in SL_2(\mathbb{C})$ such that $g_k x_k \to y$. By Cases~(3) and (4), this implies $\widetilde{\Phi}_W(x_{k}) \longrightarrow \widetilde{\Phi}_W(y) = I^1$.
	
	\medskip
	Combining the five cases above, we have proved Theorem~\ref{thm-geometric-realization-GIT}.

	\bibliography{K3-compactification}

\end{document}